\documentclass{article}

\usepackage{amssymb,amsthm} 
\usepackage{MnSymbol}
\usepackage[mathscr]{euscript}
\usepackage{tikz-cd}
\usepackage{hyperref}
\hypersetup{
    colorlinks=true,
    linkcolor=[HTML]{0D47A1},		
    citecolor=[HTML]{0D47A1},		
    urlcolor=[HTML]{0D47A1}		    
}
\usepackage{fullpage}
\usepackage{rotating}
\usepackage[capitalize,nameinlink]{cleveref}

\usepackage{enumitem}
\usepackage{bbm}

\usepackage[shortcuts]{extdash}

\usepackage{mleftright} 
\mleftright

\usepackage{tabularx}
\usepackage{multirow}
\usepackage{float}

\definecolor{mat}{HTML}{ffd6ad}

\definecolor{andre}{HTML}{a0ceff}
\definecolor{eric}{HTML}{ffadad}
\definecolor{georg}{HTML}{000000}

\newtheorem{thm}{Theorem}[subsection]
\newtheorem{theorem}[thm]{Theorem}
\newtheorem{lemma}[thm]{Lemma}
\newtheorem{proposition}[thm]{Proposition}
\newtheorem{corollary}[thm]{Corollary}

\theoremstyle{definition}

\newtheorem{definition}[thm]{Definition}

\newtheorem{examples}[thm]{Examples}
\newtheorem{remark}[thm]{Remark}

\newlist{exmpenum}{enumerate}{1}
\setlist[exmpenum]{label=\alph*), ref=\theproposition~(\alph*)}
\crefalias{exmpenumi}{example} 
\newlist{remenum}{enumerate}{1}
\setlist[remenum]{label=\alph*), ref=\theproposition~(\alph*)}
\crefalias{remenumi}{remark} 
\newlist{thmenum}{enumerate}{1}
\setlist[thmenum]{label=\arabic*., ref=\theproposition~(\arabic*)}
\crefalias{thmenumi}{theorem} 
\newlist{propenum}{enumerate}{1}
\setlist[propenum]{label=\arabic*., ref=\theproposition~(\arabic*)}
\crefalias{propenumi}{proposition} 
\newlist{corenum}{enumerate}{1}
\setlist[corenum]{label=\arabic*., ref=\theproposition~(\arabic*)}
\crefalias{corenumi}{corollary} 
\newlist{lemenum}{enumerate}{1}
\setlist[lemenum]{label=\arabic*., ref=\theproposition~(\arabic*)}
\crefalias{lemenumi}{lemma} 
\newlist{defenum}{enumerate}{1}
\setlist[defenum]{label=\roman*), ref=\theproposition~(\roman*)}
\crefalias{defenumi}{definition}

\newtheorem{thmintro}{}
\newenvironment{thm-intro}[1]
  {\renewcommand\thethmintro{#1}\thmintro\itshape}
  {\endthmintro}
\newenvironment{thm-introref}[2]
  {\renewcommand\thethmintro{#1}\thmintro(#2)\itshape}
  {\endthmintro}
\newtheorem{defnintro}{}
\newenvironment{defn-intro}[1]
  {\renewcommand\thedefnintro{#1}\defnintro\normalfont}
  {\enddefnintro}

\newcommand{\tto}{{\begin{tikzcd}[ampersand replacement=\&]{}\ar[r]\&{}\end{tikzcd}}}
\newcommand{\mto}{{\begin{tikzcd}[ampersand replacement=\&]{}\ar[r,mapsto]\&{}\end{tikzcd}}}
\newcommand{\mot}{{\begin{tikzcd}[ampersand replacement=\&]{}\&{}\ar[l,mapsto]\end{tikzcd}}}
\newcommand{\stto}{{\begin{tikzcd}[ampersand replacement=\&, sep=small]{}\ar[r]\&{}\end{tikzcd}}}
\newcommand{\subtto}{{\begin{tikzcd}[ampersand replacement=\&, sep=small]{}\ar[r, hook]\&{}\end{tikzcd}}}
\newcommand{\ntto}[1]{{\begin{tikzcd}[ampersand replacement=\&]{}\ar[r,"{#1}"]\&{}\end{tikzcd}}}
\newcommand{\nntto}[1]{{\begin{tikzcd}[ampersand replacement=\&]{}\ar[rr,"{#1}"]\&\&{}\end{tikzcd}}}

\newcommand{\adjunction}{{\begin{tikzcd}[ampersand replacement=\&]{}\ar[r, shift left=1]\&\ar[l,shift left=1]{}\end{tikzcd}}}

\newcommand\Ra{\Rightarrow}

\renewcommand\iff{\Leftrightarrow}

\newcommand\slice[1]{_{/#1}}

\newcommand\pbmark{\ar[dr, phantom, "\ulcorner" very near start, shift right=1ex]}
\newcommand\pomark{\ar[ul, phantom, "\lrcorner" very near start, shift right=1ex]}
\newcommand\pbmarkk{\ar[drr, phantom, "\ulcorner" very near start, shift right=1ex]}
\newcommand\pomarkk{\ar[ull, phantom, "\lrcorner" very near start, shift right=1ex]}

\newcommand\fperp{\upModels}

\newcommand\lorth[1]{{}^\perp{#1}}
\newcommand\rorth[1]{{#1}^\perp}
\newcommand\lforth[1]{{}^\fperp{#1}}
\newcommand\rforth[1]{{#1}^\fperp}
\newcommand\lrforth[1]{{}^\fperp{\left({#1}^\fperp\right)}}

\DeclareMathOperator*{\colim}{colim}
\newcommand\Map[2]{\mathrm{Map}{\left(#1,#2\right)}}
\newcommand\PSh[1]{\mathrm{PSh}{\left(#1\right)}}
\newcommand\Sh[1]{\mathrm{Sh}{\left(#1\right)}}
\newcommand\Op[1]{\mathrm{Op}{\left(#1\right)}}

\newcommand\ie{i.e.\ }
\newcommand\eg{e.g.\ }

\newcommand\cA{\mathscr{A}}
\newcommand\cB{\mathscr{B}}
\newcommand\cC{\mathscr{C}}
\newcommand\cD{\mathscr{D}}
\newcommand\cE{\mathscr{E}}
\newcommand\cF{\mathscr{F}}
\newcommand\cG{\mathscr{G}}

\newcommand\cI{\mathscr{I}}
\newcommand\cJ{\mathscr{J}}
\newcommand\cK{\mathscr{K}}
\newcommand\cL{\mathscr{L}}
\newcommand\cM{\mathscr{M}}
\newcommand\cN{\mathscr{N}}

\newcommand\cP{\mathscr{P}}

\newcommand\cR{\mathscr{R}}
\newcommand\cS{\mathscr{S}}
\newcommand\cT{\mathscr{T}}

\newcommand\cW{\mathscr{W}}

\renewcommand\AA{\mathbb{A}}

\newcommand\NN{\mathbb{N}}

\newcommand\RR{\mathbb{R}}

\newcommand\ZZ{\mathbb{Z}}

\newcommand{\tA}{\mathsf A}

\newcommand{\tU}{\mathsf U}
\newcommand{\tV}{\mathsf V}

\newcommand{\tX}{\mathsf X}
\newcommand{\tY}{\mathsf Y}
\newcommand{\tZ}{\mathsf Z}
\newcommand{\tone}{\mathsf 1}

\renewcommand\powerset[1]{\mathcal{P}{\left(#1\right)}}

\newcommand\xto\xrightarrow
\newcommand\xot\xleftarrow
\newcommand\ot\leftarrow

\newcommand\Tau{\mathrm{T}}

\newcommand\Fin{\mathrm{\cF in}}
\newcommand\Finp{\mathrm{\cF in^\bullet}}
\newcommand\Cat{\mathsf{Cat}}
\newcommand\CAT{\mathsf{CAT}}
\newcommand\op{^\mathrm{op}}
\newcommand\sfop{^\mathsf{op}}
\newcommand\ho[1]{\mathsf{ho}(#1)}
\newcommand\fun[2]{\left[#1,#2\right]}

\newcommand\Topos{\mathsf{Topos}}
\newcommand\Logos{\mathsf{Logos}}

\newcommand\lex{^\mathsf{lex}}

\newcommand\cclex{^\mathsf{lex}_\mathsf{cc}}
\newcommand\oo{$\infty$\=/}
\newcommand\ooo{$(\infty,1)$\=/}
\newcommand\oooo{$(\infty,2)$\=/}
\newcommand\pointed{^{\bullet}}

\newcommand\Arr[1]{{#1}^{\to}}
\newcommand\arr{^\to}

\newcommand\sat{^\mathsf{s}}
\newcommand\ssat{^\mathsf{ss}}
\newcommand\loc{^\mathsf{loc}}
\newcommand\ac{^\mathsf{a}}
\renewcommand\cong{^\mathsf{c}}

\newcommand\gtop{^\mathsf{G}}

\newcommand\diag{^\Delta}
\newcommand\bc{^\mathsf{bc}}
\newcommand\mono{^\mathsf{mono}}

\newcommand\epi{^\mathsf{epi}}
\newcommand\quot{^\mathsf{quot}}
\newcommand\nil{^\mathsf{nil}}
\newcommand\wnil{^\mathsf{wnil}}

\newcommand\reduced{^\mathsf{red}}
\newcommand\hred{^\mathsf{hred}}
\newcommand\cons{^\mathsf{cons}}
\newcommand\wcons{^\mathsf{wcons}}

\newcommand\jred{^\mathsf{jred}}
\newcommand\hcrad[1]{\!\sqrt[{\sf Cov}]{#1}}    
\newcommand\acrad[1]{\!\sqrt[{\sf Ac}]{#1}}  
\newcommand\hrad[1]{\!\sqrt[{\sf Hyp}]{#1}}    
\newcommand\Jac[1]{\!\sqrt[{\sf Jac}]{#1}}    

\newcommand\nrad[1]{\sqrt[{\sf Nil}]{#1}}  
\newcommand\nred{\mathsf{red}}    
\newcommand\ampli{^\mathsf{ample}}
\newcommand\idem{^\mathsf{idem}}

\newcommand\Acmap[1]{\mathsf{AcMap}{\left(#1\right)}}
\newcommand\Hypoab[1]{\mathsf{HypoAb}{\left(#1\right)}}

\newcommand\idl[1]{\langle#1\rangle}

\newcommand{\pp}{\square}
\newcommand{\join}{\star}

\newcommand\aprod{\pp\ac}
\newcommand\cprod{\pp\cong}

\newcommand{\pbh}[2]{\left\langle #1, #2 \right\rangle}                
\newcommand{\magic}[2]{\left\{ #1, #2 \right\}}                

\newcommand{\exc}[1]{^{(#1)}}                

\newcommand\Iso{\mathsf{Iso}}
\newcommand\All{\mathsf{All}}
\newcommand\Mono{\mathsf{Mono}}
\newcommand\Surj{\mathsf{Surj}}
\newcommand\Trunc[1]{\mathsf{Trunc}_{#1}}
\newcommand\Conn[1]{\mathsf{Conn}_{#1}}
\renewcommand\Im{\mathrm{Im}}
\newcommand\im[1]{\mathsf{im}{\left(#1\right)}}
\newcommand\coim[1]{\mathsf{coim}{\left(#1\right)}} 

\newcommand\truncated[1]{^{\leq #1}}
\newcommand\connected[1]{_{> #1}}

\newcommand\Acyclic[1]{\mathsf{Acy}{\left(#1\right)}}

\newcommand\Cong[1]{\mathsf{Cong}{\left(#1\right)}}

\newcommand\Congsg[1]{\mathsf{Cong_{sg}}{\left(#1\right)}}

\newcommand\Class[1]{\mathsf{ClassMaps}{\left(#1\right)}}

\newcommand\MAcyclic[1]{\mathsf{MAcy}{\left(#1\right)}}
\newcommand\EAcyclic[1]{\mathsf{EAcy}{\left(#1\right)}}
\newcommand\MCong[1]{\mathsf{MCong}{\left(#1\right)}} 
\newcommand\RCong[1]{\mathsf{HRCong}{\left(#1\right)}} 
\newcommand\GTop{\mathsf{GTop}}

\newcommand\CTop{\mathsf{CTop}}

\newcommand\quotient{/\!/} 

\newcommand\susp[1]{\mathsf{S}{\left(#1\right)}}
\newcommand\suspn[2]{\mathsf{S}^{#1}{\left(#2\right)}}
\newcommand\suspinfty[1]{\mathsf{S}^\infty{\left(#1\right)}}
\newcommand\dec[1]{\mathsf{D}{\left(#1\right)}}
\newcommand\decn[2]{\mathsf{D}^{#1}{\left(#2\right)}}
\newcommand\decinfty[1]{\mathsf{D}^\infty{\left(#1\right)}}
\newcommand\decname{\mathsf{D}}
\newcommand\decinftyname{\mathsf{D}^\infty}

\renewcommand\S[1]{\mathscr S\left[#1\right]}

\newcommand\Sp{\mathrm{\cS p}}
\newcommand\PSp{\mathrm{\cP\cS p}}

\newcommand\splx{^{\Delta\op}}

\renewcommand\div{\backslash}

\renewcommand\upslice[1]{^{\geq #1}}
\renewcommand\downslice[1]{_{\leq #1}}

\newcommand\Ideal[1]{\mathsf{Ideal}{\left(#1\right)}}
\newcommand\RIdeal[1]{\mathsf{RIdeal}{\left(#1\right)}}

\newcommand\Spec[1]{\mathsf{Spec}{\left(#1\right)}}

\newcommand\TLattice{\mathsf{\otimes\textsf{-}Lattice}}

\newcommand\TFrame{\mathsf{\otimes\textsf{-}Frame}}

\newcommand\Frame{\mathsf{Frame}}

\newcounter{intpara}

\title{
Left-exact Localizations of $\infty$\=/Topoi III:\\
The Acyclic Product
}

\author{
Mathieu Anel%
\footnote{Carnegie Mellon University,
	\href{mailto:mathieu.anel@protonmail.com}{mathieu.anel@protonmail.com}} ,
Georg Biedermann%
\footnote{SFB 1085 Higher Invariants, Universität Regensburg,
	\href{mailto:gbm@posteo.de}{gbm@posteo.de}} ,
Eric Finster%
\footnote{University of Birmingham,
	\href{mailto:e.l.finster@bham.ac.uk}{e.l.finster@bham.ac.uk}} ,
and Andr\'{e} Joyal%
\footnote{CIRGET, UQ\`AM, 
	\href{mailto:joyal.andre@uqam.ca}{joyal.andre@uqam.ca}} 
}
\date{}

\begin{document}

\maketitle

\begin{abstract}
We define a commutative monoid structure on the poset of left-exact localizations of a higher topos, that we call the acyclic product.
Our approach is anchored in a structural analogy between the poset of left-exact localizations of a topos and the poset of ideals of a commutative ring.
The acyclic product is analogous to the product of ideals.
The sequence of powers of a given left-exact localization defines a tower of localizations.
We show how this recovers the towers of Goodwillie calculus in the unstable homotopical setting.
We use this to describe the topoi of $n$\=/excisive functors as classifying $n$\=/nilpotent objects.
\end{abstract}

\setcounter{tocdepth}{2}
\tableofcontents

\section{Introduction}

The purpose of this work is to develop a conceptual framework for Goodwillie's calculus of functors in the unstable homotopical setting.
We continue the point of view introduced in \cite{ABFJ:GC} that Goodwillie's calculus defines towers of left-exact localizations of \oo topoi, and we formalize in a successful way the intuition that these towers are akin to the completion tower of a ring at an ideal.
This work is the continuation of \cite{ABFJ:GC,ABFJ:HS,ABFJ:GT} but it should somehow be read first, since it presents the broad context subsuming our previous works.

\medskip
The precise statements of our results will demand the introduction of a number of notions.
We shall thus start with a simplified version.
Our main result is the following.

\begin{thm-intro}{Theorem A}
\label{thmA}
The poset of left-exact localizations of an \oo topos $\cE$ has a symmetric monoidal closed structure.
\end{thm-intro}
The product will be derived from the pushout product of maps.
We shall postpone its description for now and only say what is needed to state our main application.
The unit of the monoidal structure is the maximal localization (inverting all maps).
The powers of a given left-exact localization $P_0:\cE\to \cE_0$ with respect to this monoidal structure define a tower of left-exact localizations
\[
\cE\stto \dots\stto \cE_2\stto \cE_1\stto \cE_0\,.
\]
The simplest example of this phenomenon is instructive.
Let $\cS$ be the \oo category of spaces (\ie \oo groupoids) and $\Fin\subset\cS$ the subcategory of finite spaces.
We consider the \oo topos $\cE=\fun \Fin \cS$ of functors $\Fin\to \cS$.
The calculus of functors of Goodwillie \cite{G03,ABFJ:GC} also defines a tower of left-exact localizations $P_n:\fun \Fin \cS \to \fun \Fin \cS\exc n$ where $\fun \Fin \cS\exc n\subseteq \fun \Fin \cS$ is the subcategory of $n$\=/excisive functors (precise definitions will be recalled later).
In particular, the $0$\=/excisive functors are the constant functors and the localization $P_0:\fun \Fin \cS \to \cS$ coincides with the evaluation at $1\in \Fin$.
Our main application is the following.

\begin{thm-intro}{Theorem B}
\label{thmB}
Goodwillie's left-exact localization $P_n$ is the $(n+1)$\=/st power of  ${P_0:\fun \Fin \cS\to \cS}$.
\end{thm-intro}

This connection with the calculus of homotopy functors shows that our product of localizations is not a mere abstraction but recovers known constructions in homotopy theory.
This theorem will lead us to describe the topoi of $n$\=/excisive functors as classifying nilpotent objects, a concept which we shall explain in due course.

\medskip
The point that we would like to make now is that such a tower of localizations should be considered in the light of the completion tower of a (commutative) ring at an ideal.
The structural analogy between the theories of topoi and rings is folkloric but underexploited.
One goal of this work is to deepen and extend it.
The rest of this introduction will present the duality between topoi and logoi introduced in  \cite{Anel-Joyal:topo-logie}.
We will then develop the comparison with the duality of schemes and rings until it becomes clear that Theorems~A and~B are analogues of the classical results:
\begin{enumerate}[label=\Alph*., parsep=0em]
\item the poset of ideals of a ring has a symmetric monoidal closed structure given by the product of ideals;
\item the quotient of the ring $\RR[x]$ by the $(n+1)$\=/power of the ideal $(x)$ is the ring $\RR[x]/(x)^{n+1}$ of polynomials of degree $n$.
\end{enumerate}
We will see on the way that it is possible to define for topoi analogues of $I$\=/adic towers, $I$\=/adic completion and separation, Jacobson radical, nilpotent elements, nilradical, and more.
We will use these developments as an incentive to renew a lot of terminology of topos theory in a way we find more suggestive (\eg \cref{table:analogies3}).

\bigskip
From now on, we shall drop all `\oo' prefixes and simply talk about categories and topoi to mean their higher versions.

\subsection{Topoi, logoi, congruences}

\paragraph{Topo-logy}

The theory of topoi is an instance of a geometric--algebraic duality, akin to that between affine schemes and rings, or locales and frames \cite{Anel-Joyal:topo-logie}.
To fully take advantage of this, it will be convenient to distinguish the two sides by using different names for the objects.
Failing to do so leads to an utterly confusing state of affairs: imagine a version of algebraic geometry, where a ring $A$ is called a scheme, where the word `ring' is never used, and where a morphism of schemes is denoted $A\to B$ to mean a ring morphism $B\to A$.
This is the current state of topos theory.
To clarify the situation, we shall introduce the terminology of a {\it logos} to refer to the algebraic counterpart of a {\it topos}, that is to refer to the category of sheaves on a topos.
This will make clear which statements of the theory are better thought of algebraically (in terms of sheaves) or geometrically (in terms of points).

Formally, a {\it logos} is defined as an accessible left-exact localization of a presheaf category, or more intrinsically as a presentable category with descent (\cref{def:logos}).
A morphism of logoi is simply a functor preserving colimits and finite limits (cocontinuous and left-exact functor).
This structure is akin to that of a ring with colimits and finite limits playing the role of addition and multiplication, and descent that of distributivity \cite{Anel-Joyal:topo-logie}.
The category of topoi is then defined as the opposite of that of logoi $\Topos:=\Logos\sfop$
(we shall always consider these categories 
without their non-invertible 2\=/cells).
We shall talk about morphisms of topoi and morphisms of logoi rather than geometric and algebraic morphisms of topoi, which we did in \cite{ABFJ:GT}.

The forgetful functor $\Logos \to \CAT$ (into large categories) has a left adjoint defined on small categories, called  the {\it free logos} functor $\S - : \Cat \to \Logos$ (\cref{thm:free-logos}).
If $C$ is a small category, $\S C = \PSh {C\lex}$ where $C\lex$ is the free completion of $C$ for finite limits.
This is similar to the construction of the free ring on a set, by first constructing the monoid of monomials and then considering sums of monomials.
The free logos on no generator is the initial logos $\S\emptyset=\cS$, the category of spaces.
The free logos on one generator $X$ is the category of functors $\S X = \fun\Fin \cS$ from the category of finite spaces $\Fin$ to $\cS$, since $\Fin$ is the finite colimit completion of a singleton.
The functor associated to the generator is the canonical inclusion $X:\Fin \to \cS$ (represented by $1\in \Fin$).
A logos morphism $\S X \to \cE$ is equivalent to an object of $\cE$, and $\S X$ is also known as the {\it object classifier}.
The dual topos will be denoted $\tA$ and called the {\it topos line}.
A {\it sheaf} on a topos $\tX$ is defined as a morphism of topoi $\tX\to \tA$.
We denote by $\Sh \tX$ the category of sheaves on a topos $\tX$.
The universal property of $\S X$ identifies $\Sh \tX$ with the logos dual to $\tX$.
The situation is similar to algebraic geometry, where the functions $S\to \AA^1$ on an affine scheme $S$ reconstruct the ring $A$ dual to $S$.

\begin{table}[H]
\caption{First analogies between rings and logoi}
\label{table:analogies1}

\begin{center}	
\renewcommand{\arraystretch}{1.5}
\begin{tabularx}{.95\textwidth}{
|>{\centering\arraybackslash}X
|>{\centering\arraybackslash}X|
}
\hline
Affine scheme (over $\RR$)  &  Topos \\
\hline
$\RR$\=/algebra &  Logos\\
\hline
Sum, products and distributivity & Colimits, finite limits and descent\\
\hline
Ground field $\RR$ & Logos $\cS$ of spaces\\
\hline
Free ring $\RR[x]$ / Affine line $\AA^1$ & Free logos $\S X = \fun \Fin \cS$ / Topos line $\tA$ \\
\hline
Free ring $\RR[x_1,\dots,x_n]$ / Affine space $\AA^n$ & Free logos $\S C = \PSh {C\lex}$ / Affine topos $\tA^C$ \\
\hline
Rational function = $S\to \AA^1$ & Sheaf = $\tX\to \tA$\\
ring of functions $O(S) = \mathsf{Hom_{Scheme}}(S,\AA^1)$ & logos of sheaves  $\Sh\tX = \mathsf{Hom}_\Topos(\tX,\tA)$\\
\hline
\end{tabularx}
\end{center}	
\end{table}

\paragraph{Quotients and congruences}
Accessible left-exact localizations of logoi will be called more suggestively {\it quotients} of logoi.
The dual morphism of topoi will be called an {\it inclusion}.
As with other kinds of algebraic structures, such quotients are in bijection with a notion of {\it congruence} of logoi (\cref{def:congruence}), developed in our previous work \cite{ABFJ:HS,ABFJ:GT}.
Congruences are an axiomatization of the classes of maps inverted by a morphism of logoi.
Recall that a congruence on a ring $A$ is a subring $R\subseteq A\times A$ satisfying reflexivity, transitivity and symmetry.
A congruence on a logos $\cE$ is a sublogos $\cK \subseteq \Arr\cE$ (full subcategory of the arrow category of $\cE$ closed under colimits and finite limits) containing all isomorphisms (reflexivity) and closed under composition (transitivity).%
\footnote{In the context of inverting arrows, the symmetry condition does not make sense anymore.}
It can be deduced from these axioms that congruences satisfy the 3-for-2 condition.
Given a logos morphism $\phi:\cE\to \cF$, the class $\cK_\phi = \phi^{-1}(\Iso)$ of maps inverted by $\phi$ is always a congruence, called the {\it congruence of $\phi$}.
Every class of maps $\Sigma$ in $\cE$ is contained in a smallest congruence $\Sigma\cong$.
We shall say that $\Sigma\cong$ is generated by $\Sigma$.
A congruence $\cK$ is of {\it small generation} if $\cK = \Sigma\cong$ for some {\it set} of maps $\Sigma$.
It is not known if every congruence is of small generation.

Congruences correspond exactly to the `strongly saturated classes of maps closed under base change' introduced in \cite{Lurie:HTT}.
There, Lurie proves that quotients of logoi are in bijection with congruences of small generation (see \cite[Section 4.2.11]{ABFJ:HS}).
The quotient of a logos $\cE$ by a congruence $\cK$ of small generation will be denoted $\cE\quotient \cK$ instead of the more classical notations $\cK^{-1}\cE$ or $\cE[\cK^{-1}]$ (the double bar is meant to suggest a quotient and to prevent a confusion with the slice).
An important example of a congruence is the class $\Conn\infty$ of \oo connected maps, whose associated quotient $\cE\quotient\Conn\infty$ is the hypercompletion of $\cE$ (\cref{exmp:congruence4}).

Every logos can be presented as a quotient of a free logos, $\cE=\S C\quotient \cK$.
This means that a morphism of logos $\cE\to \cF$ is equivalent to a diagram $C\to \cF$ satisfying some equations written in terms of small colimits and finite limits.
We shall say that $\cE$ is the logos {\it classifying} such diagrams.
For example, the condition for an object $X$ to be $n$\=/truncated is such a condition, since it says that the higher diagonal $\Delta^{n+2}X$ is invertible. 
The classifying logos is $\S X\quotient\{\Delta^{n+2}X\}\cong$, which is the logos classifying $n$\=/truncated objects (see \cref{sec:congruence} or~\cite{ABFJ:HS} for more examples).

We shall denote by $\Cong\cE$ the poset of all congruences ordered by inclusion, and we shall mention explicitly when small generation assumptions are needed (we shall, however, neglect this issue in this introduction).
Classically, congruences on a ring $A$ are in bijection with ideals of $A$, and we shall compare the poset $\Cong\cE$ to the poset $\Ideal A$.
The inclusion of logos congruences is opposed to the inclusion of the corresponding subtopoi (as for ideals and subschemes).
The minimal element of $\Cong\cE$ is the congruence $\Iso$ of all isomorphisms, and the maximal element is the congruence $\All$ of all maps.
The initial logos $\cS$ has only these two trivial congruences (\ie $\Cong\cS = \{\Iso, \All\}$).
This compares it to a {\it ground field}
(hence our choice of $\RR$\=/algebras in \cref{table:analogies1}).%
\footnote{Let us add, for the reader who is familiar with these notions, that the duality of topoi and logoi---particularly in the light of our approach of Goodwillie calculus---compares  better with that of $C^\infty$\=/schemes and $C^\infty$\=/rings \cite{Moerdijk-Reyes}, than with that of schemes and rings over $\ZZ$ (see \eg \cref{rem:triple-facto-tower}).
}

\begin{table}[H]
\caption{Congruences and quotients}
\label{table:analogies2}

\begin{center}	
\renewcommand{\arraystretch}{1.6}
\begin{tabularx}{.9\textwidth}{
|>{\centering\arraybackslash}X
|>{\centering\arraybackslash}X|
}
\hline
\it Commutative rings & \it Logoi\\
\hline
congruence $R\subseteq A\times A$ ($\Leftrightarrow$ ideal~$I\subseteq A$) & congruence $\cK\subseteq \Arr\cE$\\
\hline
ideal $(S)$ generated by a subset $S$ & congruence $\Sigma\cong$ generated by a set of maps $\Sigma$\\
\hline
quotients $A\to A/I$ & quotient $\cE\to \cE\quotient \cK$ (left-exact~localizations)\\ 
\hline
kernel of a morphism $\phi:A\to B$ & congruence $\cK_\phi:=\phi^{-1}(\Iso)$ \mbox{of a logos morphism $\phi:\cE\to \cF$}\\
\hline
image factorisation $\phi:A\to A/\ker \phi\to B$ & image factorization $\phi:\cE\to \cE\quotient {\cK_\phi}\to \cF$\\
\hline
\end{tabularx}
\end{center}	
\end{table}

\paragraph{Monogenic and epic congruences}
Several notions of congruences (\ie left-exact localizations) are classically distinguished since \cite{Lurie:HTT}.
A congruence is {\it topological} if it can be generated by a class of monomorphisms, and it is {\it cotopological} if it is contained in the congruence $\Conn\infty$ of \oo connected maps.
Here we found more meaningful to change the terminology.
We shall call a congruence in a logos $\cE$ {\it monogenic} if it can be generated by monomorphisms, and {\it epic} if it is contained in $\Conn\infty$.
The latter name is motivated by the fact that for a congruence $\cK$ one has $\cK\subseteq\Conn\infty \iff \cK\subseteq \Surj$, where $\Surj$ is the class of surjections in $\cE$ (aka effective epimorphisms).

When $\cE=\PSh C$ is a presheaf logos, monogenic congruences of $\cE$ are classically in bijection with Grothendieck topologies on $C$.
It is actually possible to extend the notion of Grothendieck topology to arbitrary logoi by axiomatizing the classes of monomorphisms inverted by a logos morphism (\cref{def:GT}).
These {\it extended Grothendieck topologies} were introduced in \cite{ABFJ:GT} where we showed that they are naturally in bijection with monogenic congruences.

Any congruence $\cK$ has a monogenic part $\cK\mono:=(\cK\cap\Mono)\cong\subseteq \cK$, which is the largest monogenic congruence inside $\cK$,
and an epic part $\cK\epi$, which is the image of $\cK$ by the quotient $\cE\to \cE\quotient\cK\mono$.
In a manner of speaking, any congruence $\cK$ is an ``extension'' of its epic part $\cK\epi$ by its monogenic part $\cK\mono$. 
If $\phi:\cE\to \cF$ is a morphism of logoi, with associated congruence $\cK_\phi = \phi^{-1}(\Iso)$, we can refine the image factorization from~\cref{table:analogies2} into
\[
\cE\stto \cE\quotient\cK\mono \stto (\cE\quotient \cK\mono)\quotient\cK\epi = \cE\quotient \cK \stto \cF\ ,
\]
where the first part is a {\it monogenic quotient}, the second an {\it epic quotient}, and the third one a conservative morphism of logoi (\ie reflecting isomorphisms). We call this (unique) factorization the {\it quotient triple factorization}.

For example, if we consider the congruence $\idl X = \{X\to 1\}\cong$ in the logos $\S X= \fun\Fin\cS$, we shall see that the monogenic quotient is the logos $\S {X\connected \infty}:=\S X\quotient \idl X \mono$ classifying \oo connected objects (\cref{exmp:mono-congruence:oo-conn-obj}).
The epic part $\idl X\epi$ is the congruence of $\Conn\infty$ of \oo connected maps in $\S {X\connected \infty}$ and the above sequence becomes
$\S X \to \S {X\connected\infty} \to \S {X\connected\infty}\quotient\Conn\infty = \S X \quotient \idl X = \cS$.

\begin{table}[H]
\caption{Terminology for quotients and subtopoi}
\label{table:analogies3}

\begin{center}	
\renewcommand{\arraystretch}{1.6}
\begin{tabularx}{.95\textwidth}{
|>{\centering\arraybackslash}X
|>{\centering\arraybackslash}X
|>{\centering\arraybackslash}X|
}
\hline
{\it Classical terminology}  & \multicolumn{2}{c|}{\it New terminology }\\
\hline
Topos & Logos & Topos\\
\hline
left-exact localization, subtopos & quotient & inclusion, subtopos\\
\hline
strongly saturated class closed under base change & congruence & ---\\
\hline
topological localization & monogenic congruence, monogenic quotient & ample inclusion, amplification, ample~subtopos\\
\hline
cotopological localization & epic congruence, epic quotient & tight inclusion, tight subtopos\\
\hline
\end{tabularx}
\end{center}	
\end{table}

We shall use the following terminology for the geometric dual of the above factorization.
The inclusions of topoi dual to monogenic and epic quotients will be called {\it ample} and {\it tight}, respectively.
A morphism of topoi dual to a conservative morphism of logoi will be called a {\it surjection} of topoi.
Geometrically, the previous factorization says that any morphism of topoi $\tY\to \tX$ factors into
\[
\tY\stto \tZ \stto \tZ\ampli \stto \tX\,,
\]
that is, into a surjection, 
followed by a tight inclusion, 
followed by an ample inclusion.
We shall say that $\tZ\ampli$ is the {\it amplification} of $\tZ$.
This terminology will be justified below.
(An analogue of this triple factorization in algebraic geometry is proposed in 
\cref{rem:loc=mono,rem:triple-facto}, see also \ref{intro:hyper-reduction} in this introduction.)

\subsection{The acyclic product}
For a commutative ring $A$ the monoidal structure of $\Ideal A$ is the product of ideals. In a similar way we shall define a product on $\Cong\cE$.
This product will be derived from the pushout product $\pp:\Arr\cE \times \Arr\cE\to \Arr\cE$ of maps in $\cE$, but we shall need a few steps.

Before we consider congruences, let us start with arbitrary classes of maps. 
If $\Sigma$ and $\Tau$ are two classes of maps in $\cE$, we define their product by $\Sigma\pp\Tau := \{f\pp g\,|\, f\in \Sigma,\, g\in \Tau\}$.
It is easy to see that this defines a symmetric monoidal closed structure on classes of maps in $\cE$ (\cref{prop:mon-all}).
If $\cK$ and $\cK'$ are now two congruences, the product $\cK\pp \cK'$ is not in general a congruence but can be completed into one.
The naive idea is then to define the product of congruences by $(\cK\pp\cK')\cong$.
This formula will end up being correct, but somehow only ``by accident''.
Also this formula is actually complicated to work with in practice because the congruence completion will end up {\it not} being a monoidal functor (\cref{rem:not-monoidal}).
Instead, we will construct our product in another way: by defining first a monoidal structure on {\it acyclic classes}, and then proving that it restricts to congruences.

\paragraph{The algebra of acyclic classes}
An acyclic class is a saturated class of maps $\cA\subseteq\Arr\cE$ which is closed under base change.
It can also be defined as a class of maps containing all isomorphisms, closed under composition, base change, and under colimits in the category $\Arr\cE$ (\cref{def:acyclic}).
Acyclic classes have been introduced in \cite{ABFJ:HS} and were further developed in \cite{ABFJ:GT}.
We continue their study here in \cref{sec:suspension}.
Fundamental examples of acyclic classes are the class $\Surj$ of surjections (aka effective epimorphisms), the classes $\Conn n$ of $n$\=/connected maps, as well as the class of acyclic maps in the sense of homotopy theory (from where we took the name)
\cite{Raptis:acyclic,Hoyois:acyclic}.
We denote by $\Acyclic\cE$ the poset of acyclic classes.
Most importantly, any congruence is an acyclic class and the inclusion $\Cong\cE\subseteq \Acyclic\cE$ is particularly nice since it has both a left and a right adjoint (\cref{thm:adjCongAcyclic}).
This will allow us to manipulate congruences with the greater flexibility enjoyed by acyclic classes.

Any class of maps $\Sigma$ is contained in a smallest acyclic class $\Sigma\ac$, and an acyclic class $\cA$ is said to be of small generation if $\cA=\Sigma\ac$ for a {\it set} of maps $\Sigma$.
Acyclic classes of small generation are the left classes of {\it modalities}, which can be defined as (unique) factorization systems stable by base change.
A modality on a logos defines a truncation operation and can be thought of as an analogue of a t-structure on a stable category (\cref{rem:t-structure}).
Additionally, when the class $\Sigma$ consists of a single map $W \to 1$ for some space $W$, the factorization system associated to the acyclic class $\Sigma\ac$ coincides with the \emph{fiberwise nullification} studied in the works of Bousfield \cite{bousfield1994localization}, Farjoun \cite{farjoun1995cellular} and Chach\'{o}lski \cite{Chach:inequ}.  Where our results coincide with or generalize results already found in this body of work, we will provide appropriate references.

We define the {\it acyclic product} of two acyclic classes $\cA$ and $\cA'$ as the acyclic class $\cA\aprod \cA' := (\cA\pp\cA')\ac$.
To state the following result, let us define a {\it $\otimes$\=/frame} as a symmetric monoidal suplattice whose unit is the top element (see \cref{sec:tensor-frames}).
An example of $\otimes$\=/frame is the poset $\Ideal A$ for the product of ideals.

\begin{thm-intro}{\cref{thm:aprod}}
The poset $\Acyclic\cE$ of acyclic classes is a $\otimes$\=/frame for the acyclic product $\aprod$.
\end{thm-intro}

\noindent The unit of this product is the class $\All$ of all maps, and the class $\Iso$ of isomorphism is an absorbing element. 
The acyclic product is relatively easy to compute in practice, because the acyclic completion functor $\Sigma\mapsto \Sigma\ac$ is monoidal, \ie $\Sigma\ac \aprod \Tau\ac = (\Sigma\pp\Tau)\ac$ (this will actually be false for the congruence completion, see \cref{rem:not-monoidal}). 
An easy application of this gives the nice formula $\Surj \aprod \Conn n = \Conn {n+1}$ (\cref{ex:aprod:n-conn}).
More generally, we shall define the {\it suspension} of an acyclic class $\cA$ as the product $\susp\cA:=\Surj\aprod\cA$.
The suspension operation will be an important technical device in some proofs (see \cref{sec:suspension}).
The acyclic product $\cA\aprod -$ has a right adjoint $\cA\div -$, where $\cA\div\cB = \{u\in \Arr\cE\,|\, u\pp \cA\subseteq \cB\}$, called the {\it acyclic division}.
The acyclic division $\Surj\div \Iso$ is the class of acyclic maps in the sense of homotopy theory mentioned above (\cref{ex:aprod:acyclic-maps}).

\paragraph{The algebra of congruences}
Our main result is to prove that congruences are closed under the acyclic product.
The following statement is the more precise version of \cref{thmA}.

\begin{thm-intro}{\cref{thm:cprod}}
The acyclic product $\cK\aprod \cK'$ of two congruences is a congruence and $\Cong\cE\subseteq \Acyclic\cE$ is sub-$\otimes$\=/frame.
\end{thm-intro}
\noindent 
Our definition for the product of congruences will then be $\cK\aprod \cK' = (\cK\pp\cK')\ac$.
This formula happens to coincide with the previous ``naive'' one $(\cK\pp\cK')\cong$ 
(one inclusion is always true, the other one is given by the theorem) 
but it is in fact more convenient for computations.
Indeed, for a congruence $\cK$, it is often easy to find a class $\Sigma\subseteq \cK$ such that $\cK= \Sigma\ac$.
We call such a class $\Sigma$ a {\it lex generator} for $\cK$ (\cref{def:lex-generator}).
For example, our main result in \cite{ABFJ:HS} says that, for $\cK=\Sigma\cong$, we have also $\cK=(\Sigma\diag)\ac$ (where $\Sigma\diag$ is the completion for diagonals).
Then for two congruences $\cJ=\Sigma\ac$ and $\cK=\Tau\ac$ with lex generators, we get that $\cJ\aprod\cK = (\Sigma\pp\Tau)\ac$.
In other words, the product of lex generators is a lex generator (\cref{boxprodlexgen}).
This will come handy in our study of Goodwillie's calculus (\eg to prove \cref{thm:Goodwillie-main}).

\medskip
The product of monogenic congruences is an important special case.
As it happens, any class $\Sigma$ of monomorphisms is a lex generator: the congruence generated by $\Sigma$ coincides with the acyclic class generated by $\Sigma$ ($\Sigma\cong=\Sigma\ac$).
This makes the acyclic product of monogenic congruences particularly easy to compute.
Moreover, we shall see that the product of monogenic congruences is idempotent!
Recall that a frame is a suplattice where finite meets distribute over suprema.
Any frame is a $\otimes$\=/frame, and the converse is true as soon as the product is idempotent.
The following result is a simplified version of \cref{thm:gtop} and the following remarks.

\begin{thm-intro}{Theorem C}
\label{thmC}
The poset of monogenic congruences $\MCong\cE$ is a sub-$\otimes$\=/frame of $\Cong\cE$ which is a frame. 
\end{thm-intro}
\noindent Interpreted in terms of Grothendieck topologies, this result recovers the classical fact that they have a frame structure (see \cite[Theorem 3.0.1]{ABFJ:GT}).
Moreover, in these terms, the product of Grothendieck topologies is simply their intersection as classes of maps (\cref{rem:formula-gtop}).
The presentation in terms of the acyclic product provides explicit formulas to compute generators for this intersection (see \cref{rem:formula-gtop,ex:aprod-gtop}).

\medskip
We shall see in \cref{prop:naturality} that all the $\otimes$\=/frames structures of \cref{thm:aprod,thm:cprod,thmC} are natural in the logos $\cE$.

\paragraph{Sketch of the proofs}
Let us give a brief technical paragraph about the proof of \cref{thm:cprod}.
The proof of \cref{thm:aprod} is formal and relies on the fact that the functor $\Sigma\mapsto \Sigma\ac$ will be monoidal.
We will deduce it from a general result on quotients of $\otimes$\=/frames (\cref{localizationquantale}).
The same technique cannot be used for \cref{thm:cprod} because the functor $\Sigma\mapsto \Sigma\cong$ is not be monoidal (\cref{rem:not-monoidal}).
It will then rely on a number of technical lemmas about acyclic classes.
The main tool will be the operation of {\it decalage} of an acyclic class $\cA$, which is the class $\dec \cA := \{f\in \cA\,|\,\Delta(f) \in \cA\}$ (where $\Delta(f)$ is the diagonal of the map $f$, see \cref{def:decalage}).
We introduced the decalage in \cite{ABFJ:HS} (albeit without this name) where we showed that it is an acyclic class (Proposition 3.3.5 and Theorem 3.3.9 together, {\it loc. cit.}).

The decalage is closely related to the suspension defined above.
If $\cA$ is an acyclic class contained in $\Surj$, then $\dec\cA = \susp\cA = \Surj \aprod \cA$ (\eg $\dec{\Conn n} = \Conn{n+1}$), but the two operations are different in general.
Their intricate relationship is the purpose of \cref{sec:suspension}.
The decalage operation is useful because congruences are exactly the fixed points of the operator $\decname$ (\cref{lem:decalage}).
The proof of \cref{thm:cprod} will then follow from the remark that $\decname$ is a lax monoidal endofunctor on $\Acyclic\cE$ (\cref{boxofsatellites}).
This last result relies on an important decomposition formula for $\dec\cA$ (\cref{thm:explicit-decalage}) which is the endgame of \cref{sec:suspension}.

\subsection{Completion towers}

\paragraph{The structure of ideals}
To explain how we are going to use the product on $\Cong\cE$, 
we recall first some classical constructions within the $\otimes$\=/frame $\Ideal A$.
Let $X$ be the affine scheme dual to the ring $A$.
We fix an ideal $I$ and we let $Z\subseteq X$ be the associated subscheme.
One can associate to $I$ a string of other ideals, 
a tower of associated quotients of $A$,
and a filtration of corresponding geometric objects:
\[
\begin{tikzcd}[sep=small]
0 \ar[r,hook]&
\bigcap_n I^n \ar[rr,hook]&{}
&\dots \ar[r,hook]&
I^3 \ar[r,hook]&
I^2 \ar[r,hook]&
I \ar[r,hook]&
\sqrt I \ar[r,hook]&
\Jac I
\\
A \ar[r,two heads]&
A/\bigcap_n I^n \ar[r]&
A^\wedge_I \ar[r,two heads]&
\dots \ar[r,two heads]&
A/I^3 \ar[r,two heads]&
A/I^2 \ar[r,two heads]&
A/I \ar[r,two heads]&
A/\sqrt I \ar[r,two heads]&
A/\Jac I
\\
X \ar[from=r,hook']&
Z^\flat \ar[from=r]&
Z^\wedge \ar[from=r,hook']&
\dots \ar[from=r,hook']&
Z\exc 2 \ar[from=r,hook']&
Z\exc 1 \ar[from=r,hook']&
Z \ar[from=r,hook']&
Z\reduced \ar[from=r,hook']&
Z\jred
\,.
\end{tikzcd}
\]
Let us explain all these objects:
\begin{enumerate}[label=--]

\item The ideals $I^n$ are the powers of $I$.
The geometric object $Z\exc n$ dual to $A/I^{n+1}$ is the $n$\=/th order {\it infinitesimal neighborhood} of $Z$ in $X$.

\item The sequence of quotients $\dots \to A/I^2\to A/I$ defines the {\it completion tower} of $A$ at $I$.
The ring ${A^\wedge_I=\lim_n A/I^n}$ is the {\it completion} of $A$ with respect to the $I$\=/adic topology.
Its dual geometric object $Z^\wedge$ is an ind-scheme called the {\it formal neighborhood} of $Z$ in $X$.

\item The ring morphism $A\to A^\wedge_I$ sends an element of $A$ to its {\it Taylor tower}, and its kernel is the ideal $\bigcap_n I^n$ of {\it $I$\=/flat functions}.
The ring $A/\bigcap_n I^n$ is the {\it separation} of $A$ with respect to the $I$\=/adic topology.
We shall call the dual geometric object $Z^\flat$ the {\it flat neighborhood} of $Z$ in $X$.

\item The ideal $\sqrt I$ is the {\it nilradical} of $I$, spanned by elements of $A$ that become nilpotent in $A/I$.
The quotient $A/\sqrt I$ is the {\it nilreduction} of $A/I$.

\item The ideal $\Jac I$ is the {\it Jacobson radical} of $I$, spanned by all the elements of $A$ that vanish on the closed points of $Z$.
The quotient $A/\!\Jac I$ is the {\it Jacobson reduction} of $A/I$.

\end{enumerate}
We will define the analogues of all these notions in topos theory.

\paragraph{Completion towers}
The finite powers $\cK^n$ of a congruence $\cK$ are defined with the acyclic product.
Because the unit of $\Cong\cE$ is the class $\All$ of all maps, the $\cK^n$ form a decreasing sequence, associated with a tower of quotients, and a filtration of subtopoi
\[
\begin{tikzcd}[row sep=5]
&\dots \ar[r,hook]&
\cK^3 \ar[r,hook]&
\cK^2 \ar[r,hook]&
\cK\\
\cE^\wedge_\cK:=\lim_n\cE\quotient\cK^n\ar[r]&
\dots \ar[r,two heads]&
\cE\quotient\cK^3 \ar[r,two heads]&
\cE\quotient\cK^2 \ar[r,two heads]&
\cE\quotient\cK\\
\tZ^\wedge \ar[from=r,hook']&\dots \ar[from=r,hook']&
\tZ\exc 2 \ar[from=r,hook']&
\tZ\exc 1 \ar[from=r,hook']&
\tZ\,.
\end{tikzcd}
\]
We shall call the sequence of powers the {\it $\cK$\=/adic filtration} of $\cE$, 
and the tower of quotients the {\it completion tower} of $\cK$.
The limit of the tower $\cE^\wedge_\cK$ is called the {\it completion} of $\cE$ with respect to the $\cK$\=/adic filtration.
It can be consider either as a logos, or as a pro-object in logoi.
The dual topos is called the {\it formal neighborhood} of $\tZ$ and denoted $\tZ^\wedge$.
The topos dual to $\cE\quotient\cK^{n+1}$ is called the {\it $n$\=/th infinitesimal  neighborhood} of $\tZ$ and denoted $\tZ\exc n$.

\paragraph{Separation}
The congruence of the logos morphism $\widehat\phi:\cE\to \cE^\wedge_\cK$ is $\bigcap_n\cK^n$.
The quotient $\cE\quotient \bigcap_n\cK^n$ is called the {\it separation of $\cE$} for the $\cK$\=/adic filtration.
We will see that the monogenic parts of the powers $\cK^n$ are all equal to $\cK\mono$ (\cref{prop:idem-mono-cong}).
And since the poset morphism $(-)\mono$ is a right adjoint, we have also $(\bigcap_n\cK^n)\mono = \cK\mono$.
This proves that the quotient triple factorization of $\widehat\phi$ is 
\[
\cE\stto \cE\quotient\cK\mono \stto \cE\quotient \bigcap_n\cK^n \stto \lim_n\cE\quotient\cK^n\,.
\]
This provides the top part of the tower.
The geometric dual of $\cE\quotient \bigcap_n\cK^n$ will be denoted $\tZ^\flat$ and called the {\it flat neighborhood} of $\tZ$.

\paragraph{Hyper-reduction}
\label[paragraph]{intro:hyper-reduction}
A famous feature of higher topoi is the failure of Whitehead theorem: there exist logoi with non-trivial \oo connected maps.
We denote the class of \oo connected maps in a logos $\cE$ by $\Conn\infty$. It is a congruence, whose associated quotient $\cE\quotient\Conn\infty$ is usually called {\it hypercompletion} of $\cE$.
More generally, for any congruence $\cK$, the inverse image of $\cK':=\phi^{-1}(\Conn\infty)$ by the quotient morphism $\phi:\cE\to \cE\quotient\cK$ defines a congruence $\cK'\supseteq \cK$, whose quotient $\cE\quotient\cK'$ is the hypercompletion of $\cE\quotient\cK$.
However, this quotient $\cE\quotient\cK'$ does not sit at all at the completion stage of the tower associated to $\cK$. Quite on the contrary it sits below the $0$\=/th level: $\cE\to\lim_n \cE\quotient\cK^n\to\hdots\to\cE\quotient\cK \to\cE\quotient\cK'$. 

In the analogy with rings, we found it illuminating to compare the congruence $\Conn\infty$ to the Jacobson radical.
This is suggested by the following fact: a sheaf on a topological space is \oo connected if and only if all its stalks are contractible spaces.
In other words, if sheaves are thought of as functions with values in spaces, \oo connected sheaves are functions ``vanishing'' at every point.
The Jacobson radical captures a similar idea in algebraic geometry.
In the same line of ideas, the notion of epic quotient (\ie cotopological localization) compares well to quotient by ideal within the Jacobson radical (see \cref{rem:triple-facto}).
Yet another argument is that we will develop a theory of nilpotent maps in a logos, and the resulting nilradical will sit within $\Conn\infty$ (see \ref{intro-sec:nilradical} of this introduction).

This analogy suggests to look at the hypercompletion $\cE\quotient \Conn\infty$ as a reduction, and we shall rather call it the {\it hyper-reduction}.
For a congruence $\cK$, we shall call the congruence $\cK'$ defined above the {\it hyper-radical} of $\cK$ and denote it $\hrad \cK$.
The quotient $\cE\quotient\hrad\cK$ is the hyper-reduction of $\cE\quotient\cK$.
The congruence $\Conn\infty=\hrad\Iso$ will be called the {\it hyper-radical} of $\cE$.
Geometrically, if $\tZ\subseteq \tX$ is the subtopos corresponding to $\cK$, we shall denote $\tZ\hred\subseteq \tZ$ the subtopos corresponding to $\hrad\cK$ and call it the {\it hyper-reduction} of $\tZ$.

Recall that $\MCong\cE\subseteq \Cong\cE$ denoted the subposet of monogenic congruences. We showed in \cite{ABFJ:GT} that this inclusion has a right adjoint given by the monogenic part $\cK\mapsto \cK\mono$, and a further right adjoint essentially given by the hyper-radical $\cK\mapsto \hrad\cK$ defined above (\cref{thm:mono-cong}).
Moreover, the image of the second right adjoint is precisely the subposet of hyper-radical congruences.
In particular, monogenic and hyper-radical congruences are in bijection.%

Finally, every subtopos sits between its hyper-reduction and its amplification:
\[
\begin{tikzcd}[row sep=small]
\tZ\hred \ar[r,hook]& \tZ \ar[r,hook] & \tZ\ampli
\\
\cE\quotient\hrad\cK & \cE\quotient\cK  \ar[l, two heads] & \cE\quotient\cK\mono \ar[l, two heads]\,.
\end{tikzcd}
\]
We shall see that $\tZ\hred$ and $\tZ\ampli$ are the two extremities of a sequence of subtopoi canonically associated to $\tZ$.

\paragraph{Nilradical}
\label{intro-sec:nilradical}
We are left with the definition of the nilradical of a congruence. This will take a few steps.
We will say that a congruence $\cK$ is {\it $n$\=/nilpotent} if $\cK^{n+1}=\Iso$ holds in $\Cong\cE$ (\cref{def:n-nilpotent}).
A map $f$ is $n$\=/nilpotent if the congruence $\{f\}\cong$ generated by $f$ is $n$\=/nilpotent,
and an object $X$ is $n$\=/nilpotent if the map $X\to 1$ is $n$\=/nilpotent.
Explicitly, a map $f$ is nilpotent if all the $(n+1)$\=/fold pushout products 
$\Delta^{k_0} f \pp \dots \pp \Delta^{k_n} f$
of iterated diagonals of $f$ are isomorphisms.
More generally, given a congruence $\cK$, we will say that a map $f$ is {\it $\cK$\=/nilpotent of order $n$} if $(\{f\}\cong)^{n+1} \subseteq \cK$.
A congruence $\cK$ is {\it nilradical} if, for every $n$, every $\cK$\=/nilpotent map of order $n$ is in $\cK$.
The {\it nilradical} of a congruence $\cK$, denoted $\nrad\cK$, is defined as the smallest nilradical congruence containing $\cK$.
The {\it nilradical} of $\cE$ is defined as $\nrad \Iso$.

We shall see that nilpotent maps are necessarily \oo connected, 
and that we have an inclusion $\nrad \Iso\subseteq \hrad \Iso = \Conn\infty$ (\cref{cor:nilrad-surj}).
This inclusion compares to that of the nilradical into the Jacobson radical for rings.
More generally, the nilradical of a congruence is always contained in its hyper-radical $\nrad \cK\subseteq \hrad \cK$.
If $\tZ$ and $\tZ\nil$ are the topoi dual to $\cE\quotient\cK$ and $\cE\quotient\nrad\cK$, we shall say that $\tZ\nil$ is the {\it nilreduction} of $\tZ$.
We have inclusions $\tZ\hred\subseteq \tZ\nil\subseteq\tZ$.

\bigskip
The following diagram summarizes all the objects associated to a congruence/subtopos (we shall give them names in \cref{table:tower}).%
\footnote{
Notice that the logos version of the completion tower has one more stage $\cE\quotient\cK\mono$ at the top than the completion tower of a ring.
An analogue in commutative algebra is proposed in \cref{rem:triple-facto}.
}

\[
\begin{tikzcd}[sep=small]
0 \ar[r,hook]&
\cK\mono \ar[r,hook]&
\bigcap_n \cK^n \ar[rr,hook]&{}
&\dots \ar[r,hook]&
\cK^3 \ar[r,hook]&
\cK^2 \ar[r,hook]&
\cK \ar[r,hook]&
\nrad\cK \ar[r,hook]&
\hrad\cK
\\
\cE \ar[r,two heads]&
\cE\quotient\cK\mono \ar[r,two heads]&
\cE\quotient\bigcap_n \cK^n \ar[r]&
\cE^\wedge_\cK \ar[r,two heads]&
\dots \ar[r,two heads]&
\cE\quotient\cK^3 \ar[r,two heads]&
\cE\quotient\cK^2 \ar[r,two heads]&
\cE\quotient\cK \ar[r,two heads]&
\cE\quotient\nrad\cK \ar[r,two heads]&
\cE\quotient\hrad\cK
\\
\tX \ar[from=r,hook']&
\tZ\ampli \ar[from=r,hook']&
\tZ^\flat \ar[from=r]&
\tZ^\wedge \ar[from=r,hook']&
\dots \ar[from=r,hook']&
\tZ\exc 2 \ar[from=r,hook']&
\tZ\exc 1 \ar[from=r,hook']&
\tZ \ar[from=r,hook']&
\tZ\nil \ar[from=r,hook']&
\tZ\hred
\end{tikzcd}
\]

\subsection{Goodwillie calculus}

Our motivating example of a completion tower is given by Goodwillie's calculus.
Although we shall not discuss it here, Weiss' orthogonal calculus \cite{Weiss:OC} provides another example, see \cite{ABFJ:TM}.

\paragraph{Goodwillie Calculus}
We consider the free logos on one generator $\S X = \fun\Fin\cS$.
It is also the category of finitary functors $\cS\to \cS$. 
Goodwillie's calculus of functors distinguishes the subcategories 
\[
\cS=\fun\Fin\cS\exc 0
\ \subseteq\ 
\fun\Fin\cS\exc 1
\ \subseteq\ 
\dots
\ \subseteq\ 
\fun\Fin\cS
\]
spanned by $n$\=/excisive functors and proves that they are reflective $P_n:\fun\Fin\cS\to \fun\Fin\cS\exc n$.
The explicit formula given for $P_n$ shows that they are left-exact functors.
The $P_n$ thus define a tower of quotients of the logos $\S X = \fun\Fin\cS$ that we call the {\it Goodwillie tower}.
The Goodwillie tower of a functor $F:\Fin\to \cS$ 
\[
F\stto \dots\stto P_1F \stto P_0F=F(1)\,.
\]
is the tower of its $n$\=/excisive approximations.
On the other hand, we can consider the congruence $\idl X:= \{X\to 1\}\cong$.
The corresponding quotient $\S X\quotient \idl X$ is $\cS$ and the quotient map is sending a functor $F:\Fin\to \cS$ to its value $F(1)$ (where 1 is the terminal object of $\Fin$).
The following result is the more precise version of \cref{thmB}.
It confirms very much the intuition that the Goodwillie tower is indeed a version of the Taylor series from classical calculus.

\begin{thm-intro}{\cref{thm:Goodwillie:tower}}
\begin{enumerate}

\item
The Goodwillie tower of $\S X = \fun\Fin\cS$ coincides with the completion tower of the congruence $\idl X$.

\item
The monogenic part is the logos $\S{X\connected\infty} = \S X \quotient \idl X \mono$ classifying \oo connected objects.

\item
The $n$\=/th stage of the Goodwillie tower is the logos $\S {X\exc n} := \S X \quotient \idl X ^{n+1} = \fun\Fin\cS\exc n$ classifying $n$\=/nilpotent objects.

\item
The hyper-radical of all the stages but the limit one is generated by the image of $X$.
The hyper-reduction of all the stages is the canonical morphism to $\cS$ induced by $P_0$.

\end{enumerate}
\end{thm-intro}

Let us say a word on the other stages of the tower.
The objects of the $\idl X$\=/adic completion $\S X^\wedge_{\idl X}$ are formal Goodwillie towers (towers of excisive functors).
As a pro-object in logoi, it classifies nilpotent objects of finite order (\cref{rem:pro-lim}).
We do not know what it classifies as a logos.
The canonical functor ${\S X \to \S X^\wedge_{\idl X}}$ sends a functor to its Goodwillie tower, and the 
separation quotient ${\S X \to \S X\quotient \bigcap_n \idl X ^n}$ identifies two functors if they have the same Goodwillie tower.
This logos classifies \oo connected objects with an unidentified extra property.
Finally, the logos $\cS = \S X\quotient\idl X$ being hyper-reduced, the nil-reduction and hyper-reduction stages of the tower are simply $\cS$.

Geometrically, the tower of \cref{thm:Goodwillie:tower} studies the infinitesimal neighborhood of the terminal object $1\in \tA$ in the topos line.
We have the following filtration of subtopoi of the topos $\tA$ (their classifying properties are given below)
\[
\overbrace{
\underbrace{
\tone
=\tA\hred
=\tA\reduced
=\tA\exc 0
}_{\text{contractible objects}}
}^{\text{hyper-reduced subtopos}}
\subtto
\overbrace{
\underbrace{
\tA\exc 1
\subtto
\tA\exc 2
\subtto
\dots
\subtto
\tA^\wedge
}_{\text{nilpotent objects}}
}^{\text{infinitesimal and formal neighborhoods}}
\stto
\!\!\!\!
\underbrace{
\overbrace{\tA^\flat}
}_{\text{unknown}}^{\text{flat nbd}}
\!\!\!\!
\subtto
\!\!\!\!\!\!\!\!
\underbrace{\overbrace{\tA\connected\infty}
}_{\text{\oo conn. objects}}^{\text{amplification}}
\!\!\!\!\!\!\!\!\!\!
\subtto
\!\!\!\!\!\!\!
\underbrace{\overbrace{\tA}
}_{\text{all objects}}^{\text{ambient topos}}
\!\!\,.
\]

\paragraph{Pointed Goodwillie calculus}
\Cref{thm:Goodwillie:tower} has a pointed analogue.
We consider the logos $\S {X\pointed} := \fun \Finp \cS = \fun\Fin\cS\slice X$ (where $\Finp$ is the category of pointed finite spaces).
A logos morphism $\S {X\pointed} \to \cE$ is the same thing a pointed object $1\to E$ in $\cE$, and the universal pointed object $X\pointed\in \S{X\pointed}$ is the functor $\Finp\to \cS$ forgetting the base point.
We consider the congruence $\idl{X\pointed} := \{X\pointed\to 1\}\cong$.
Goodwillie's definition of $n$\=/excision still make sense for functors $\Finp\to \cS$ and the pointed analogue of \cref{thm:Goodwillie:tower} holds.

In this pointed context, a theorem of Goodwillie \cite{G03} identifies the category $\fun \Finp \cS\exc 1$ of 1-excisive functors $\Finp\to \cS$ to the category $\PSp$ of parametrized spectra (\ie bundles of spectra over variable spaces).
In this correspondence, spectra $\Sp\subseteq \PSp$ correspond to 1-excisive functors $F:\Finp\to \cS$ that also {\it reduced} (\ie such that $F(1)=1$).
We then deduce the following classifying property of $\PSp$.

\begin{thm-intro}{\cref{prop:parametrized-spectra}}
The logos of parametrized spectra $\PSp = \S{X\pointed}\quotient\idl{X\pointed}^2$ classifies pointed 1\=/nilpotent objects.
\end{thm-intro}

\noindent By definition, an object $E$ is 1\=/nilpotent if the maps $\Delta^m E \pp \Delta^n E$ are invertible for every $m,n\in \NN$.
When $E$ is pointed, this simplifies into the condition that the canonical maps
$\Omega^m E\vee \Omega^n E \to \Omega^m E\times \Omega^n E$ are invertible (where $\Omega^n E$ are the iterated loop spaces at the base point).
In other words, a pointed object $E$ is 1\=/nilpotent if the subcategory spanned by the $\Omega^n E$ is {\it additive} (in the sense that finite sums and products coincides).
In $\PSp$, the 1\=/nilpotent objects are precisely spectra $\Sp\subseteq \PSp$.
In particular, we recover that they are \oo connected.

\paragraph{Modality towers}
Although our main focus here is the towers associated to congruences, 
we want to underline that there exists also important towers associated to acyclic classes/modalities.
The paradigmatic example is given by Postnikov towers.
We mentioned that the acyclic powers of the class $\Surj$ of surjections are the classes $\Conn n$ of $n$\=/connected maps ($\Surj^{n+2} = \Conn n$).
These powers can be organized in a decreasing sequence 
\[
\Iso
\ \subseteq\ 
\Conn\infty=\bigcap_n\Conn n
\ \subseteq\ 
\dots
\ \subseteq\ 
\Conn 1
\ \subseteq\ 
\Conn 0
\ \subseteq\ 
\Conn {-1}=\Surj
\]
Let $\cE\quotient\Conn n$ denote the presentable category obtained by inverting all the maps in $\Conn n$.
The quotient ${\cE\to \cE\quotient\Conn n}$ is given by the factorization system (the modality) generated by $\Conn n$.
A classical computation shows that ${\cE\quotient\Conn n = \cE\truncated n}$ is the category of $n$\=/truncated objects in $\cE$.
The previous sequence is then dual to the {\it Postnikov sections} of the logos $\cE$
\[
\cE
\stto
\cE\quotient\Conn\infty
\stto
\lim_n \cE\quotient\Conn n
\stto
\dots
\stto
\cE\quotient\Conn 0 = \cE\truncated 0
\stto
\cE\quotient\Conn {-1} = \cE\truncated {-1}\,.
\]
In this tower, only $\cE\quotient\Conn\infty$ and $\lim_n \cE\quotient\Conn n$ are logoi.
Moreover the two morphisms (and their composite) $\cE\to \cE\quotient\Conn\infty \to \lim_n \cE\quotient\Conn n$ are morphisms of topoi, and they provide the image factorization of the 
$\cE \to \lim_n \cE\quotient\Conn n$ (into a quotient followed by a conservative morphism of logoi).
This construction can be generalized in a number of ways to arbitrary acyclic classes/modalities, see \cref{sec:modality-tower}.

\paragraph{\oo Connected maps}
When higher topoi were first studied \cite{Rezk:topos,TV:hag1}, the existence of \oo connected maps (\ie the failure of Whitehead theorem) was an oddity.
They have since been accepted as an important feature \cite{Lurie:HTT,Hoyois:Locus}, but still, somehow, an odd one.
This work hopes to shed a new light on them.
We have mentioned that nilpotent maps are always \oo connected.
Completion towers are entirely made of epic quotients (inverting \oo connected maps).
This puts \oo connected maps at the core of the differential calculus of topoi sketched here. 
Not every \oo connected map is expected to be nilpotent (although we do not have a counter-example at this time, see \cref{rem:nilrad-in-hrad}).
This is why we chose to compare the class of \oo connected maps to the Jacobson radical, which is larger than the nilradical.
This analogy has been very enlightening to us.

The central role played by \oo connected maps in the calculus of topoi also explains why this calculus is invisible in 1-topoi and set theory, and why it took homotopy theory---that is the study of \oo groupoids---to unravel it.
In 1-topos theory (and also $n$\=/topoi, for $n<\infty$), all \oo connected maps are isomorphisms and all congruences are monogenic.
The acyclic product becomes thus idempotent (\cref{thmC}) and the whole completion tower collapses into the single quotient map $\cE\to \cE\quotient\cK$.
Geometrically, one could say that there is no calculus because every sub-1-topos is ``ample'', that is, it already contains its infinitesimal neighborhood.

\setcounter{secnumdepth}{3}

\paragraph{Acknowledgments}
The authors thanks 
Marc Hoyois and Mike Shulman for useful discussions on the material of this paper.
They also thank the anonymous reviewer for their careful reading and all their suggestions.

The first author gratefully acknowledges the support of the Air Force Office of Scientific Research through grant FA9550-23-1-0434.
The second author gratefully acknowledges the support of the Simons Foundation and the Centre de Recherches Math\'ematiques (CRM) at the Universit\'e de Montr\'eal through a CRM-Simons Visiting Researcher position, as well as the Centre interuniversitaire de recherches en g\'eom\'etrie et topologie (CIRGET) at the Universit\'e du Qu\'ebec \`a Montr\'eal (UQAM) through a Visiting Professor position.
He would like to acknowledge also the support of the SFB 1085 Higher Invariants at the Universit\"at Regensburg and in particular the support of Denis-Charles Cisinski.
The third author gratefully acknowledges the support of the Air Force Office of Scientific Research through grant FA9550-23-1-0029.
The last author acknowledges the support of the Natural Sciences and Engineering Research Council of Canada through grant 371436.

\section{Preliminaries}
\label{sec:preliminaries}

This section introduces the basic definitions and recollects results from our previous works.
To spare some back and forth between papers to the reader, we have reproduced the definitions and statements (even some very basic ones) from \cite{ABFJ:GC,ABFJ:GBM,ABFJ:HS,ABFJ:GT} that were necessary for this paper.
This has substantially lengthened it.
The reader familiar with the material in this section may focus on the few new statements (the ones with proofs).
For the others, these preliminaries are intended to be an introduction to our previous work.

The original material is mostly in \cref{sec:suspension}, where the intricate relations between the suspension and the decalage of acyclic classes are established
(\cref{thm:suspension-acyclic,thm:explicit-decalage}), and in \cref{sec:lex-gen} where we introduce the useful notion of a lex generator for a congruence.

\subsection*{Conventions}
\label{sec:conventions}

Throughout the paper, we use the language of higher category theory.
We will simplify the vocabulary and drop the prefix ``$\infty$'' when referring to higher categories and their associated notions.
The word \emph{category} refers to \ooo category, and all constructions are assumed to be homotopy invariant.
Furthermore, we work in a model independent style, which is to say, we do not choose an explicit set theoretical model for \ooo categories such as quasicategories, but rather give arguments which we feel are robust enough to hold in any model.
We will refer to the work of Lurie \cite{Lurie:HTT,Lurie:HA} for the general theory of \oo categories and \oo topoi.
Other useful references are \cite{Cisinski:book,Riehl-Verity:EICT}.

We use the word \emph{space} to refer generically to a homotopy type or \oo groupoid.
We denote the category of spaces by $\cS$.
We shall say that a map between two spaces $f:X\to Y$ is an {\it isomorphism} if it is a homotopy equivalence.
We say an object is \emph{unique} if the space it inhabits is contractible.
For example, the inverse of an isomorphism is unique in this sense.

If $\cC$ is a category, we write $A\in\cC$ to indicate that $A$ is an object of $\cC$.
We shall denote by $\cC(A,B)$ or by $\Map AB$ the space of maps between two objects $A$ and $B$, we shall write $f:A\to B$ to indicate that $f\in \cC(A,B)$.
The \emph{opposite} of a category $\cC$ is denoted $\cC\op$. 
We write $\cC\slice A$ for the slice category of $\cC$ over an object $A$.
If $f : X \to A$ is a map of $\cC$, we often write $(X,f) \in \cC\slice{A}$, as it is frequently convenient to have both the object and structure map visible when working in a slice category.

Every category $\cC$ has a {\it homotopy category} $\ho\cC$ which is a 1-category with the same objects as $\cC$, but where $\ho\cC(A,B)=\pi_0\cC(A,B)$.
We shall say that a map $f:A\to B$ in $\cC$ is {\it invertible}, or that it is an {\it isomorphism}, if the map is invertible in the homotopy category $\ho\cC$. 
We make a small exception to this terminology with regard to equivalence of categories: we continue to employ the more traditional term {\it equivalence}.
A category is a groupoid if every map is invertible.
Every category has a maximal subgroupoid, which is the subcategory of isomorphisms.
We shall sometimes refer to this groupoid as the {\it core}.
As it is common un algebraic topology, we shall also use the {\it space} as a synonym of groupoid.

We will use the classical terminology of sets and classes to talk about small and large collections of objects in a category, but with the following meaning.
We shall only consider properties of objects that are invariant under isomorphisms.
By a {\it class of objects} in a category $\cC$ we mean a full subgroupoid of its core.
By a {\it set of objects} in a category $\cC$ we mean a class of objects whose collection of connected components is small.
With this convention, {\it any class/set of objects is always replete}.

\medskip
The category of small categories is denoted $\Cat$, and that of large categories is denoted $\CAT$.
By default we shall consider these categories as \ooo categories and not \oooo categories.
We denote the category of functors from $\cC$ to $\cD$ alternatively by $\fun \cC \cD$ or $\cD^{\cC}$ as seems appropriate from the context.
The arrow category of $\cC$ is $\Arr\cC:=\fun{[1]}\cC$ where $[1]=\{0<1\}$ is the category with one arrow.
By a {\it class (set) of maps} in a category $\cC$ we mean a class (set) of objects in $\Arr\cC$.

We shall say that a functor $F:\cC\to \cD$ is {\it surjective} if for every object $X\in\cD$ there exists an object $A\in\cC$ together with an isomorphism $X\simeq FA$.
We shall say that $F$ is {\it fully faithful} if the induced map $\cC(A,B)\to \cD(FA,FB)$ is invertible for every pair of objects $A,B\in\cC$.
A functor $F$ is an equivalence of categories if and only if it is fully faithful and surjective.

\medskip
When a functor $F:\cC\to \cD$ is left adjoint to a functor $G:\cD\to \cC$, we shall write $F\dashv G$.
When representing adjoint functors horizontally, our convention will be that the functor on top is left adjoint to the one below.
For example, for three adjoint functors $F\dashv G\dashv H$, we shall write 
\[
\begin{tikzcd}
\cC \ar[rr,"G" description] \ar[from=rr, shift right = 3,"F"'] \ar[from=rr, shift left = 3,"H"]
&& \cD
\end{tikzcd}
\]
Beware that, with this convention, left adjoint functors are not always oriented from the left to the right (and vice-versa for right adjoints).

\subsection{Topo-logy}
\label{sec:topos}

In this section, we collect results on congruences from \cite{ABFJ:HS,ABFJ:GT} and
introduce our new vocabulary for topos theory, extending the one from 
\cite{Anel-Lejay:topos-exp,Anel-Joyal:topo-logie}.

\medskip

\begin{definition}[Logos {\cite[Definitions 6.1.0.4 and 6.3.1.1]{Lurie:HTT}}]
\label{def:logos}
A category $\cE$ is a {\it logos} if it is an accessible left-exact reflection of the category $\PSh C$ of presheaves over a small category $C$.
A {\it morphism of logoi} is a functor $\phi:\cE\to \cF$ which is preserves colimits and finite limits (cocontinuous and left-exact functor).
Such a morphism has automatically a right adjoint $\phi_*$.
We denote by $\Logos$ the \ooo category of logoi and their morphisms.
We shall also consider the category $\fun\cE\cF\cclex$ of morphisms between two logoi $\cE$ and $\cF$.
The maximal subgroupoid of $\fun\cE\cF\cclex$ is the hom space $\Map\cE\cF$ in $\Logos$.
\end{definition}

More intrinsically, a logos is a presentable category with descent \cite[Section 6.1.3]{Lurie:HTT}, or with enough univalent morphisms \cite[Section 6.1.6]{Lurie:HTT}.
We shall not use these characterizations in this paper.

\begin{definition}[Topos]
\label{def:topos}
The category of topoi is defined as the opposite of that of logoi: $\Topos:=\Logos\sfop$.
We denote by $\fun\tX\tY$ the category of morphisms between two topoi $\tX$ and $\tY$.
The dual logos of a topos $\tX$ will sometimes be denoted $\Sh\tX$ and called the {\it category of sheaves on $\tX$} (see \cref{def:sheaf}).
\end{definition}
In the notations of \cite{Lurie:HTT}, we have 
\begin{align*}
\fun \tX\tY &= \mathrm{Fun}_*(\Sh\tX,\Sh\tY) &\Topos &= \mathrm{\cR\cT op}  \\
\fun \cE\cF\cclex &= \mathrm{Fun}^*(\cE,\cF)  &\Logos &= \mathrm{\cL\cT op}\,.
\end{align*}

The category of logoi has the advantage to have a nice forgetful functor to the category of large categories $\Logos\to \CAT$.
For example, this functor creates limits of logoi \cite[Proposition 6.3.2.3]{Lurie:HTT}.
Also, this functor has a left adjoint defined on small categories.
If $C$ is a small category, we denote $C\lex$ the completion of $C$ for finite limits.
This category is still small and the presheaf category $\PSh{C\lex}$ is a logos.
The following result is
\cite[Proposition 6.1.5.2]{Lurie:HTT} or
\cite[Proposition 2.3.2]{Anel-Lejay:topos-exp}.

\begin{theorem}[Free logos]
\label{thm:free-logos}
For any logos $\cE$, the restriction along the composition ${C\to C\lex \to \PSh {C\lex}}$ induces an equivalence of categories
\[
\fun {\PSh{C\lex}}\cE\cclex
\ =\ 
\fun C\cE\,.
\]
\end{theorem}

\begin{definition}[Free logos, affine logos]
\label{def:free-logos}
We shall denote the logos $\PSh{C\lex}$ by $\S C$ and call it the {\it free logos on $C$}.
The topos dual to the free logos $\S C$ is denoted $\tA^C$ and called an {\it affine} topos. 
\end{definition}

When $C=\emptyset$, we have $\S \emptyset = \cS$, the initial logos.
The dual topos is terminal, we denote it by $\tone$.
When $C=1$ is a category with one object, we denote the free logos by $\S X:=\PSh{1\lex}$ and say it is the logos {\it classifying objects}.
In this case, we have $1\lex = \Fin\op$ where $\Fin$ is the category of finite spaces, and $\S X = \fun\Fin\cS$.
The {\it universal object} $X$ is the functor corepresented by the terminal object $1\in \Fin$, that is the canonical inclusion $\Fin\to \cS$.
For any logos $\cE$, the evaluation at $X$ provides an equivalence $\fun {\S X}\cE \cclex = \cE$, hence the name of $\S X$.
The topos dual to $\S X$ will be denoted $\tA$ and refer to as the {\it topos line}.

The functor represented by $A\in \Fin$ will be denoted $X^A$, since it is the $A$\=/power of the object $X$ in $\S X = \fun\Fin\cS$.
For any $F:\Fin\to \cS$, the Yoneda lemma gives the formula
\[
F=\int^{K\in \Fin} F(A)\times X^A\,,
\]
which shows that every object $F\in \S X$ is some kind of ``polynomial" $F=F(X)$ in the universal object $X\in \S X$.
For an object $E$ in a logos $\cE$, let $\epsilon_E:\S X\to \cE$ be the morphism of logoi corresponding to $E$ by \cref{thm:free-logos}.
By construction, for every $F\in \S X$ we have 
$\epsilon_E(F)=\int^{K\in \Fin} F(A)\times E^A$
and this formula shows that $\epsilon_E(F)$ is the {\it value} $F(E)$ of the polynomial $F=F(X)$ at $E$.
Thus, $\epsilon_E$ is the {\it evaluation functor} $F\mapsto F(E)$.

\begin{definition}[Sheaf on a topos]
\label{def:sheaf}
We shall say that a {\it sheaf} on a topos $\tX$ is a morphism of topoi $\tX\to \tA$.
By definition, this is the same thing as a logos morphism $\S X \to \Sh\tX$, and, by universal property of $\S X$, the same thing as an object of $\Sh\tX$.
Thus, the logos $\Sh\tX$ dual to $\tX$ is equivalent to the category of sheaves on $\tX$.	
\end{definition}

\subsubsection{Quotients and congruences of presentable categories}
\label{sec:quotient-presentable}
When categories are thought of as a collection of arrows, localizations of categories are analogous to localizations of rings (since they invert arrows).
But when categories are thought of as a collection of objects, localizations of categories are better compared to {\it quotients} (since they identify objects).
For this reason we shall change the classical terminology and always talk about {\it quotients} rather than {\it localizations} of categories.

\begin{definition}[Quotient of presentable categories]
\label{def:quotient-presentable}
A cocontinuous functor $\phi:\cC\to \cD$ between two presentable categories is called a {\it quotient of presentable categories} (or simply a {\it quotient} if the context is clear) if it is a reflective localization, that is if its right adjoint $\phi_*$ is fully faithful.
The class of maps inverted by $\phi$ is a strongly saturated class in the sense of \cite[Definition 5.5.4.5]{Lurie:HTT}.
\end{definition}

Every class of maps $\Sigma$ in a presentable category $\cC$ is contained in a smallest strongly saturated class $\Sigma\ssat$.
A strongly saturated class of maps $\cW$ is said to be of small generation if $\cW=\Sigma\ssat$ for some {\it set} of maps $\Sigma$.
A quotient of presentable categories $\phi:\cC\to \cD$ is said to be {\it generated by $\Sigma$} if it is initial among the quotients inverting $\Sigma$ (sending $\Sigma$ to isomorphisms in $\cD$), or equivalently inverting $\Sigma\ssat$.
Such a quotient is uniquely determined and we shall denote its codomain by $\cC\quotient\Sigma\ssat$.
For any strongly saturated class of small generation $\cW$ in $\cC$, the quotient of $\cC$ by $\cW$ exists and we shall denote it by $\cC\quotient\cW$.

\begin{remark}
In the theory of algebraic structures (\eg groups, rings...), quotients are controlled by {\it congruences}, which are equivalence relations in the categories of these structures.
Strongly saturated class of maps can (and must) be understood as congruences for the notion of cocomplete categories and cocontinuous functors.
\end{remark}

\subsubsection{Quotients and congruences of logoi, inclusions of topoi}
\label{sec:congruence}

\begin{definition}[Quotient of logos, inclusion of topoi]
\label{def:quotient-logos}
A morphism of logoi $\phi:\cE\to \cF$ is called a {\it quotient} if it is a quotient of presentable categories, that is if its right adjoint $\phi_*$ is fully faithful.
By abuse, we shall sometimes say that $\cF$ is a quotient of $\cE$ and forget about $\phi$.
A morphism of topoi $\tX\to \tY$ is called an {\it inclusion} if it is dual to a quotient of logoi. 
We shall also say that $\tX$ is a {\it subtopos} of $\tY$.
\end{definition}

In the same way that quotients of rings are in bijection with ring congruences (equivalences relations in the category of rings), quotients of a logoi $\cE$ are in bijection with a notion of logos congruence.

\begin{definition}[Congruence {\cite[Definition 4.2.1]{ABFJ:HS}}]
\label{def:congruence}
We say that a class of maps $\cK$ in a logos $\cE$ is a {\it logos congruence} if the following conditions hold:
\begin{defenum}
\item $\cK$ contains the isomorphisms (reflexivity) and is closed under composition (transitivity);
\item $\cK$ is closed under colimits and finite limits in the arrow category of $\cE$ (\ie the inclusion $\cK\to \Arr\cE$ is a fully faithful morphism of logoi).
\end{defenum}
When it is clear that we are working within the category of logoi, we shall talk simply of a {\it congruence}.
\end{definition}

A class of maps in a logos is a congruence if and only if it is strongly saturated and closed under base changes \cite[Proposition 4.2.3]{ABFJ:HS}.
In particular, any congruence satisfies the 3-for-2 property.

\medskip
Any intersection of congruences is a congruence.
Any class of maps $\Sigma$ in a logos $\cE$ is contained in a smallest congruence $\Sigma\cong$.
We say that $\Sigma\cong$ is the congruence \emph{generated} by the class of maps $\Sigma$.
A congruence $\cK$ is said to be {\it of small generation} if $\cK=\Sigma\cong$ for some {\it set} of maps $\Sigma$.
A morphism of logoi $\phi:\cE\to \cF$ inverts $\Sigma$ if and only if it inverts the whole congruence $\Sigma\cong$ ($\Sigma\subseteq \cK_\phi \Leftrightarrow \Sigma\cong\subseteq \cK_\phi$).
A quotient of logoi $\phi:\cE\to \cF$ is said to be {\it generated by $\Sigma$} if it is initial among the quotients inverting $\Sigma$ (sending $\Sigma$ to isomorphisms in $\cF$), or equivalently inverting $\Sigma\cong$.
Such a quotient is uniquely determined and we shall denote its codomain by $\cE\quotient\Sigma\cong$.
If $\cK$ is a congruence of small generation on $\cE$, the quotient of $\cE$ by $\cK$ exists and will be denoted by $\cE\quotient\cK$.
Notice that, since congruences are strongly saturated classes, quotients of logoi are quotients of presentable categories and there is no problem in using the same notation as in \cref{sec:quotient-presentable}.

\begin{examples}
\label{exmp:congruence} 

\begin{exmpenum}
\item\label{exmp:congruence1} 
The classes $\Iso$ of isomorphisms and $\All$ of all maps are respectively the smallest and the largest congruences (for the inclusion relation).
They are of small generation since they are generated respectively by $\{1\xto = 1\}$ and $\{0\to 1\}$.

\item\label{exmp:congruence:simple}
Let us say that a logos is {\it simple} if its only congruences are $\Iso$ and $\All$.
In the analogy with rings, such logoi behave like fields.
The initial logos $\cS$ is simple.
More generally, for any group $G$ in $\cS$ with classifying space $BG$, the slice logos $\cS\slice {BG}$ (which is also the logos of spaces with a $G$\=/action) is simple. 
Are there more simple logoi?

\item\label{exmp:congruence2} 
Let $\phi:\cE\to \cF$ be a morphism of logoi.
Recall that such morphisms preserve isomorphisms, compositions, colimits and finite limits.
Then for any congruence $\cK$ in $\cF$, the class $\phi^{-1}(\cK) =\{f\in \cE\ |\ \phi(f)\in \cK\}$ is a congruence on $\cE$.
In particular, the class $\cK_\phi:= \phi^{-1}(\Iso)$ of maps inverted by $\phi$ is a congruence.
We shall refer to $\cK_\phi$ as the {\it congruence of the morphism $\phi$}. 
\label{Luriethm2} 
The congruence $\cK_\phi$ is always of small generation \cite[Proposition 5.5.4.16]{Lurie:HTT}.

\item\label{exmp:congruence3} 
Let $\phi:\cE\to \cF$ be a left-exact localization.
Then for any congruence $\cK$ in $\cE$ such that $\cK_\phi\subseteq \cK$, its image $\phi(\cK)$ is a congruence on $\cF$.
Moreover, we have $\cK = \phi^{-1}(\phi(\cK))$.

\item\label{exmp:congruence4}
The class $\Conn \infty$ of \oo connected maps is a congruence of small generation \cite[Proposition 6.5.2.8]{Lurie:HTT}.

\end{exmpenum}
\end{examples}

The following theorem from \cite{Lurie:HTT} proves that congruences control all left-exact localizations.
Contrary to the case of 1-logoi, it is not known whether all left-exact localizations of logoi are accessible.
Therefore, a condition of small generation must be imposed.
For $\cE$ a fixed logos, we consider
the (large) poset $\Cong\cE$ of all congruences in $\cE$ (ordered by inclusion), 
the subposet $\Congsg\cE\subseteq \Cong\cE$ of congruences of small generation in $\cE$,
and the poset $\mathsf{LexLoc_{acc}}(\cE)$ of (isomorphism classes of) accessible left-exact localizations of $\cE$.
The map $\phi\mapsto \cK_\phi$ defines a morphism of posets $\mathsf{LexLoc_{acc}}(\cE) \to \Congsg\cE$.
Conversely, if $\cK=\Sigma\cong$ is a congruence of small generation, then the localization $\phi_\cK:\cE\to \cE\quotient\cK$ exists and is accessible \cite[Propositions 5.5.4.15 and 6.2.1.2 together]{Lurie:HTT} and we get a function
$\Congsg\cE \to \mathsf{LexLoc_{acc}}(\cE)$.

\begin{theorem}[{\cite[Propositions 5.5.4.2 and 6.2.1.1 together]{Lurie:HTT}}]
\label{thm:bij-congruence-lexloc}
The functions $\phi\mapsto \cK_\phi$ and $\cK \mapsto \phi_\cK$ define inverse isomorphisms of posets
\[
\mathsf{LexLoc_{acc}}(\cE)
\ \simeq\ 
\mathsf{Cong_{sg}}(\cE)\,.
\]
\end{theorem}
In the rest of the paper, we will work with arbitrary congruences, and not only those of small generation.
We shall mention explicitly when the hypothesis of small generation is needed.
Notice that the poset $\Cong\cE$ is a large suplattice.
The subposet $\mathsf{Cong_{sg}}(\cE)$ is also large, but with small suprema only.
This makes the structure of $\Cong\cE$ sligthly easier to handle.

\medskip

For a congruence of small generation, the quotient $\cE\to \cE\quotient\cK$ has a fully faithful right adjoint that identifies $\cE\quotient\cK$ to the category of local objects for the class $\cK$.
When $\cK=\Sigma\cong$, our main theorem in \cite{ABFJ:HS} describes these local objects as the  {\it $\Sigma$\=/sheaves}:
let $\Sigma\diag$ be the closure of $\Sigma$ by diagonals, 
and let $(\Sigma\diag)\bc$ be the closure of $\Sigma\diag$ under base change, then a $\Sigma$\=/sheaf is an object local for the class $(\Sigma\diag)\bc$.
We shall come back to this condition in \cref{sec:recognition}.

\begin{theorem}[{\cite[Theorem 4.3.3]{ABFJ:HS}}]
\label{thm:sigma-sheaves}
For any set of maps $\Sigma$ in a logos $\cE$, the right adjoint to the quotient $\cE\to \cE\quotient\Sigma\cong$ identifies $\cE\quotient\Sigma\cong$ to the full subcategory of $\Sigma$\=/sheaves.
\end{theorem}

This theorem is particularly handy since any logos can be presented as the quotient of a free logos $\cE=\S C\quotient\Sigma\cong$ for some small category $C$ and a set of maps $\Sigma$.
We call such a pair $(C,\Sigma)$ a {\it presentation} of the logos $\cE$.
It is an alternative to sites (in the sense of \cite[Definition 3.1.1]{TV:hag1} and \cite[Definition 6.2.2.1]{Lurie:HTT}), which is better suited for many purposes.
Given a presentation, we shall say that $\cE$ {\it classifies} the $C$\=/diagrams inverting $\Sigma$.

\begin{examples}
\label{exmp:classifying-logos}

\begin{exmpenum}

\item\label{exmp:classifying-logos:X}
Any quotient of the free logos $\S X$ corresponds to a property on $X$ which is preserved by logoi morphisms. 
Such quotients distinguish subtopoi $\tA'\subseteq \tA$ of the topos line.

\item\label{exmp:classifying-logos:trunc}
Recall that an object $X$ is {\it $n$\=/truncated} (for $-1\leq n$) if the higher diagonal $\Delta^{n+2} X:X\to X^{S^{n+1}}$ is invertible.
The quotient $\S X\quotient\{\Delta^{n+2} X\}\cong$ classifies $n$\=/truncated objects.
See \cite[Section 5.3]{ABFJ:HS} for details.

\item\label{exmp:classifying-logos:conn}
Similarly, an object is {\it $n$\=/connected} (for $-1\leq n\leq \infty$) if the diagonals $\Delta^k X$ are all surjective for ${0\leq k\leq n+1}$.
If $\im f$ is the image of the map $f$, the logos $\S X\quotient\{\im{\Delta^k X}\,|\,k\leq n+1\}\cong$ classifies $n$\=/connected objects.
See \cite[Section 5.4]{ABFJ:HS} for details.

\item\label{exmp:classifying-logos:pointed}
Let $\S {Y\to X}$ be the logos free on one arrow. 
The quotient by the congruence $\{Y\to 1\}\cong$ is the logos $\S {X\pointed}$ classifying pointed objects.
This logos can also be presented as $\S {X\pointed}=\fun\Finp\cS$.

\end{exmpenum}
\end{examples}

\begin{lemma}
\label{lem:SX}
We consider $\S X= \fun \Fin \cS$, the free logos on one generator, and the congruence $\idl X:= \{X\to 1\}\cong$ generated by $X$.
The quotient $\S X \quotient \idl X$ is $\cS$ and the quotient morphism $\S X\to \cS$ is given by evaluation at $1\in \Fin$.
\end{lemma}
\begin{proof}
The evaluation at $1\in \Fin$ is the colimit functor $\fun\Fin\cS\to \cS$ since 1 is terminal in $\Fin$.
Its right adjoint is then the inclusion of constant functors, which is a fully faithful functor.
Thus $\fun\Fin\cS\to \cS$ is a localization.
It is a left-exact localization, hence quotient of logoi, since evaluation at 1 preserves all limits.
Let $\cG$ be the congruence of $ev_1:\fun\Fin\cS\to \cS$.
It is the class of maps $F\to G$ inducing an equivalence $F(1)=G(1)$.
The maps $X\to 1$ is inverted by evaluation at 1, hence $\{X\to 1\}\cong\subseteq \cG$.
Let us now prove to converse inclusion.
Congruences are closed under finite limits, so all the maps $X^A\to 1$, for $A\in\Fin$, are in $\{X\to 1\}\cong$.
Recall that any functor $F:\Fin\to \cS$ is a colimit of the $X^A$ indexed by its category of elements $el(F)$. 
The colimit of the maps $X^A\to 1$ indexed by $el(F)$ is the map $F\to F(1)$, where the value $F(1)$ is viewed as a constant functor.
Since congruences are closed under colimits, all maps $F\to F(1)$ are in $\{X\to 1\}\cong$.
Then any map $F\to G$ inducing an equivalence $F(1)=G(1)$, is also in $\{X\to 1\}\cong$ by the 3-for-2 condition.
This proves $\cG\subseteq \{X\to 1\}\cong$ and the statement.
\end{proof}

\subsubsection{Monogenic and hyper-radical congruences, Grothendieck topologies}
\label{sec:mono-cong}
\label{sec:Gtop}

We recall results from \cite{ABFJ:GT} and introduce some new terminology.
We denote by $\Mono$ the class of monomorphisms in a logos.

\begin{definition}[Monogenic congruence]
A congruence $\cK$ is said to be {\it monogenic} if $\cK=\Sigma\cong$ for $\Sigma$ a class of monomorphisms in $\cE$.
In fact, $\Sigma$ can be assumed to be a set, since any monogenic congruence is always of small generation (see \cite[Corollary 6.2.1.6]{Lurie:HTT} or \cite[Proposition 4.1.4]{ABFJ:GT}).
Consequently, the quotient $\cE\quotient\cK$ by a monogenic congruence always exists.
We shall say that it is a {\it monogenic quotient}.
In \cite{ABFJ:GT}, monogenic congruences were called {\it topological}.
Every congruence $\cK$ contains a maximal monogenic congruence $\cK\mono\subseteq\cK$ 
which is $\cK\mono=(\cK\cap \Mono)\cong$ \cite[Corollary 4.1.11]{ABFJ:GT}.
We always have $\cK\mono\subseteq \cK$ and we shall say that $\cK\mono$ is the {\it monogenic part} of $\cK$.
\end{definition}

\begin{examples}
\label{exmp::mono-congruence} 
\begin{exmpenum}
\item\label{exmp:mono-congruence1} 
The congruences $\Iso$ and $\All$ are monogenic since they are generated respectively by $\{1\xto = 1\}$ and $\{0\to 1\}$.

\item\label{exmp:mono-congruence2}
When $\Sigma = \{U\to 1\}$ for a subterminal object $U$, we shall say that the quotient $\cE\to \cE\quotient\{U\to 1\}$ is {\it open} and that the dual inclusion $\tZ\to \tX$ is an {\it open inclusion}.
In that case, the quotient $\cE\to \cE\quotient{\{U\to 1\}}$ is the functor $U\times-:\cE\to \cE\slice U$ with values in the slice of $\cE$ over $U$.

\item\label{exmp:mono-congruence3}
When $\Sigma = \{0\to U\}$ for a subterminal object $U$, we shall say that the quotient is {\it closed} and that the dual inclusion $\tZ\to \tX$ is an {\it closed inclusion}.

\item\label{exmp:mono-congruence:site}
The category $\Sh{C,\tau}$ of sheaves on a site defines a quotient $\phi:\PSh C \to \Sh{C,\tau}$ whose associated congruence $\cK_\phi$ is monogenic.
Precisely, if $\Sigma$ is the collection of all covering sieves of the topology $\tau$, then $\cK_\phi=\Sigma\cong$ \cite[Lemma 4.3.9]{ABFJ:HS}.

\item\label{exmp:mono-congruence:conn}
The congruence used in \cref{exmp:classifying-logos:conn} to present the logos $\S {X\connected n}$ classifying $n$\=/connected objects is monogenic.

\item\label{exmp:mono-congruence:oo-conn}
Recall the congruence $\Conn\infty$ of \oo connected maps from \cref{exmp:congruence4}.
Since $\Conn\infty\subseteq\Surj$, we have $\Conn\infty\cap\Mono \subseteq \Surj\cap\Mono\subseteq\Iso$, hence $\Conn\infty\mono = \Iso$.

\item\label{exmp:mono-congruence:oo-conn-obj}
Recall $\S X= \fun \Fin \cS$ and the congruence $\idl X:= \{X\to 1\}\cong$.
From \cref{lem:SX}, we know that $\S X \quotient \idl X=\cS$.
The monogenic part $\idl X\mono$ is the congruence $\{\im{\Delta^nX}\}\cong$.
This is the congruence forcing all iterated diagonals $\Delta^n X$ to be surjective, that is forcing $X$ to be \oo connected, see \cref{exmp:classifying-logos:conn}.
The corresponding quotient is denoted $\S {X\connected\infty}$, where the universal \oo connected object $X\connected\infty$ is the image of $X$.

\item\label{exmp:mono-congruence:all-mono}
All the congruences in $\cS$ are monogenic by \cref{exmp:mono-congruence1}.
More generally, for any space $A\in \cS$ all the congruences of the slice logos $\cS\slice A$ are monogenic.
To see it, we can use $\cS\slice A=\prod_i \cS\slice {A_i}$ where the $A_i$ are the connected components of $A$.
A congruence in $\cS\slice A$ is then a choice of congruence in each $\cS\slice {A_i}$ and the result follows from \cref{exmp:congruence:simple}
Are there other examples of logoi for which all quotients are monogenic?

\end{exmpenum}
\end{examples}

\begin{remark}[Monogenic quotients as localizations of rings]
\label{rem:loc=mono}
In the analogy between logoi and rings, one difference is that all subtopoi (open, closed, or else) are dual to logos quotients, but not all subschemes are dual to ring quotients.
Only the closed subschemes are dual to ring quotients, others can be defined by localizations, or by a mix of both constructions.
As it happens, there are good reasons to compare monogenic quotients to localizations.
Developing this in details would be too long, so we shall only give an informal explanation.
If $A\to A[S^{-1}]$ is a localization of rings, it is the union (in the poset of localizations) 
of the family of principal localizations $A\to A[s^{-1}]$, for every $s$ in $S$.
Geometrically, this means that the subscheme dual to $A[S^{-1}]$ is the intersection of the open subschemes dual to the $A[s^{-1}]$.
On the logos side, now, every monogenic congruence $\cK$ can be generated by a single monomorphism $A\to B$ \cite[Proposition 4.1.4]{ABFJ:GT}.
Morally, any monomorphism can be thought of as a family $A(b)$ of subterminal objects in $\cE$ indexed by the sheaf $B$.
Using this intuition, the quotient $\cE\to\cE\quotient\cK$ becomes the union (in the poset of quotients) of the family of slices $\cE\to \cE\slice {A(b)}$, for ``$b$ in $B$''.
Geometrically, this proposes to think of the subtopos dual to $\cE\quotient\cK$ as the intersection of the open subtopoi dual to the $\cE\slice {A(b)}$.
See \cref{rem:triple-facto} for a continuation of this remark.
\end{remark}

\begin{definition}[Hyper-radical of a congruence]
\label{def:hrad}
For a congruence $\cK$ in a logos $\cE$, 
a map $f$ in $\cE$ is called a {\it $\cK$\=/covering} if $\im f\in \cK$,
and a {\it $\cK$\=/hypercovering} if all iterated diagonals of $f$ are $\cK$\=/coverings (\ie if, for all $n$, $\im{\Delta^n f}\in \cK$) \cite[Definition 5.1.1]{ABFJ:GT}.
Every map in $\cK$ is a $\cK$\=/hypercovering.
The class of all $\cK$\=/hypercoverings is denoted $\hrad\cK$ and called the {\it hyper-radical} of $\cK$.
It is always a congruence of small generation 
\cite[Proposition 5.1.5 and Corollary 5.1.16]{ABFJ:GT}.
Moreover, we always have ${\hrad{\hrad\cK} = \hrad\cK}$ \cite[Lemma 5.1.9]{ABFJ:GT}.
A congruence $\cK$ is said to be {\it hyper-radical} if $\cK=\hrad\cK$ (we used to say hypercomplete in \cite{ABFJ:GT}).
If $\cK$ is of small generation, the quotient $\cE\quotient\cK\to \cE\quotient{\hrad\cK}$ is the hypercompletion of $\cE\quotient\cK$ \cite[Corollary 5.1.13]{ABFJ:GT}, but we shall rather say that it is the {\it hyper-reduction} of $\cE\quotient\cK$.
Moreover, a congruence $\cK$ of small generation is hyper-radical if and only if the logos $\cE\quotient\cK$ is hyper-reduced \cite[Corollary 5.1.16]{ABFJ:GT}.
 \end{definition}

We recall some lemmas that will be useful.

\begin{lemma}[{\cite[Lemma 5.1.4]{ABFJ:GT}}]
\label{lem:hypercovering}
Given a quotient $\phi:\cE\to \cE\quotient\cK$, a map $f$ is a $\cK$\=/hypercovering in $\cE$ if and only if $\phi(f)$ is \oo connected.
\end{lemma}

\begin{lemma}[{\cite[Lemma 5.1.12]{ABFJ:GT}}]
\label{lem:hypercovering-mono}
For any congruence $\cK$ we have $\hrad\cK = \hrad{\cK\mono}$.
\end{lemma}

Recall that a map is \oo connected if and only if, for all $n$, $\im{\Delta^n f}\in \Iso$.
This shows that $\hrad \Iso = \Conn\infty$.
We shall give a special name to the congruence $\Conn\infty$.

\begin{definition}[Hyper-radical of a logos]
\label{def:radical-logos}
For a logos $\cE$, the congruence $\Conn \infty$ of \oo connected maps is called the {\it hyper-radical} of $\cE$.
We shall sometimes denote it by $\hrad 1$ or $\hrad \Iso$.
We shall say that a logos (and the dual topos) is {\it hyper-reduced} (instead of hypercomplete) if $\hrad 1 = \Iso$.
The logos morphism $\cE\to \cE\quotient{\hrad1}$ is the {\it hyper-reduction} of $\cE$.
\end{definition}

\begin{remark}
\label{rem:will-def-nilrad}
We use the prefix `hyper-' to keep in touch with the existing terminology but also to distinguish the hyper-radical from the nil-radical that we shall define in \cref{sec:nilradical}.
\end{remark}

\begin{examples}
\label{exmp::rad-congruence} 
\begin{exmpenum}

\item\label{exmp:rad-congruence1}
By \cref{lem:SX} we have $\S X\quotient\idl X=\cS$. This shows that the congruence $\idl X$ is hyper-radical, since $\cS$ is hyper-reduced.

\item\label{exmp:rad-congruence2}
In $\S {X\connected\infty}$ of \cref{exmp:mono-congruence:oo-conn-obj}, 
let us see that the congruence $\idl{X\connected\infty}$ is the hyper-radical.
We have $\S {X\connected\infty}\quotient\idl{X\connected\infty} = \S X\quotient\idl X = \cS$.
Since $\cS$ is hyper-reduced, we must have $\hrad 1\subseteq \idl{X\connected\infty}$.
Conversely, since $X\connected\infty$ is \oo connected, we must have $\idl{X\connected\infty}\subseteq\hrad 1$.
This shows that $\cS$ is the hyper-reduction of $\S {X\connected\infty}$.
\end{exmpenum}
\end{examples}

\begin{remark}
\label{rem:ampli-reduc}	
Any congruence $\cK$ sits between its monogenic part $\cK\mono$ and its radical $\hrad\cK$.
Geometrically, if $\tZ\subseteq \tX$ is the inclusion of topoi dual to the quotient $\cE\to \cE\quotient\cK$, we shall call the topos dual to $\cE\quotient\cK\mono$ the {\it amplification} of $\tZ$, and the topos dual to $\cE\quotient\hrad\cK$ the {\it hyper-reduction} of $\tZ$.
Any subtopos $\tZ\subseteq \tX$ sits between its reduction and its amplification.
\[
\begin{tikzcd}[row sep=5]
\Iso \ar[r, hook] & \cK\mono \ar[r, hook] & \cK\ar[r, hook] & \hrad\cK\\
\cE \ar[r, two heads] & \cE\quotient\cK\mono \ar[r, two heads] & \cE\quotient\cK\ar[r, two heads] & \cE\quotient\hrad\cK \\
\tX \ar[from=r, hook'] & \tZ\ampli \ar[from=r, hook'] & \tZ \ar[from=r, hook'] & \tZ\hred\,.
\end{tikzcd}
\]
\end{remark}

\medskip

We denote by $\MCong\cE$ and $\RCong\cE$ the posets of monogenic and hyper-radical congruences.
We showed in \cite[Theorem 5.2.4]{ABFJ:GT} that these posets are naturally in bijection.
More precisely, they are both in bijection with the poset $\GTop(\cE)$ of extended Grothendieck topologies on $\cE$.
Extended Grothendieck topologies are an axiomatization of the classes of monomorphisms inverted by morphisms of logoi.
We shall mostly work with monogenic congruences, but some statements are sometimes more meaningful when formulated in terms of topologies.
We recall the definition for the reader's convenience.

\begin{definition}[{\cite[Definition 3.1.2]{ABFJ:GT}}]
\label{def:GT}
An {\it extended Grothendieck topology} on a logos $\cE$ is a class of monomorphisms $\cG$ such that 
\begin{enumerate}[label=\roman*)]
\item $\cG$ is a local class \cite[Definition 2.3.9]{ABFJ:HS},
\item $\cG$ contains all isomorphisms and is closed under composition,
\item\label{GT:axiom3} if the composite of two monomorphisms $u:A\to B$ and $v:B\to C$ belongs to $\cG$, then so does $v$.
\end{enumerate}
\end{definition}

When $\cE=\PSh C$ is a presheaf category, there is a bijection between 
Grothendieck topologies on $C$ in the usual sense
and extended Grothendieck topologies on $\PSh C$
\cite[Proposition 3.1.6]{ABFJ:GT}.
The latter notion is a version of the former that is independent of any generators.
The adjective `extended' is meant to prevent the confusion between the two notions.
When the context is clear though, we shall drop it and simply talk of a Grothendieck topology, or a topology, on a logos $\cE$.

For any congruence $\cK$, the class $\cK\cap \Mono$ is always a Grothendieck topology and all Grothendieck topologies can be obtained this way.
This defines a surjective morphisms of posets $\Cong\cE\to \GTop(\cE)$.
This morphism has a fully faithful left adjoint that sends a Grothendieck topology $\cG$ to the monogenic congruence $\cG\cong$,
and a fully faithful right adjoint sending a Grothendieck topology $\cG$ to the hyper-radical congruence $\hrad{\cG\cong}$ \cite[Theorem 5.2.4]{ABFJ:GT}.
The fully faithful adjoints provide the equivalences $\MCong\cE=\GTop(\cE)=\RCong\cE$.
We reformulate this result through the equivalence $\GTop(\cE)=\MCong\cE$.

\begin{theorem}[{\cite[Theorem 5.2.4 with Remark 5.2.6]{ABFJ:GT}}]
\label{thm:mono-cong}
The inclusion $\MCong\cE\to \Cong\cE$ has a right adjoint given by $\cK\mapsto\cK\mono$ and a second right adjoint given the morphism sending a monogenic congruence to its hyper-radical.
\[
\begin{tikzcd}
\Cong\cE 
&&&\MCong\cE
\ar[lll, "{\sf can.}"', hook', shift right = 3]
\ar[from=lll, "(-)\mono" description] 
\ar[lll, "\hrad-", hook', shift left = 3]
\end{tikzcd}	
\]
This further right adjoint is fully faithful and its image is the subposet $\RCong\cE$ of hyper-radical congruences.
\end{theorem}

\subsubsection{Epic congruences and the quotient triple factorization}
\label{sec:triple-facto}

Let $\cK_\phi$ be the congruence of a morphism of logoi $\phi:\cE\to \cF$, and $\cK_\phi\mono$ its monogenic part.
Then $\phi:\cE\to \cF$ can be factored in three morphisms
\begin{equation}
\label{eq:triple-facto}	
\begin{tikzcd}
\cE 
\ar[r,"\phi\mono"]\ar[rr, bend right, "\phi\quot"']
&\cE\quotient\cK_\phi\mono  
\ar[r,"\phi\epi"]\ar[rr, bend left, "\phi\wcons"]
&\cE\quotient\cK_\phi 
\ar[r,"\phi\cons"']
&\cF \,.    
\end{tikzcd}
\end{equation}

\begin{definition}[Quotient/image factorization]
\label{def:quotient-facto}
We shall refer to the diagram~\eqref{eq:triple-facto} as the {\it quotient triple factorization} of $\phi$, or the {\it image triple factorization} if we think about it in terms of topoi.
We shall use the terminology and notation of \cref{table:facto} to talk about the various objects involved.
\end{definition}

\begin{definition}[Epic congruence and quotient]
\label{def:epic}
We shall say that a congruence $\cK$ is {\it epic} if $\cK\subseteq \Conn\infty$.
Equivalently, $\cK$ is epic if $\cK\subseteq \Surj$ \cite[Proposition 4.2.2]{ABFJ:GT}.
In \cite{ABFJ:GT}, epic congruences were called {\it cotopological}.
We shall say that a quotient of logoi $\phi:\cE\to \cF$ is {\it epic} if the congruence $\cK_\phi$ is epic.
\end{definition}

\begin{examples}
\label{exmp:epic-congruence} 
\begin{exmpenum}
\item $\Conn\infty$ is the largest epic congruence. The hyper-reduction $\cE\to \cE\quotient\Conn\infty$ is an epic quotient.
\item For an arbitrary congruence $\cK$, its image by the quotient $\phi\mono:\cE\to \cE\quotient\cK\mono$ is an epic congruence \cite[Proposition 4.2.5]{ABFJ:GT} that we shall denote $\cK\epi$.
When $\cK$ is of small generation, we have
\[
\cE\quotient\cK = (\cE\quotient\cK\mono)\quotient\cK\epi\,.
\]
\end{exmpenum}
\end{examples}

The following result is \cite[Proposition 4.3.2]{ABFJ:GT} specialized to a quotient of logoi.
For any $-1\leq n\leq \infty$, we shall say that an algebraic morphism of topoi $\phi:\cE\to \cF$ {\it reflects $n$\=/connected maps} if, for a map $f$ in $\cE$, 
$f$ is $n$\=/connected if and only if $\phi(f)$ is $n$\=/connected in $\cF$.
Since $f\in \Conn n(\cE)\Rightarrow \phi(f)\in \Conn n(\cF)$ is always true, $\phi$ reflects $n$\=/connected maps if and only if $\phi^{-1}(\Conn n(\cF))\subseteq \Conn n(\cE)$.

\begin{proposition}[Characterization of epic quotients]
\label{morphismvslemmacotop}
The following conditions on a morphism of logoi $\phi:\cE\to \cF$ are equivalent:
\begin{propenum}
\item\label{morphismvslemmacotop:0} $\phi$ is an epic quotient ;
\item\label{morphismvslemmacotop:1} $\phi$ is a quotient inverting no monomorphism (\ie $\cK_\phi\cap \Mono=\Iso$);
\item\label{morphismvslemmacotop:2} $\phi$ is a quotient reflecting surjective maps;
\item\label{morphismvslemmacotop:3} $\phi$ is a quotient reflecting $n$\=/connected maps for all $-1\leq n\leq \infty$;
\item\label{morphismvslemmacotop:4} $\phi$ is a quotient reflecting \oo connected maps.
\end{propenum}
\end{proposition}

\begin{lemma}
\label{lem:cancel-epic}
Given two quotients of logoi $\phi:\cE\to \cF$ and $\psi:\cF\to \cG$, 
the composition $\psi\phi$ is epic if and only if both $\phi$ and $\psi$ are epic.
\end{lemma}
\begin{proof}
The condition is clearly sufficient by \cref{morphismvslemmacotop}.
Conversely, let us put $\cK=\cK_\phi$ and $\cK'=\cK_{\psi\phi}$. 
If $\psi\phi$ is epic, we have $\cK\subseteq \cK'\subseteq\Conn\infty$ in $\cE$.
This proves that $\phi$ is epic.
The quotient $\psi$ is generated by the congruence $\psi(\cK')\cong$.
The inclusions $\psi(\cK')\subseteq \psi(\Conn\infty) \subseteq \Conn\infty$ show that 
$\psi(\cK)\cong\subseteq \Conn\infty$ and that $\psi$ is epic.
\end{proof}

\begin{lemma}
\label{lem:transport-hyper-radical}
If $\phi:\cE\to \cF$ is an epic quotient, then $\phi(\Conn\infty)=\Conn\infty$ in $\cF$.
\end{lemma}
\begin{proof}
\Cref{morphismvslemmacotop:4} implies that $\Conn\infty=\phi^{-1}(\Conn\infty)$ in $\cE$.
Because $\phi$ is a quotient, its fully faithful right adjoint ensures that $\phi(\phi^{-1}(\Sigma))=\Sigma$ for any class of maps in $\cF$.
Thus $\phi(\Conn\infty)=\phi(\phi^{-1}(\Conn\infty))=\Conn\infty$.
\end{proof}

\begin{lemma}
\label{lem:hyper-reduction-epic}
If $\phi:\cE\to \cF$ is an epic quotient, the induced morphism between the hyper-reductions $\cE\quotient\Conn\infty\to \cF\quotient\Conn\infty$ is an equivalence.
\end{lemma}
\begin{proof}
Since $\phi:\cE\to\cF$ is epic, $\phi^{-1}(\Conn\infty) = \Conn\infty$ by  \cref{lem:transport-hyper-radical}.
This shows that the kernel of the quotient $\psi:\cE\to \cF\to \cF\quotient\Conn\infty$ is $\Conn\infty$ and thus $\cE\quotient\Conn\infty = \cF\quotient\Conn\infty$.
\end{proof}

\begin{table}[htbp]
\caption{Terminology for the quotient/image triple factorization \eqref{eq:triple-facto}	}
\label{table:facto}

\begin{center}	
\renewcommand{\arraystretch}{1.6}
\begin{tabularx}{.8\textwidth}{
|>{\centering\arraybackslash}X
|>{\centering\arraybackslash}X
|>{\centering\arraybackslash}X|
}
\hline
morphism of topoi $\tY\to \tX$ & morphism of logoi $\phi:\Sh\tX\to \Sh \tY$ & condition on $\phi$\\
\hline
inclusion & quotient & $\phi=\phi\quot$\\
\hline
ample inclusion & monogenic quotient & $\phi=\phi\mono$\\
\hline
thin inclusion & epic quotient & $\phi=\phi\epi$\\
\hline
surjection & conservative morphism & $\phi=\phi\cons$\\
\hline
weak surjection & weak conservative morphism & $\phi=\phi\wcons$\\
\hline
image factorization & quotient factorization & $\phi=\phi\cons\circ \phi\quot$\\
\hline
weak image factorization & weak quotient factorization & $\phi=\phi\wcons\circ \phi\mono$\\
\hline
thin--ample factorization of inclusions & monogenic--epic factorization of quotients & $\phi\quot=\phi\epi\circ \phi\mono$\\
\hline
\end{tabularx}
\end{center}
\end{table}

\begin{remark}[Epic quotients in algebraic geometry]
\label{rem:triple-facto}
We shall propose an analogue of the factorization \eqref{eq:triple-facto} in algebraic geometry.
Let $\phi:A\to B$ be a ring morphism with kernel $I$.
The image factorization $A\to A/I\to B$ is an obvious analogue of the image factorization $\phi:\cE\to \cE\quotient \cK_\phi\to \cF$.
Let us see now what could be the analogue of the monogenic--epic factorization of the quotient $\cE\to \cE\quotient \cK_\phi$.
Let $Z\subseteq X$ be the inclusion of schemes dual to $A\to A/I$.
As argued in \cref{rem:loc=mono}, an analogue of $\cE\quotient \cK\mono$ is the subscheme $Z'$ of $X$ defined as the intersection of all open subschemes containing $Z$.
This way, we get a factorization $Y\to Z\to Z'\to X$ into a dominant morphism onto a closed subscheme $Z$, followed by the inclusion of the closed subscheme into its `local neighborhood' $Z'$ and the inclusion of this neighborhood in the ambient scheme $X$.

Algebraically, $Z'$ correspond to the localization $A\to A[S^{-1}]$ with respect to the set $S\subseteq A$ of elements inverted by $A\to A/I$.
This provides a triple factorization $A\to A[S^{-1}]\to A/I\to B$.
The morphism $A[S^{-1}]\to A/I$ is the quotient by the image $I'$ of $I$ in $A[S^{-1}]$.
Using that $A[S^{-1}]\to A/I$ is conservative by construction, it is easy to see that $I'$ is inside the Jacobson radical of $A[S^{-1}]$.
Let us say that a ring quotient $B\to B/J$ is sub-Jacobson if $J$ is inside the Jacobson radical of $B$.
This is the analogue that we propose for epic quotients of logoi.
This fits well with the correspondance between the Jacobson radical and the congruence $\Conn\infty$ outlined already.
(An example of this triple factorization is given in \cref{rem:triple-facto-tower}.)
\end{remark}

\subsection{Pushout product, orthogonality, modalities}

In this section, we collect some technical results about the pushout product of maps in a logos, and we recall the definitions of modalities and fiberwise orthogonality from \cite{ABFJ:GC,ABFJ:HS}.
We refer to \cite[Section~5.2.8]{Lurie:HTT} or \cite[Section~3.1]{ABFJ:HS} for more details on factorization systems, and to \cite[Section~3.2]{ABFJ:HS} and \cite{RSS} for modalities.

\subsubsection{Pushout product and pullback hom}
\label{sec:pp-pbh}
\label{classical:boxmonoidal}
\label{sec:codiagonal}

The poset $[1] = \{0<1\}$ is monoidal for the $\sf inf$ structure.
A logos $\cE$ is in particular a cartesian closed category, thus the category $\Arr\cE:=\fun {[1]}\cE$ can be equipped with the closed symmetric monoidal Day convolution structure \cite{Glasman:Day,Lurie:HA}.
The resulting monoidal structure on $\Arr\cE$ is the {\it pushout product}.
For two maps $u:A\to B$ and $f:X \to Y$, it can be defined as the cocartesian gap map $u\pp f$ of the following commutative square:
\begin{equation}
\label[diagram]{pushout-prod1}
\begin{tikzcd}
A\times X \ar[r,"u\times X"] \ar[d,"A\times f"']
& B\times X \ar[d,"{i_2}"] \ar[ddr, bend left, "B\times f"]
\\
A\times Y \ar[r,"i_1"] \ar[drr, bend right,"u\times Y"]
& (A\times Y) \bigsqcup_{A\times X}  (B\times X)   \ar[dr,"{u\pp f}"']
\\
&& B\times Y\,.
\end{tikzcd}
\end{equation}
The unit of the pushout product is the canonical map $0\to 1$.
When the two maps have the terminal object as codomain, their pushout product reduces to the {\it join} product of the domains:
$(A\to 1)\pp(X\to 1) = A\join X\to 1$, where
$A\join X$ is the pushout of $A\ot A\times X\to X$.

The corresponding internal hom is the {\it pullback hom}.
For two maps $u:A\to B$ and $f:X \to Y$, it can be defined as the cartesian gap map $\pbh u f$ of the following commutative square:
\[
\begin{tikzcd}
\Map B X \ar[rrrd,"{\Map u X}", bend left=20] \ar[dddr,"{\Map B f}"', bend right] \ar[rd,"\pbh u f"] \\
&\Map B Y \times_{\Map A Y}\Map A X \ar[rr,"p_2"] \ar[dd,"p_1"'] && 
\Map A X \ar[dd,"{\Map A f}"] \\
\\
&\Map B Y \ar[rr,"{\Map u Y}"]
&& \Map A Y\,.
\end{tikzcd}
\]

\begin{proposition}[Day convolution]
\label{boxadjunction}
For any map $u$ in a logos $\cE$, the functor $u\pp-:\Arr\cE\to \Arr\cE$ is left adjoint to the functor $\pbh u -$:
Moreover, for every triple of maps $u,v,f\in \cE$, we have
\[
\pbh{ u\pp v} f\ =\ \pbh u {\pbh v f}\,.
\]
\end{proposition}

\medskip

Recall that the {\it codiagonal} $\nabla(u):B\sqcup_A B\to B$ of a map $u:A\to B$ is defined by the following diagram with a pushout square
\[
\begin{tikzcd}
A\ar[d, "{u}"'] \ar[r, "{u }"] &  B\ar[d] \ar[ddr, equal, bend left] &  \\
B \ar[r] 
\ar[drr, equal, bend right] & \pomark B\sqcup_A B   \ar[dr,"{\nabla(u)}"'] &   \\
&& B\,.
\end{tikzcd}
\]
The codiagonal of the map $A\to 1$ is the map $\Sigma A\to 1$, where $\Sigma A$ is the (unreduced) {\it suspension} of the object $A$.
By universality of colimits in $\cE$, the fibers of the codiagonal $\nabla(u):B\sqcup_A B\to B$ are the suspensions of the fibers of $A\to B$.
The $n$\=/spheres in $\cE$ are the iterated suspensions of $S^{-1}:=\emptyset$.
We have $S^0 := \Sigma \emptyset = 1+1$, and $S^{n+1} := \Sigma S^n$.
We denote $s^n$ the map $S^n\to 1$.
The diagram defining $\nabla(u)$ shows that $\nabla(u)=s^0\pp u$, where $s^0$ is the map $S^0\to 1$.
In particular, we have $s^n = s^0\pp \dots \pp s^0$ ($n+1$ factors).

\medskip
Recall that the diagonal of a map $X\to Y$ is the cartesian gap map of the following square
\[
\begin{tikzcd}
X \ar[dr, "\Delta(f)"] \ar[drr, equal, bend left] \ar[ddr, equal, bend right] && \\
& X\times_Y X\ar[r, "p_2"] \ar[d, "p_1"'] \pbmark & X\ar[d, "f"] \\
& X \ar[r, "f"']  & Y\,.
\end{tikzcd}
\]

For $K$ an object in a logos $\cE$, let $k$ be the map $K\to 1$. 
For any map $f:X\to Y$, we will say that the map $\pbh k f$ is the {\it $K$\=/diagonal} of $f$.
The diagram defining $\Delta(f)$ shows that $\Delta(u)=\pbh {s^0} f$, where $s^0$ is the map $S^0\to 1$.
We define the iterated diagonals of $f$ by $\Delta^{0}(f) = f$ and $\Delta^{n+1} f = \Delta (\Delta^n f)$.
The formula $s^n = s^0\pp \dots \pp s^0$ shows that $\Delta^{n+1} f = \pbh {s^n} f$.

\begin{remark}
\label{1boxadjunction} 
Specializing \cref{boxadjunction} with $u=s^0$, we get the classical adjunction $\nabla \dashv \Delta$ between the codiagonal and diagonal functors.
Moreover, we have
$\pbh{\nabla(u)} f
=
\pbh u {\Delta(f)}$
for every pair of maps $u,f\in \cE$.
\end{remark}

\subsubsection{Compatibility with base change}

We will say that a morphism $\alpha:f'\to f$ in the category $\Arr\cE$ is {\it cartesian} ({\it cocartesian}) 
and that $f'$ is a {\it base change} of $f$ ($f$ a {\it cobase change} of $f'$)
if the corresponding square in the category $\cE$ is cartesian (cocartesian):
\[
\begin{tikzcd}
X' \ar[r] \ar[d, "f'"'] & X \ar[d, "f"] \\
Y'  \ar[r]& Y\,.
\end{tikzcd}
\]

\begin{lemma}
\label{basechangebox}
If $f'\to f$ is a cartesian morphism, then, for every map $u$, the morphism $u\pp f'\to u\pp f$ is cartesian.
\end{lemma}
\begin{proof}
Let $u:A\to B$ be a map.
The map $u\pp f'\to u\pp f$ corresponds to a square computed from the cube (in solid lines)
\begin{equation}
\label{cube:pp}
\begin{tikzcd}
A\times X' \ar[rr] \ar[dd] \ar[rd]
&& A\times Y' \ar[dd] \ar[rd,dotted] \ar[rrrdd, bend left=10]
\\
&B\times X' \ar[rr, crossing over, dotted]
&&\bullet \ar[rrd,dashed, "u\pp f'"']
\\
A\times X \ar[rr] \ar[rd]
&& A\times Y  \ar[rd,dotted] \ar[rrrdd, bend left=10]
&&& B\times Y'  \ar[dd] 
\\
&B\times X  \ar[rr, dotted] \ar[from=uu, crossing over] \ar[rrrrd, bend right=5]
&&\bullet \ar[rrd,dashed, "u\pp f"']
\ar[from=uu, dashed, crossing over] \ar[from=2-2, to=3-6, crossing over, bend right=5]
\\
&&&&& B\times Y
\end{tikzcd}
\end{equation}
by taking the cocartesian gap maps on the top and bottom sides (dotted lines).
Notice that all the faces of the cube are cartesian squares.
In particular, the top square can be seen as the base change of the bottom one along $B\times Y'\to B\times Y$.
Since base change preserves pushouts, this shows that the cocartesian gap square (in dashed lines), which is the morphism $u'\pp f'\to u\pp f$, is cartesian.
\end{proof}

\begin{lemma}
\label{cobasechangebox}
If $f'\to f$ is a cocartesian morphism, then, for every map $u$, the morphism $u\pp f'\to u\pp f$ is cocartesian.
\end{lemma}
\begin{proof}
We use the cube \eqref{cube:pp}.
The front and back solid faces are now cocartesian.
We want to show that the dashed square is cocartesian.
The top and bottom dotted squares are cocartesian by construction.
Using these squares and the back square, we get that the front dotted square is cocartesian.
Using this and the front (solid) face, we then get that the dashed square is cocartesian.
\end{proof}

\begin{lemma}
\label{box-compo}
If $f:X\to Y$, $g:Y\to Z$ and $u:A\to B$, then 
\[
u\pp gf
\ =\ 
(u\pp g)(u\pp f)'\,,
\]
where $(u\pp f)'$ is a cobase change of $u\pp f$.
\end{lemma}

\begin{proof}
The result follows from the following diagram, where the three squares are pushouts:
\[
\begin{tikzcd}
A\times X \ar[rr, "{u\times X}"] \ar[d, "{A\times f}"'] && 
B\times X   \ar[d, "{i_2}"] \ar[drr, "{B\times f}"]  && && \\
A\times Y  \ar[rr, "{i_1}"] \ar[d, "{A\times g}"'] &&\bullet \pomarkk \ar[d, "{i_2}"]   \ar[rr, "{u\pp f}"]  && B\times Y  \ar[d, "{i_2}"]    \ar[drr, "{B\times g}"]  &&  \\
A\times Z 
\ar[rrrrrr, "{u\times Z}", bend right]   \ar[rr, "{i_1}"] &&  \bullet \pomarkk \ar[rr, "{(u\pp f)'}"] \ar[rrrr, "{u\pp gf}", bend right=20]   &&  \bullet \pomarkk  \ar[rr, "{u\pp g}"]  && B\times Z  \,.
\end{tikzcd}
\]
\end{proof}

\subsubsection{Pushout products and monomorphisms}

\begin{lemma}
\label{lem:cocartgapmapmono}
In a logos $\cE$, the cocartesian gap map of a pullback square of monomorphisms
\[
\begin{tikzcd}
W \ar[r] \ar[d] & V \ar[d] \\
U \ar[r]& X
\end{tikzcd}
\]
is a monomorphism.
\end{lemma}
\begin{proof}
We can assume $X=1$ by working in $\cE\slice X$.
Let us not assume that the square is a pullback yet, but only that all maps are monomorphisms (\ie that all objects are subterminal).
We let $Y:=U\sqcup_WV\to 1$ be the cocartesian gap map.
Using the universality of colimits, the computation of $Y\times Y$ expands in the colimit of the diagram
\[
\begin{tikzcd}
U\times V & W\times W \ar[r]\ar[l,dashed]& V=V\times V\\
U\times W\ar[u,dashed]\ar[d]&W\ar[r,equal]\ar[l,equal]\ar[u,equal]\ar[d,equal]&V\times W\ar[u]\ar[d,dashed]\\
U\times U=U&W\times W\ar[r,dashed]\ar[l]& V\times U
\end{tikzcd}
\]
When the diagram is cartesian (\ie when $W=U\times V$), the dashed arrows are isomorphisms and the colimit of diagram is $U\sqcup_WV=Y$.
This proves $Y^2=Y$ and that $Y\to 1$ is a monomorphism.
\end{proof}

Of course, this pushout is the union of $U$ and $V$ in the poset of subobjects of $X$.

\begin{lemma}
\label{boxprodmono}
In a logos, the pushout product of two monomorphisms is a monomorphism.
\end{lemma}

\begin{proof}
By definition, the pushout product $(U\to A)\pp (V\to B)$ of two monomorphisms $u:U\to A$ and $v:V\to B$ is the cocartesian gap map of the following cartesian square:
\[
\begin{tikzcd}
U\times V \ar[rr, "{u\times V}"] \ar[d, "{U\times v}"'] && 
A\times V \ar[d, "{A\times v}"] \\
 U\times B  \ar[rr, "{u\times B}"]&& A\times B\,.
\end{tikzcd}
\]
Since all maps are monomophisms, the result follows from \cref{lem:cocartgapmapmono}.
\end{proof}

\begin{lemma}
\label{lem:im-box-im}
If $m:A\to B$ is a monomorphism in a logos, then $m$ is a base change of $m\pp m$.
\end{lemma}
\begin{proof}
Consider the cube
\[
\begin{tikzcd}
A \ar[rr, equal] \ar[dd, "\Delta_A"'] \ar[rd, equal]
&& A \ar[dd, "A\times m"' near start] \ar[rd]
\\
& A \ar[rr, crossing over]
&& B  \ar[dd, "\Delta_B"]
\\
A\times A \ar[rr] \ar[rd]
&& A\times B  \ar[rd]
\\
&B\times A  \ar[rr] \ar[from=uu, crossing over, "m\times A"' near start]
&& B\times B\,.
\end{tikzcd}
\]
Let us see that all the faces are cartesian.
The bottom square is always cartesian.
The top square is cartesian because $m$ is a monomorphism.
The front face is cartesian because the right and outer squares of the diagram
\[
\begin{tikzcd}
A\ar[d,"m"']\ar[r,"\Delta_B"]&B\times A \ar[d] \ar[r,"p_2"]& A\ar[d,"m"]\\
B\ar[r,"\Delta_B"]&B\times B \ar[r,"p_2"] &B
\end{tikzcd}
\]
are cartesian. 
The argument is the same for the right side.
The cartesian nature of the top, right  and bottom sides proves that the left side is cartesian.
The cartesian nature of the top, front and bottom sides proves that the back side is cartesian.

Then by universality of pushouts, the pushouts of the top and bottom square define a cartesian square
\[
\begin{tikzcd}
A \ar[rr]\ar[d,"m"'] 
&&A\times B \coprod_{A\times A}B\times A \ar[d,"m\pp m"]\\
B\ar[rr, "\Delta_B"]
&&B\times B
\end{tikzcd}
\]
which proves the statement.
\end{proof}

For a class of maps $\Sigma$, we denote $\Sigma\bc$ its closure under base change.
\begin{lemma}
\label{lem:square-box-mono}
Let $\cM$ be a class of monomorphisms in a logos $\cE$ which is closed under base change, 
and satisfies Axiom~\ref{GT:axiom3} of \cref{def:GT} ($(vu\in \cM,u\in \Mono) \Ra v\in \cM$),
then $\cM = (\cM \pp \cM)\bc$.
\end{lemma}
\begin{proof}
We use the notations of Diagram~\eqref{pushout-prod1}
with $u:A\to B$ and $f:X\to Y$ in $\cM$.
The map $i_1$ is a monomorphism by cobase change (see \eg \cite[Proposition 2.2.6]{ABFJ:GBM})
and the map $u\times Y$ is in $\cM$ by base change. 
Then the map $u\pp f$ is in $\cM$ by Axiom~\ref{GT:axiom3}.
This proves that $\cM \pp \cM \subseteq\cM$ and that $(\cM \pp \cM)\bc\subseteq\cM$.
The converse inclusion is \cref{lem:im-box-im}.
\end{proof}

\begin{lemma}
\label{lem:box-surj}
If $u:A\to B$ is surjective then for any map $f:X\to Y$, the map $u\pp f$ is surjective.
\end{lemma}
\begin{proof}
We use the notations of Diagram~\eqref{pushout-prod1}.
Surjections are closed under base change, thus $u\times X$ and $u\times Y$ are surjections.
Surjections are also closed under cobase base change, thus the map $i_1$ is a surjection.
Then the map $u\pp f$ is a surjection by right cancellation.
\end{proof}

\begin{lemma}
\label{lem:image-box}
For any map $f$ and $g$ in a logos, we have $\im {f\pp g} = \im f \pp \im g$.
\end{lemma}
\begin{proof}
We apply \cref{box-compo} twice to the factorizations $g= \im g \coim g$ and $f= \im f \coim f$:
\begin{align*}
f\pp g  & = (f\pp \im g) \circ (f\pp \coim g')   \\ 
        & = (\im f\pp \im g) \circ (\coim f' \pp \im g) \circ (f\pp \coim g') \,.
\end{align*}
By \cref{lem:box-surj}, both maps $\coim f' \pp \im g$ and $f\pp \coim g'$ are surjective since one of the factors is.
Thus the composite map $k = (\coim f' \pp \im g) \circ (f\pp \coim g')$ is surjective.
The map $\im f\pp \im g$ is a monomorphism by \cref{boxprodmono}.
This proves that $f\pp g  = (\im f\pp \im g) \circ k$ is the surjection--mono factorization of $f\pp g$.
The lemma follows.
\end{proof}

\subsubsection{Factorization systems and modalities}
\label{sec:modality}

Recall that a map $u:A\to B$ in a category $\cE$ is said to be (left) \emph{orthogonal} 
to a map $f:X\to Y$ if their pullback hom  $\pbh u f$ is invertible.
We write $u\perp f$ and say 
that $u$ is \emph{left orthogonal} to $f$, 
and that $f$ is \emph{right orthogonal} to $u$.
If $\cA$ and $\cB$ are two classes of maps in a category $\cE$, we shall write $\cA\perp \cB$ to mean that we have $u\perp v$ for every $u\in \cA$ and $v\in \cB$.
We shall denote by $\rorth\cA$ (resp. $\lorth \cA$) the class of maps in $\cE$ that are right orthogonal (resp. left orthogonal) to every map in $\cA$.
We have
\[
\cA\subseteq \lorth\cB 
\quad \Leftrightarrow \quad
\cA\perp \cB 
\quad \Leftrightarrow  \quad
\rorth\cA \supseteq \cB\,.
\]

A pair $(\cA, \cB)$ of classes of maps in a category $\cE$ is said to be a \emph{factorization system} if the following conditions hold:
\begin{enumerate}[label=\roman*)]
\item\label{defFactSystem:2} $\cA\perp \cB$;
\item\label{defFactSystem:3} every map $f:X\to Y$ in $\cE$ admits a factorization $f=pu:X\to P(f)\to Y$ with $u\in \cA$ and $p\in \cB$.
\end{enumerate}
We refer to \cite[5.2.8]{Lurie:HTT} and \cite[Section 3.1]{ABFJ:HS} for a more detailed study of factorization systems. The latter reference also shows that the condition of closure under retracts of \cite[5.2.8.8]{Lurie:HTT} is not necessary.

If $(\cA, \cB)$ is a factorization system, then $\cB=\rorth \cA$ and $\cA=\lorth \cB$.
The class $\cA$ is said to be the \emph{left class} of the factorization system and the class $\cB$ to be the \emph{right class}. 
We shall say that the factorization in \ref{defFactSystem:3} is a {\it $(\cA,\cB)$\=/factorization} of the map $f:X\to Y$. 
The $(\cA, \cB)$\=/factorization of a map is unique (up to unique isomorphism), and the function $f\mapsto P(f)$ defines a functor $\Arr\cE\to \cE$.
For any object $E\in \cE$, we let $\cA\slice E$ and $\cB\slice E$ be inverse image of $\cA$ and $\cB$ by the forgetful functor $\cE\slice E\to\cE$.
The pair $(\cA\slice E,\cB\slice E)$ is a factorisation system in $\cE\slice E$.

\begin{definition}[Modality]
\label{defmodality} 
We shall say that a factorization system $(\cA, \cB)$ in a logos $\cE$ is a \emph{modality} if its left class $\cA$ is closed under base change. 
A modality $(\cA,\cB)$ is {\it left-exact}, or simply {\it lex}, if the functor $P:\Arr\cE\to \cE$ is left-exact.
\end{definition}

Since the right class of a factorization system is always closed under base change, a factorization system is a modality if and only if both classes are closed under base change, if and only if the factorization is preserved by base change.

\begin{remark}
Modalities were introduced in the context of homotopy type theory in \cite[7.7]{hottbook} and various equivalent definitions are studied in \cite{RSS}.
Notice that these characterization are all given in the internal language of the ambient logos $\cE$, whereas our definition of modality in terms of stable factorization system is external.
\end{remark}

\begin{examples}
\label{examplemodality}
\begin{exmpenum}
\item\label{examplemodality:1}
    The factorization system $(\Surj,\Mono)$ of surjections and monomorphisms is a modality \cite[Exemple 3.2.12~(b)]{ABFJ:HS}.
    For a map $f:A\to B$, we shall denote its factorization by $f={\coim f}\circ {\im f}:A\to \Im(f)\to B$ where ${\coim f}:A\to \Im(f)$ is a surjection and ${\im f}:\Im(f)\to B$ a monomorphism.

\item\label{examplemodality:2} 
    The factorization system system $(\Conn n,\Trunc n)$ of $n$\=/connected maps and $n$\=/truncated maps is a modality for every $n \geq -1$ \cite[Example 3.4.2]{ABFJ:GBM}.
    Notice that $(\Conn {-1},\Trunc {-1})=(\Surj,\Mono)$.

{\bf Warning.} An $n$\=/connected map in our sense is $(n+1)$\=/connected in the conventional topological indexing and is called $(n+1)$\=/connective in \cite{Lurie:HTT}.

\item\label{examplemodality:3} 
    If $\phi:\cE\to \cF$ is a morphism of logoi, the pair $(\cK_\phi,\cK_\phi^\perp)$ is a left-exact modality in $\cE$ \cite[Proposition 3.2.3]{ABFJ:HS}.

\item\label{examplemodality:4}\label{inducedmodality}
If $(\cA,\cB)$ is a modality in a logos $\cE$, then so is the factorisation system $(\cA\slice E,\cB\slice E)$ in the category $\cE\slice E$ for every object $E\in \cE$.
Moreover, if $(\cA,\cB)$ is a lex modality, so are the $(\cA\slice E,\cB\slice E)$.

\item\label{examplemodality:acyclic} 
Recall from \cite{Raptis:acyclic,Hoyois:acyclic} that a map in a logos $\cE$ is called {\it acyclic} if its codiagonal $\nabla f$ is invertible (\ie if it is an epimorphism in $\cE$).
We denote by $\Acmap\cE$ the class of acyclic maps in a logos $\cE$.
It is proved in \cite[Theorem 3.3]{Raptis:acyclic} and \cite[Lemma 1]{Hoyois:acyclic} that $\Acmap\cE$ is the left class of a modality.
The factorization of a map $X\to Y$ is Quillen plus construction in the logos $\cE\slice Y$.

\end{exmpenum}
\end{examples}

\begin{remark}
The left classes of factorization systems are saturated classes.
On can show that a modality is left-exact if and only if its left class is a congruence \cite[Proposition 4.2.6]{ABFJ:HS}.
This motivated us to axiomatize left classes of modalities (saturated classes closed under base change) as an intermediate between saturated classes and congruences.
This is the definition of our notion of acyclic class (see \cref{def:acyclic}).
\end{remark}

\begin{remark}[Modality as truncation]
\label{rem:t-structure}
Given a factorization system $(\cA,\cB)$ on a logos $\cE$, we denote by $\cB_X\subseteq \cE\slice X$ the full subcategory spanned by maps $Y\to X$ in $\cB$.
The factorization of a map $Y\to X$ provides a reflection $\cE\slice X\to \cB_X$.
For any map $f:X'\to X$, the base change $f^*:\cE\slice X\to \cE\slice {X'}$ sends $\cB_X$ to $\cB_{X'}$ since $\cB$ is stable under base change.
Then $(\cA,\cB)$ is a modality precisely when $f^*$ commute with these reflections.
This presents a modality as a system of truncation operations defined uniformly on all slices.
In \cref{examplemodality:2}, this recovers the Postnikov truncations.
\end{remark}

\subsubsection{Fiberwise orthogonality and the fiberwise diagonal}

\begin{definition}[Fiberwise orthogonality]
\label{fperp} 
We shall say that a map $u:A\to B$ in a logos is \emph{fiberwise left orthogonal} to a map $f:X\to Y$, and write $u\fperp f$ (and say that $f$ is \emph{fiberwise right orthogonal} to $u$) if every base change $u'$ of $u$ is left orthogonal to $f$.
\end{definition}

If $\cA$ and $\cB$ are two classes of maps in a logos, we shall write $\cA\fperp \cB$ to mean that we have $u\fperp f$ for every $u\in \cA$ and $f\in \cB$.
We shall denote by $\rforth\cA$ (resp. $\lforth\cA$) the class of maps in $\cE$ that are fiberwise right orthogonal (resp.  fiberwise left orthogonal) to every map in $\cA$.
We have
\[
\cA\subseteq \lforth\cB 
\quad \Leftrightarrow \quad
\cA\fperp \cB
\quad \Leftrightarrow
\quad \rforth\cA \supseteq \cB\,.
\]
A factorization system $(\cL, \cR)$ in a logos is a modality if and only if $\cL \fperp \cR$, in which case $\cR=\rforth\cL$ and $\cL=\lforth\cR$.
For any set of maps $\Sigma$ in a logos, the pair $(\lforth(\rforth\Sigma),\rforth\Sigma)$ is a modality, said to be generated by $\Sigma$ \cite[Theorem 3.2.20]{ABFJ:HS}.
Notice that, if $(\cL,\cR)$ is already known to be a modality, it is generated by $\Sigma$ if and only if $\rforth\Sigma=\cR$ (since we already know that $\cL=\lforth\cR$).

\medskip
We now recall a few things about the {\it fiberwise diagonal} introduced in \cite[Section 2.5]{ABFJ:GC}.
For two maps $u:A\to B$ and $f:X\to Y$ in a logos $\cE$ the relation $u\perp f$ holds if and only if $\pbh u f$ is invertible.
The fiberwise diagonal is a map $\magic u f$ which is invertible if and only if the stronger relation $u\fperp f$ holds.
It is constructed as follows.
We start by pulling back $u$ and $f$ over the same base $B\times Y$,
and we look at them as defining two objects $(A_Y,u_Y)$ and $(X_B,f_B)$ in the slice logos $\cE\slice {B\times Y}$.
Then we compute the pullback hom between them $\pbh {(A_Y,u_Y)}{(X_B,f_B)}$, which is the following diagonal map in $\cE\slice {B\times Y}$
\[
\magic u f :=(X_B,f_B)\stto (X_B,f_B)^{(A_Y,u_Y)}\,.
\]
Then we look at this map back in $\cE$ using the forgetful functor $\cE\slice {B\times Y}\to \cE$.

\medskip
The computation of the fiberwise diagonal 
$\magic u f$ becomes easy when the map $u$ is essentially an object.
\begin{lemma}
\label{lem:magic}
For two maps $f$ and $u=A\to 1$, we have a natural isomorphism 
$\magic u f = \pbh u f$.
\end{lemma}
\begin{proof}
When $B=1$, the construction of the fiberwise diagonal simplifies into $(X,v)\to (X,v)^A$.
The exponential by $A\in \cE$ in $\cE\slice B$ is computed as the fiber product $X^A\times_{Y^A}Y$.
The fiberwise diagonal becomes $X \to X^A\times_{Y^A}Y$ which is exactly $\pbh u v$.
\end{proof}

\begin{proposition}[{\cite[Propositions 2.5.3 and 2.5.4]{ABFJ:GC}}]
\label{mathieuslemma}
If $u$, $v$ and $f$ are three maps in a logos, then
\begin{propenum}
\item\label{mathieuslemma:1} the map $\magic u f$ is invertible if and only if $u\fperp f$;
\item\label{mathieuslemma:2} there exists a natural isomorphism
$
\magic {u\pp v} f
\ =\ 
\magic u {\magic v f}
$.
\end{propenum}
\end{proposition}

The following corollary strengthens \cref{1boxadjunction}.

\begin{corollary}[Codiagonal--diagonal adjunction]
\label{cor:codiag-diag-ortho}
If $u$ and $f$ are two maps in a logos, then
\[
\nabla u \fperp f
\quad\Leftrightarrow\quad
u\fperp \Delta f\,.
\]
\end{corollary}
\begin{proof}
Recall from \cref{sec:pp-pbh} that $\nabla u=s^0\pp u$ and $\Delta f=\pbh {s^0} f$.
By \cref{lem:magic} and \cref{mathieuslemma:2}, we have 
$\magic {s^0\pp u} f = \magic u {\pbh{s^0} f}$,
so one map is invertible if and only if the other is.
The equivalence of the statement follows from \cref{mathieuslemma:1}.
\end{proof}

\begin{examples}
\label{examplefperp}
\begin{exmpenum}
\item\label{examplefperp:-2}
The modality $(\Iso,\All)$ is generated by the map $1\xto{id} 1$:
The class of all base changes of $1\xto{id} 1$ is the class $\Iso$ of all isomorphisms, thus $\{1\xto{id} 1\}^{\fperp}=\Iso^\perp = \All$.

\item\label{ex:acyclicgene:-1} 
The modality $(\All,\Iso)$ is generated by the single map $0\to 1$.
The base changes of the map $0\to 1$ form the class $\{0\to A\,|\,A\in \cE\}$.
A map $f:X\to Y$ is orthogonal to $0\to A$ if and only if $\Map AX\simeq \Map AY$.
If this is true for every $A$, then $f$ is an isomorphism, and
$\{0\to 1\}^{\fperp}=\{0\to A\,|\,A\in \cE\}^\perp = \Iso$.

\item\label{examplefperp:0} 
The modality $(\Surj,\Mono)$ is generated by the single map $s^0:S^0\to 1$ where $S^0 = 1\sqcup 1$.
Using \cref{lem:magic}, we get that a map $f$ is a mono if and only if $\Delta f = \pbh {s^0} f = \magic {s^0} f$ is invertible if and only if $s^0\fperp f$.
This shows that $\Mono = \rforth{\{s^0\}}$.

\item\label{examplefperp:n} 
For $-1\leq n<\infty$, let $S^n$ be the $n$\=/sphere in $\cS$. 
We shall denote also $S^n$ its image by the canonical logos morphism $\cS\to \cE$.
The modality $(\Conn n, \Trunc n)$ is generated by the single map $s^{n+1}:S^{n+1}\to 1$ \cite[Lemma 2.2.22]{ABFJ:GT}.
Recall that a map $f$ is $n$\=/truncated if $\Delta^{n+2} f = \pbh {s^{n+1}} f$ is invertible.
Using \cref{lem:magic}, this is equivalent to $\magic {s^{n+1}} f$ being invertible, and thus to $s^{n+1}\fperp f$.
This shows that $\Trunc n = \rforth{ \{s^{n+1}\} }$.

\item
The right orthogonal to the class $\Acmap\cE$ of acyclic maps from \cref{examplemodality:acyclic} is the class $\Hypoab\cE$ of hypoabelian morphism \cite[Corollary 10]{Hoyois:acyclic}.
Since $\Acmap\cE$ is closed under base change, we have in fact $\Hypoab\cE = \Acmap\cE^\fperp$.

\item\label{examplefperp:Sullivan}(Sullivan modality)
In the category $\cS$, let $\Sigma$ be the class of maps $A\to 1$ where $A$ is a connected $\pi$\=/finite space.
Since the set of isomorphisms classes of such spaces is countable \cite[Lemma 4.1.1]{Anel:pi-finite}, we can assume that $\Sigma$ is a countable set.
Therefore, it generates a modality $(\overline\Sigma,\rforth\Sigma)$.
We do not know a concrete description of the class $\mathsf{Sull}:=\overline\Sigma$, nor of its orthogonal, nor of the factorization of a map.
But, from the Sullivan conjecture, the class $\rforth\Sigma$ contains all maps whose fibers are finite spaces, so it is non-trivial.
In particular, $\Mono\subsetneq\rforth\Sigma$, but there are strictly more maps.
Consequently, the inclusion $\overline\Sigma\subsetneq\Surj$ is also strict.

\end{exmpenum}
\end{examples}

\subsection{Acyclic classes, decalage, suspension}
\label{sec:acyclic}

We recall the notion of acyclic class from \cite{ABFJ:HS,ABFJ:GT}.

\begin{definition}[Acyclic class {\cite[Definition 3.2.8]{ABFJ:HS}}]
\label{def:acyclic}
We say that a class of maps $\cA$ in a logos $\cE$ is {\it acyclic} if the following conditions hold:
\begin{defenum}
\item\label{def:acyclic:1} the class $\cA$ contains the isomorphisms and is closed under composition;
\item\label{def:acyclic:3} the class ${\cA}$ is closed under base change;
\item\label{def:acyclic:2} the class ${\cA}$ is closed under colimits (in the arrow category of $\cE$).
\end{defenum} 
\end{definition}

Equivalently, a class of maps $\cA$ is acyclic if and only if it is closed under base change and a saturated class in the sense of \cite[Definition 5.5.5.1]{Lurie:HTT}.
In particular, every acyclic class is closed under cobase change and has the right cancellation property.
Every acyclic class is also a local class \cite[Proposition 2.2.35]{ABFJ:GT}.

\begin{examples}
\label{ex:acyclic} 

\begin{exmpenum}
\item\label{ex:acyclic:1} 
The classes $\Iso$ and $\All$ of isomorphisms and all maps in a logos $\cE$ are respectively the smallest and the largest acyclic classes (for the inclusion relation).

\item\label{ex:acyclic:2}
Any congruence is an acyclic class.
We will recall in \cref{lem:caraccong} a condition for when an acyclic class is a congruence.

\item\label{ex:acyclic:3bis}
The left class of a modality is always acyclic.
In fact, by \cite[Proposition 3.2.14]{ABFJ:HS}, the class  $\lforth\cM$ is acyclic for any class of maps $\cM$ in a logos $\cE$.

\item\label{ex:acyclic:3}
In particular, the classes $\Surj$ and all the $\Conn n$ (for $-1\leq n\leq \infty$) are acyclic since they are the left classes of some modalities \cite[Examples 3.2.12~(b) and (c)]{ABFJ:HS}. 

\item\label{ex:acyclic:3ter}
For the same reason, the classes $\Acmap\cE$ of acyclic maps \cref{examplemodality:acyclic} and the ``Sullivan acyclic class" $\mathsf{Sull}$ of \cref{examplefperp:Sullivan} are also acyclic.

\item\label{ex:acyclic:4} 
Any morphism of logoi $\phi:\cE\to \cF$ preserves isomorphisms, composition, colimits and pullbacks, therefore the class $\phi^{-1}(\cA) =\{f\in \cE\ |\ \phi(f)\in \cA\}$ is an acyclic class in $\cE$, for any acyclic class $\cA\subseteq \cF$.
In particular, the class $\phi^{-1}(\Surj)$ of maps sent to surjections by $\phi$ is acyclic. 
More generally, $\phi^{-1}(\Conn n)$ is acyclic for $-1\leq n\leq \infty$.

\end{exmpenum}
\end{examples}

Any intersection of acyclic classes is acyclic.
Every class of maps $\Sigma$ in a logos $ \cE$ is contained in a smallest acyclic class $\Sigma\ac$ called the acyclic class {\it generated} by $\Sigma$.
We shall say that an acyclic class $\cA$ is of small generation if $\cA=\Sigma\ac$ for a set of maps $\Sigma\subseteq \cA$.
It is not known if every acyclic class is of small generation.

Given a logos $\cE$, we shall say that a class of objects $\cG$ of $\cE$ is a {\it class of generators}, if every object of $\cE$ is the colimit of a small diagram of objects in $\cG$.
The class of all objects is always a class of generators.
We recall a description $\Sigma\ac$ in terms of saturated classes.

\begin{lemma}[{\cite[Corollary 3.2.19]{ABFJ:HS}}]
\label{lem:generation-acyclic-class}
Let $\cG$ be a class of generators in a logos $\cE$.
For $\Sigma$ a class of maps in $\cE$, 
let $\Sigma\bc$ be the class of all its base changes over the objects of $\cG$, then 
$\Sigma\ac = (\Sigma\bc)\sat$,
where $(-)\sat = \lorth{(\rorth{(-)})}$ is the saturation.
\end{lemma}

Recall the notion of fiberwise orthogonality from \cref{fperp}.

\begin{lemma}
\label{lem:acyclicsat}
For any class of maps $\Sigma$, we have 
$\Sigma\ac = \lrforth\Sigma = \lorth{(\rforth\Sigma)}$
and $\rforth\Sigma = \rforth{(\Sigma\ac)} = \rorth{(\Sigma\ac)}$.
\end{lemma}
\begin{proof}
The equality $\Sigma\ac = \lrforth\Sigma$ is from \cite[Lemma~3.2.15]{ABFJ:HS} and we shall not reprove it here.
Let us show the others.
Using the previous formula, we have $\rforth{(\Sigma\ac)} = \rforth{\left(\lrforth\Sigma\right)} = \rforth\Sigma$, where the last equality is true for any orthogonality relation.
To show $\Sigma\ac = \lorth{(\rforth\Sigma)}$, we will use the formula $\Sigma\ac = (\Sigma\bc)\sat$ from \cref{lem:generation-acyclic-class} (with the class of all objects are generators).
By definition of the fiberwise othorgonality, we have $\rforth\Tau = \rorth{(\Tau\bc)}$, for any class of maps $\Tau$. 
Thus, $\lorth{(\rforth\Sigma)} = \lorth{(\rorth{(\Sigma\bc)})} = (\Sigma\bc)\sat = \Sigma\ac$.
If $\Tau$ is closed under base change, we have $\rforth\Tau = \rorth\Tau$.
Applied to $\Tau=\Sigma\ac$, this shows $\rforth{(\Sigma\ac)} = \rorth{(\Sigma\ac)}$.
\end{proof}

The following result is \cite[Theorem 3.2.20]{ABFJ:HS}.
\begin{proposition}[Acyclic classes and modalities]
\label{prop:acyclic2modality}
Any acyclic class of small generation $\cA=\Sigma\ac$ is the left class of the modality $(\cA,\rforth\cA)=(\Sigma\ac,\rforth\Sigma)$.
\end{proposition}

The proof of the following lemma is left to the reader.

\begin{lemma}
\label{inducedacyclic}
If $\cA$ is an acyclic class in a logos $\cE$, then so is the class $\cA\slice{B}$ in the logos $\cE\slice{B}$ for any object $B\in \cE$.
By definition, a map $f:(X,p)\to (Y,q)$ in $\cE\slice{B}$ belongs to $\cA\slice{B}$ if and only if the map $f:X\to Y$ belongs to $ \cA$.
\end{lemma}

\begin{lemma} \label{smallgenpres}
An acyclic class $\cA$ in a logos $\cE$ is of small generation 
if and only if the full subcategory $\underline{\cA}\subseteq\Arr\cE$ spanned by the maps in $\cA$ is presentable.
\end{lemma}

\begin{proof} 
Every acyclic class is saturated.
Let us denote by $\Sigma\sat$ the saturated class generated by a set $\Sigma$ of maps in $\cE$ \cite[Definition 3.1.13]{ABFJ:HS}. 
By definition, a saturated class $\cA$ of maps is of small generation if $\cA=\Sigma\sat$ for a set $\Sigma$.
An acyclic class is of small generation as an acyclic class if and only if it is of small generation as a saturated class:
the inclusion $\Sigma\sat\subseteq \Sigma\ac$ is an equality if $\Sigma\sat$ is acyclic,
and, conversely, we use $\Sigma\ac =(\Sigma\bc)\sat$ from \cref{lem:generation-acyclic-class}. 
But a saturated class $\cA$ is of small generation as a saturated class if and only if 
the corresponding full subcategory $\underline{\cA}\subseteq \Arr\cE$ is presentable by \cite[Proposition 5.5.5.9]{Lurie:HTT}.
Thus, the acyclic class $\cA$ is of small generation (as an acyclic class) if and only if the category $\underline{\cA}$ is presentable.
\end{proof}

\begin{proposition}
\label{smallintersmallgen}
The intersection of a set of acyclic classes of small generation is of small generation.
\end{proposition}

\begin{proof} 
Let $\{\cA(i) \ |\ i\in I\}$ be a small family of acyclic classes of small generation in a logos. 
For each $i\in I$ the category $\underline{\cA(i)}$ is presentable by \Cref{smallgenpres}.
By~\cite[Definition 5.5.0.1]{Lurie:HTT} a category is presentable if and only if it is cocomplete and accessible.
According to~\cite[Proposition 5.4.7.10]{Lurie:HTT} the category $\underline{\cA}=\bigcap_{i\in I} \underline{\cA(i)}$ is accessible.
It is also cocomplete, hence presentable. So \Cref{smallgenpres} yields the result.
\end{proof}

We denote by $\Acyclic\cE$ the poset of acyclic classes in $\cE$.
It is a (large) suplattice and the inclusion $\Cong\cE\to \Acyclic\cE$ is a morphism of suplattices.

\begin{remark}
\label{rem:nontrivial-Acy(S)}
The poset $\Cong\cS$ has only two elements, but the poset $\Acyclic\cS$ has much more.
For example, all classes $\Conn n$ ($-1\leq n<\infty$) are acyclic, as well as the class of acyclic maps \cref{examplemodality:acyclic} and the ``Sullivan acyclic class" of \cref{examplefperp:Sullivan}.
More examples are considered in the work of 
Bousfield \cite{bousfield1994localization}, 
Farjoun \cite{farjoun1995cellular}, or
Chach\'olski \cite{Chach:inequ}.
See also \cref{ex:tower:suspension}
\end{remark}

\begin{proposition}[Transport of acyclic classes]
\label{prop:transport-acyclic}
Let $\phi:\cE\to \cF$ be a morphism of logoi, with associated congruence $\cK_\phi := \phi^{-1}(\Iso)$.
\begin{propenum}
\item\label{prop:transport-acyclic:0}
For any acyclic class $\cB$ in $\cF$ its inverse image $\phi^{-1}(\cB)$ is an acyclic class in $\cE$ containing $\cK_\phi$.

\item\label{prop:transport-acyclic:image-ac}
There exists an adjunction 
\begin{align*}
\phi(-)\ac:\Acyclic\cE	&\adjunction \Acyclic\cF : \phi^{-1}\\
\cA &\mto\phi(\cA)\ac\\
\phi^{-1}(\cA) &\mot \cA
\end{align*}
where the left adjoint is a morphism of suplattices.
\item\label{prop:transport-acyclic:image-ac-ac} 
For any class of maps $\Sigma$ in $\cE$, we have
\[
\phi(\Sigma\ac)\ac=\phi(\Sigma)\ac\,.
\]
And, if $\psi:\cF\to \cG$ is another morphism of logoi, then $(\psi\phi)(-)\ac = \psi(\phi(-)\ac)\ac$.

\item\label{prop:transport-acyclic:quotient}
When $\phi$ is a quotient, $\phi^{-1}$ induces an isomorphism between the poset $\Acyclic\cF$ of acyclic classes in $\cF$ and the poset $\Acyclic\cE\upslice {\cK_\phi}$ of acyclic classes in $\cE$ containing $\cK_\phi$.
\end{propenum}
\end{proposition}
\begin{proof}
\noindent(1) 
We already saw in \cref{ex:acyclic:4} that the inverse image of an acyclic class is acyclic.
The inclusion $\cK_\phi\subseteq\phi^{-1}(\cB)$ comes from $\Iso\subseteq \cB$.

\smallskip
\noindent(2) 
If $\cA\in \Acyclic\cE$ and $\cB\in \Acyclic\cF$, we have equivalences
\[
\phi(\cA)\ac\,\subseteq\,\cB
\quad\iff\quad
\phi(\cA)\,\subseteq\,\cB
\quad\iff\quad
\cA\,\subseteq\,\phi^{-1}(\cB)\,.
\]
This shows that $\phi(-)\ac$ is left adjoint to the map $\phi^{-1}$.	
Thus, it preserves arbitrary suprema and is a morphism of suplattices.

\smallskip
\noindent(3) 
Let us see that $\phi(\Sigma\ac)\ac=\phi(\Sigma)\ac$ for any class of maps $\Sigma$ in $\cE$.
The inclusion $\phi(\Sigma)\ac\subseteq \phi(\Sigma\ac)\ac$ is clear, since $\Sigma\subseteq \Sigma\ac$.
Conversely, we have $\Sigma \subseteq \phi^{-1}(\phi(\Sigma)\ac)$, since $\phi(\Sigma)\subseteq \phi(\Sigma)\ac$.
The class $\phi^{-1}(\phi(\Sigma)\ac)$ is acyclic, 
thus $\Sigma\ac \subseteq \phi^{-1}(\phi(\Sigma)\ac)$, 
and hence $\phi(\Sigma\ac)\subseteq \phi(\Sigma)\ac$.
This proves $\phi(\Sigma)\ac=\phi(\Sigma\ac)\ac$.
The last statement is a direct consequence, since
$
\psi(\phi(\cA)\ac)\ac
=
\psi(\phi(\cA))\ac
=
(\psi\phi)(\cA)\ac
$.

\smallskip
\noindent(3) 
is proven in \cite[Proposition 2.2.32]{ABFJ:GT}.
\end{proof}

\begin{corollary}
\label{cor:facto}
The factorization $\cE\to \cE\quotient \cK_\phi\to \cF$ of $\phi:\cE\to \cF$ induces morphisms of suplattices
\[
\Acyclic\cE
\stto
\Acyclic\cE\upslice{\cK_\phi}
\stto
\Acyclic\cF
\,,
\]
where the first morphism is surjective with an injective right adjoint,
and the second morphism has a right adjoint preserving the minimal element.
\end{corollary}
\begin{proof}
This is a direct application of \cref{prop:transport-acyclic}. 
We need only to show the part about the minimal element, 
but this is just saying that $\phi^{-1}(\Iso)=\cK_\phi$ is minimal in $\Acyclic\cE\upslice{\cK_\phi}$.
\end{proof}

\subsubsection{Image--coimage decomposition of acyclic classes}
\label{sec:image}

Recall that every map $f:A\to B$ in a logos admits a unique factorization $f={\coim f}\circ {\im f}:A\to \Im(f)\to B$ where ${\coim f}:A\to \Im(f)$ is a surjection and ${\im f}:\Im(f)\to B$ a monomorphism (see \cite[6.2.3]{Lurie:HTT}, \cite[Lecture 4]{Rezk:Leeds}, \cite[2.2.3]{ABFJ:GT}).

\begin{lemma}[{\cite[2.2.36]{ABFJ:GT}}]
\label{lem:acyclic-image}
Let $\cA$ be an acyclic class in a logos $\cE$. 
Then a map $f:A\to B$ in $\cE$ belongs to $\cA$ if and only if both maps $\im f$ and $\coim f$ belong to $\cA$.
\end{lemma}

\medskip

For a class of maps $\Sigma$, we define 
\[
\im\Sigma
\ :=\ 
\big\{\im f \,|\, f\in \Sigma\big\}
\qquad\text{and}\qquad
\coim\Sigma
\ :=\ 
\big\{\coim f \,|\, f\in \Sigma\big\}\,.
\]
By \cite[Lemma 2.2.37]{ABFJ:GT}, for any acyclic class $\cA$, we have
\[
\im\cA = \cA\cap \Mono
\qquad\text{and}\qquad
\coim\cA = \cA\cap \Surj\,.
\]
The class $\coim \cA= \cA\cap \Surj$ is always acyclic since it is an intersection of two acyclic classes.
It is called the {\it epigenic part} of $\cA$.
An acyclic class is called {\it epigenic} if $\cA = \coim\cA$. 
Equivalently, $\cA$ is epigenic if and only if $\cA\subseteq \Surj$, if and only if there exists a class $\Sigma$ of surjections such that $\cA = \Sigma\ac$.
The class $\im \cA = \cA \cap \Mono$ however is not acyclic in general, 
it is an extended Grothendieck topology (\cref{def:GT}), see \cite[Lemma 3.1.8]{ABFJ:GT}.
Its acyclic completion $\cA\mono:=\im \cA\ac$ is called the {\it monogenic part} of $\cA$.
An acyclic class is called {\it monogenic} if $\cA = \cA\mono$ or, equivalently, if there exists a class $\Sigma$ of monomorphism such that $\cA = \Sigma\ac$.

\begin{proposition}[{\cite[2.2.40]{ABFJ:GT}}]
\label{prop:imcoimdecomp}
The following relation holds in the poset $\Acyclic\cE$:
\[
\cA
\ =\ 
\im \cA\ac
\,\vee\,
\coim \cA\,,
\]
where $\vee$ is the supremum in $\Acyclic\cE$.
\end{proposition}

\begin{remark}
The decomposition of \cref{prop:imcoimdecomp} can be used to describe $\Acyclic\cE$ as the the subposet of $\MAcyclic\cE \times \EAcyclic\cE$ made of pairs $(\cA',\cA'')$ where $\coim{\cA'}\subseteq \cA''$.
More generally, such decompositions can de done relatively to any modality $(\cL,\cR)$:
$\cA = (\cA\cap \cR)\ac\, \vee\,\cA\cap \cL$.
However, the acyclic class $(\cA\cap \cR)\ac$ is not a congruence in general.
\end{remark}

\begin{proposition}[{\cite[Proposition 4.1.14]{ABFJ:GT}}]
\label{prop:mono-part}
For any class of maps $\Sigma$ in a logos, we have
\[
\im{\Sigma\ac}\ac = {\im \Sigma}\ac\,.
\]
In particular, if $\cA$ is of small generation, then so is $\cA\mono=\im\cA\ac$.
\end{proposition}

\subsubsection{Decalage of acyclic classes}
\label{sec:decalage}

The decalage of acyclic classes has been introduced in \cite{ABFJ:HS} where, for a modality $(\cL,\cR)$, it was called the class of {\it fiberwise $\cR$\=/equivalences}.
It was further studied in \cite{ABFJ:GT}.

\begin{definition}[Decalage]
\label{def:decalage}
For a class of maps $\cA$, we define the {\it decalage} of $\cA$:
\[
\dec \cA
\ :=\ 
\{\, f\in \cE \ |\  f\in \cA,\, \Delta f \in \cA \,\}\,,
\]
and, by induction,
\[
\decn 0 \cA
\ :=\ 
\cA \,,
\qquad\text{and}\qquad
\decn {n+1} \cA
\ :=\ 
\dec {\decn n \cA}
\ =\ 
\{\, f\in \cE \ |\  \forall\,0\leq k\leq n+1,\, \Delta^k f \in \cA\,\}\,.
\]
We also put
\[
\decinfty \cA
\ :=\ 
\bigcap_n \decn n \cA
\ =\ 
\{\, f\in \cE\ |\  \forall\, n\geq 0\,, \Delta^n f \in \cA\,\}\,.
\]
\end{definition}

\begin{examples} 
\label{ex:decalage}
We have the following examples of the constructions $\decname$:
\begin{exmpenum}
\item\label{ex:decalage:1} $\dec \Surj = \Conn 0$;
\item\label{ex:decalage:2} $\dec {\Conn n} = \Conn {n+1}$ and $\decn {n+1} \Surj = \Conn n$;
\item\label{ex:decalage:3} $\decinfty \Surj = \decinfty {\Conn n} = \Conn\infty$.
\item\label{ex:decalage:4} The acyclic class $\Acmap\cE$ of \cref{examplemodality:acyclic} satisfies ${\dec{\Acmap\cE}=\Iso}$. This can be seen as a consequence of \cref{examplesusp:3} and \cref{corollarydescendantimcoimdecomp}. Moreover, $\Acmap\cE$ is the largest class of maps in $\cE$ with this property.

\end{exmpenum}
\end{examples}

The following result justifies the importance of the construction $\dec \cA$.
\begin{lemma}[{\cite[Lemma 2.2.46 and Proposition 2.2.53]{ABFJ:GT}}]
\label{lem:decalage}
\label{lem:racine-acyclic}
Let $\cA$ be an acyclic class.
\begin{lemenum}
\item\label{lem:decalage:1} The class $\dec \cA$ is acyclic.
\item\label{lem:decalage:0} The class $\cA$ is a congruence if and only if $\dec \cA = \cA$ if and only if $\cA \subseteq \dec \cA$.
\item\label{lem:decalage:2} The class $\decinfty \cA$ is a congruence. It is the largest congruence contained in $\cA$.
\item We have $\decinfty\cA\cap \Mono = \cA\cap \Mono$.
\end{lemenum}
\end{lemma}

The following theorem compares the notion of acyclic classes and congruences.

\begin{theorem}[{\cite[Theorem 2.2.51]{ABFJ:GT}}]
\label{thm:adjCongAcyclic}

The inclusion of congruences in acyclic classes admits both a left and a right adjoint
\[
\begin{tikzcd}
\Cong\cE \ar[rr, hook]
\ar[from=rr, shift left = 3,"{\decinfty-}"]
\ar[from=rr, shift right = 3,"{(-)\cong}"']
&& \Acyclic\cE\,.
\end{tikzcd}
\]
The left adjoint is given by the congruence completion $\cA\mapsto \cA\cong$.
The right adjoint is given by the map $\cA\mapsto \decinfty\cA$.
\end{theorem}

We shall see in \cref{lem:sg-dec} that $\dec \cA$ and $\decinfty\cA$ are of small generation when $\cA$ is.

\begin{remark}
\label{rem:triple-adjunction}
The inclusion of $\MCong\cE\to \Acyclic\cE$ has also two right adjoints \cite[Theorem 3.3.11]{ABFJ:GT}
\[
\begin{tikzcd}
\Acyclic\cE
\ar[rrr,"{(-)\mono}" description]
\ar[from=rrr, shift left = 3,"{(-)\vee \Surj}",hook']
\ar[from=rrr, shift right = 3,"{\sf can.}"', hook']
&&& \MCong\cE\,,
\end{tikzcd}
\]
where the first right adjoint sends an acyclic class $\cA$ to its monogenic part $\cA\mono:= (\cA\cap \Mono)\ac$,
and where the image of the second right adjoint is the poset of {\it covering topologies}, which are the acyclic classes containing the class $\Surj$ \cite[definition 3.3.1]{ABFJ:GT}.
When this triple adjunction is composed with the one of \cref{thm:adjCongAcyclic}
\[
\begin{tikzcd}
\Cong\cE
\ar[rr, hook]
\ar[from=rr, shift right = 3,"(-)\cong"']
\ar[from=rr, shift left = 3,"\decinfty-"]
&& \Acyclic\cE
\ar[rrr,"{(-\cap \Mono)\ac}" description]
\ar[from=rrr, shift left = 3,"{(-)\vee \Surj}",hook']
\ar[from=rrr, shift right = 3,"{\sf can.}"', hook']
&&& \MCong\cE
\end{tikzcd}\,,
\]
we recover the triple adjunction of \cref{thm:mono-cong}.
\end{remark}

\subsubsection{Suspension of an acyclic class}
\label{sec:suspension}

In this section, we study the {\it suspension} of an acyclic class $\cA$, which is the smallest acyclic class containing the codiagonals of all the maps in $\cA$.
The suspension of the class $\All$ is the class $\Surj$ of surjections, and its iterated suspensions define the classes $\Conn n$ of $n$\=/connected maps.
Once defined the monoidal structure on acyclic classes in \cref{subsec:quant-locclass}, we shall see that the suspension of a modality is simply the product with the class $\Surj$.
The main results of the section are comparisons between the suspension of an acyclic class $\cA$ and its {\it decalage} introduced in \cref{sec:decalage} (see \cref{thm:suspension-acyclic,thm:explicit-decalage}).

\medskip
Recall from \cref{sec:codiagonal}, that, for a map $u:A\to B$, its codiagonal is the map $\nabla(u) = B\cup_AB\to B$.
We have also $\nabla(u)= s^0\pp u$, where $s^0$ is the map $S^0\to 1$ and $\pp$ is the pushout product.

\begin{definition}[Suspension]
\label{suspdefinition} 
The {\it suspension} of an acyclic class $\cA$ in a logos is defined to be the class 
\[
\susp\cA
\ :=\ 
\nabla(\cA)\ac
\ =\ 
(s^0\pp \cA)\ac
\,.
\]
We always have $\susp \cA \subseteq \cA$ and $\susp {\cA \vee \cB} = \susp \cA \vee \susp \cB$ in the poset $\Acyclic\cE$.
\end{definition}

The purpose of the section is to prove the following theorem, which provides a explicit description of the suspension of an acyclic class and its right orthogonal.
Recall the decalage $\dec\cA$ of an acyclic class $\cA$ from \cref{sec:decalage}.

\begin{theorem}[Suspension of acyclic classes]
\label{thm:suspension-acyclic}
If $\cA$ is an acyclic class in a logos, then 
\begin{align*}
\susp \cA & \ =\ \Delta^{-1}(\cA)\cap \Surj \ = \ \dec\cA\cap \Surj\,,\\
\susp \cA^\fperp & \ =\ \Delta^{-1}(\cA^\fperp)\,.
\end{align*}
\end{theorem}

We note in passing that the equality
$\susp \cA =\ \Delta^{-1}(\cA)\cap \Surj$ can be seen as
generalization of \cite[Theorem 4.3(1)]{Chach:inequ}. The proof will
need a number of preliminary results.

\begin{proposition}[Right orthogonal to suspension]
\label{prop:right-class-susp}
\label{lemma137} 
If $\Sigma$ is a class of maps in a logos, then 
\[
\nabla(\Sigma)^{\fperp}=\Delta^{-1}(\Sigma^{\fperp})\,.
\]
\end{proposition}
\begin{proof} 
Let $f:X\to Y$ be a map in $\cE$.
By \cref{cor:codiag-diag-ortho}, the condition $\nabla(\Sigma)\fperp f$ is equivalent to $\Sigma\fperp \Delta(f)$.
\end{proof}

\begin{lemma}[Generators for suspension]
\label{lem:suspensiongenerator}
For any class of maps $\Sigma$, we have 
\[
\susp{\Sigma\ac} = \nabla(\Sigma)\ac
\]
where $\nabla(\Sigma):=s^0\pp \Sigma$.
In particular, the suspension of an acyclic class of small generation is of small generation.
\end{lemma}
\begin{proof}
The statement is equivalent to the equality $\big(s^0\pp(\Sigma\ac)\big)\ac = (s^0\pp \Sigma)\ac$.
Using the saturation property of acyclic classes $(s^0\pp \Sigma)\ac = {^\fperp}\big((s^0\pp \Sigma)^\fperp\big)$ (\cref{lem:acyclicsat}), this is equivalent to show the equality of right orthogonals $\big(s^0\pp(\Sigma\ac)\big)^\fperp = (s^0\pp \Sigma)^\fperp$:
\begin{align*}
(s^0\pp \Sigma)^\fperp
&= \Delta^{-1}(\Sigma^\fperp) &&\text{by \cref{prop:right-class-susp}}\\
&= \Delta^{-1}\big((\Sigma\ac)^\fperp\big) &&\text{because $\Sigma^\fperp = (\Sigma\ac)^\fperp$, by \cref{lem:acyclicsat}}\\
&= (s^0\pp(\Sigma\ac))^\fperp &&\text{by \cref{prop:right-class-susp} again}\,.\qedhere
\end{align*}
\end{proof}

\begin{remark}
\label{rem:ortho-susp}
\Cref{prop:right-class-susp,lem:suspensiongenerator} applied to $\Sigma = \cA$ give the equality $\susp\cA^\fperp = \Delta^{-1}(\cA^\fperp)$ of \cref{thm:suspension-acyclic}.
Notice that since both $\cA$ and $\susp\cA$ are closed under base change, this equation is equivalent to the simpler statement $\susp\cA^\perp = \Delta^{-1}(\cA^\perp)$.	
\end{remark}

\begin{examples}
\begin{exmpenum}
\item \label{examplesusp:1}
Recall from \cref{examplefperp:n} that the acyclic class $\Conn n$ of $n$\=/connected maps is generated by the map $s^{n+1}:S^{n+1}\to 1$.
Then $\susp{\Conn n} = (s^0\pp s^{n+1})\ac =(s^{n+2})\ac = \Conn {n+1}$.
In particular, $\susp \All = \Surj$ and $\susp \Surj = \Conn 0$.

\item \label{examplesusp:3}
By definition of the acyclic class $\Acmap\cE$ of \cref{examplemodality:acyclic}, we have ${\susp{\Acmap\cE}=\Iso}$.
Moreover, this is the largest class of maps in $\cE$ satisfying this property.

\item \label{examplesusp:2}\label{rem:CORS}
Let $\cK$ be a congruence in a presheaf logos $\PSh C$.
The objects in $\cK^\fperp$ (the $F$ such that $F\to 1\in \cK^\fperp$) are the {\it sheaves} for $\cK$, and arbitrary maps in $\cK^\fperp$ are sometimes called the {\it relative sheaves} for $\cK$.
Recall that a presheaf $X$ is called {\it separated} if its diagonal is a relative sheaf.
By \cref{rem:ortho-susp}, the objects in $\susp\cK^\fperp$ are exactly the separated presheaves.
This justifies the name `separation' given to the suspension in \cite{CORS}.
\end{exmpenum}	
\end{examples}

\medskip
The following two lemmas are the keys of the proof of \cref{thm:suspension-acyclic}.
The first one is a slight generaliztion of \cite[Corollary 3.5]{Chach:inequ}.

\begin{lemma}
\label{Deltanabla}
If $\cA$ is an acyclic class in a logos, then $\Delta \nabla(\cA)\subseteq \cA$.
\end{lemma}

\begin{proof}
If $f\in \cA$, let us show that $\Delta \nabla f\in \cA$.
We shall first consider the case where $f$ is a map $p_A:A\to 1$. 
By construction, $\nabla(p_A)$ is the map $\Sigma A\to 1$, where $\Sigma A$ is the unreduced suspension of $A$.
Let us verify that the map $\Delta(\Sigma A): \Sigma A\to \Sigma A\times\Sigma A$ belongs to $\cA$.
By definition of $\Sigma A$, we have a pushout square
\begin{equation}
\label{pushoutsups}
\begin{tikzcd}
A\ar[d,"{p_A}"'] \ar[r,"{p_A}"] & 1  \ar[d,"{u_2}"] \\
1 \ar[r,"{u_1}"] & \pomark \Sigma A
\end{tikzcd}
\end{equation}
We have $u_i\in \cA$, since $p_A\in \cA$ and the class $\cA$ is closed under cobase changes since it is saturated.
Let us say that a family of maps $\{v_i : B_i \to B \}_{i\in I}$ in a logos is {\it collectively surjective} if the induced map $\coprod_iB_i\to B$ is surjective.
The pair of maps $u_i:1\to \Sigma A$ for $i=1,2$ is collectively surjective since the square \eqref{pushoutsups} is a pushout.
For each pair $(i,j)\in \{1,2\}$, consider the object $P(i,j)$ defined by following pullback square 
\begin{equation}
\label{sqpartij}
\begin{tikzcd}
P(i,j)\ar[d] \ar[r] & 1  \ar[d,"{u_j}"] \\
1 \ar[r,"{u_i}"] & \Sigma A
\end{tikzcd}
\end{equation}
The map $P(i,j)\to 1$ belongs to $\cA$, since the map $u_j$ belongs to $\cA$ and the class $\cA$ is closed under base changes.
The following square is cartesian, since the square \eqref{sqpartij} is cartesian.
\[
\begin{tikzcd}
P(i,j)\ar[d]\ar[rr,"{}"] && \Sigma A  \ar[d,"{\Delta(\Sigma A)}"] \\
1 \ar[rr,"{(u_i,u_j)}"] &&  \Sigma A \times \Sigma A
\end{tikzcd}
\]
But the quadruple of maps $(u_i,u_j):1\to \Sigma A\times \Sigma A$ for $i,j \in \{1,2\}$ is collectively surjective,
since the pair of maps $u_i:1\to \Sigma A$ for $i=1,2$ is collectively surjective.
Thus, $\Delta(\Sigma A)\in\cA$, since an acyclic class is always local and the map $P(i,j)\to 1$ belongs to $\cA$ for every $i,j \in \{1,2\}$.
Let us now consider the general case of a map $f:A\to B$.
We denote by $\cA\slice B$ the class of maps in $\cE\slice B$ obtained by pulling back the class $\cA$ along the forgetful functor $\cE\slice B\to \cE$.
The class $\cA\slice B$ is acyclic by \cref{inducedacyclic}.
If the map $f:A\to B$ belongs to $\cA$, then the map $(A,f)\to (B,1_B)$ in $\cE\slice B$ belongs to $\cA\slice B$.
Hence the map $\Delta (\Sigma(A,f))$ belongs to $\cA\slice B$ by the first part of the proof. 
The forgetful functor $U:\cE\slice B\to \cE$ preserves the operations $\Delta$ and $\nabla$, since it preserves pullbacks and pushouts.
Thus, $U\Delta(\Sigma(A,f))=\Delta(\nabla f)$.
It follows that the map $\Delta(\nabla f)$ belongs to the class $\cA$, since the map $\Delta(\Sigma(A,f))$ belongs to the class $\cA\slice B$.
 \end{proof}

In the next lemma we shall need various functors existing between a logos $\cE$ and its category of simplicial objects $\cE\splx$:
\[
\begin{tikzcd}
\cE\splx
\ar[rrr,"\colim = |-|" description, shift left=6]
\ar[from=rrr,"c" description, hook', shift right=2]
\ar[rrr,"\lim = (-)_0" description, shift right=2]
\ar[from=rrr,"C = \mathrm{cosk_0}" description, shift left=6, hook']
&&& \cE \,,
\end{tikzcd}
\]
where $c$ is the constant diagram functor, 
whose adjoints are
the limit functor (evaluation at the terminal object $[0]\in \Delta\op$)
and the colimit functor $|-|$.
The limit functor has a further right adjoint which is the 0-coskeleton.
Explicitly, this functor takes an object $A\in \cE$ to its {\it \v Cech nerve} $C(A)$ defined by putting $C(A)_n:= A^{n+1}$.
By \cite[6.2.3]{Lurie:HTT} or \cite[Lecture 4]{Rezk:Leeds} we have $|C(A)|=1$ if and only if the map $A\to 1$ is surjective.
Since both functors $c$ and $C$ are fully faithful, for every $A,B\in \cE$, we have canonical isomorphisms:
\begin{equation}
\label{fullfaith}
\Map{cA}{cB}
\ =\ 
\Map A B
\ =\ 
\Map {C(A)}{C(B)}\,.
\end{equation}
The various adjunctions provide further canonical isomorphisms
\begin{equation}
\label{fullfaithadj}
\Map A B
\ =\ 
\Map{(cA)_0} B
\ =\ 
\Map A{C(B)_0}
\ =\ 
\Map {cA}{C(B)}\,.
\end{equation}
The isomorphism \eqref{fullfaithadj} takes the identity map $1_A:A\to A$ to a map $\delta(A):cA\to C(A)$.
By construction, the map $\delta_n(A):A\to C(A)_n$ is the diagonal map $A\to A^{n+1}$ for every $n\geq 0$. Notice that $\delta_0(A)=1_A$ and that $\delta_1(A)=\Delta(A):A\to  A\times A$.

\begin{lemma}
\label{lem:diag-surj-ortho}
If $\cA$ is an acyclic class in a logos, then 
\[
(\Delta^{-1}(\cA)\cap\Surj )\fperp \Delta^{-1}(\cA^\fperp)
\]
and $(\Delta^{-1}(\cA)\cap\Surj )\subseteq \susp \cA$.
\end{lemma}

\begin{proof}

The class $\susp \cA$ is acyclic, thus $\susp \cA = {}^\fperp\big(\susp \cA^\fperp\big) = {}^\fperp\big(\Delta^{-1}(\cA^\fperp)\big)$ using \cref{lem:acyclicsat,rem:ortho-susp}.
This shows that the second statement followed from the first one.
By \cref{lem:acyclicsat} again, the relation $(\Delta^{-1}(\cA)\cap\Surj )\fperp \Delta^{-1}(\cA^\fperp)$ is equivalent to the simpler condition $(\Delta^{-1}(\cA)\cap\Surj)\perp \Delta^{-1}(\cA^\fperp)$, which is the one we are going to prove.

We fix two maps $u\in \Delta^{-1}(\cA)\cap \Surj$ and $f\in \Delta^{-1}(\cA^\fperp)$. 
We want to show that $u \upvdash f$.
We shall first consider the case where $u$ is a map $p_A:A\to 1$ and $f$ is a map $p_X:X\to 1$.
The assumption $p_A\in \Delta^{-1}(\cA)\cap \Surj$ means the the map $A\to 1$ is surjective and that the diagonal map $ \Delta(A):A\to A\times A$ belongs to $\cA$; the assumption $p_X\in \Delta^{-1}(\cA^\fperp)$ means that the diagonal map $ \Delta(X):X\to X\times X$  belongs to $\cA^\fperp$. 
We wish to show that  $p_A \upvdash p_X$.
We shall first prove that we have $\delta(A) \upvdash \delta(X)$ in the category of simplicial object $s\cE$.
For this, let us show that by induction on $n\geq 0$ that the map $\delta_n(A):A\to A^{n+1}$ belongs to $\cA$. 
Obviously, $\delta_0(A)=1_A\in \cA$. 
It is easy to see that $\delta_{n+1}(A)=(A\times \delta_{n}(A))\circ \Delta(A)$.
The map $A\times \delta_{n}(A)$ belongs to $\cA$, since the map $\delta_{n}(A)$ belongs to $ \cA$ by the induction hypothesis and the class $\cA$ is closed under base change.
Hence the map $\delta_{n+1}(A)$ belongs to $\cA$, since $\Delta(A)\in \cA$ by hypothesis and $\cA$ is closed under composition.
Similarly, the map $\delta_n(X):X\to X^{n+1}$ belongs to $\cA^\fperp$ for every $n\geq 0$, since the map $\Delta(X):X\to X\times X$ belongs to $\cA^\fperp$ and the class $\cA^\fperp$ is closed under composition and base change.
Hence we have $\delta_n(A) \perp \delta_n(X)$ for every $n\geq 0$, since $\cA \perp \cA^\fperp$.
It then follows  by  \cite[5.2.8.18]{Lurie:HTT} that we have $\delta(A) \perp \delta(X)$ in the category of simplicial objects $s\cE$.
Hence the following square is cartesian,
\begin{equation}
\label{orthogonalityforsusp}
\begin{tikzcd}
\Map{C(A)}{cX}
\ar[rrr,"\Map{\delta(A)}{cX}"]
\ar[d,"\Map{C(A)}{\delta(X)}"']
&&& 
\Map{cA}{cX}
\ar[d,"\Map{cA}{\delta(X)}"] \\
\Map{C(A)}{C(X)}
\ar[rrr,"\Map{\delta(A)}{C(X)}"]
&&&
\Map{cA}{C(X)}
\end{tikzcd} \,
\end{equation}
We saw in \eqref{fullfaith} and \eqref{fullfaithadj} that 
\[
\Map{C(A)}{C(X)}
\ =\ 
\Map{cA}{cX}
\ =\ 
\Map{cA}{C(X)}
\ =\ 
\Map A X
\,.
\]
Moreover, we have $\Map{C(A)}{cX} = \Map{|C(A)|}{X}$ by the adjunction $|-|\dashv c$.
But $|C(A)|=1$ since the map $A\to 1$ is surjective by assumption.
Thus, $\Map{C(A)}{cX} = \Map 1 X$. 
It follows that the square \eqref{orthogonalityforsusp} is isomorphic to the following square 
\begin{equation}
\label{orthogonalityforsusp2}
\begin{tikzcd}
\Map 1 X \ar[rrr,"\Map{p_A}{1_X}"] \ar[d,"\Map{p_A}{1_X}"'] &&& 
\Map A X \ar[d,"\Map{1_A}{1_X}", equal] \\
\Map A X \ar[rrr,"\Map{1_A}{1_X}", equal]  &&&
\Map A X
\end{tikzcd} \,
\end{equation}
The square \eqref{orthogonalityforsusp2} is cartesian, since the square \eqref{orthogonalityforsusp} is cartesian.
Hence the map $\Map{p_A}{1_X}$ is invertible, and $p_A\upvdash p_X$.

In the general case, we wish to show that we have $u\upvdash f$ for every map $u:A\to B$ in $\Delta^{-1}(\cA)\cap \Surj$ and every map $f:X\to Y$ in $\Delta^{-1}(\cA^\fperp)$.
For this, it suffices to show that every commutative square 
\begin{equation}
\label{sqfordfiller}
\begin{tikzcd}
A\ar[d, "u"']\ar[rr,"x"] && X
\ar[d,"f"] \\
B \ar[rr,"y"]&& Y
\end{tikzcd}
\end{equation}
has a unique diagonal filler $B\to X$.
Factoring the square \eqref{sqfordfiller} through a pullback
\begin{equation}
\label{decompos}
\begin{tikzcd}
A\ar[d,"u"']\ar[rr,"{(u,x)}"]  && B\times_Y X \ar[d,"{p_1}"] \ar[rr,"{p_2}"] \pbmarkk && X
\ar[d,"{f}"] \\
B \ar[rr, "1_B"] && B \ar[rr,"{y}"]&& Y
\end{tikzcd}
\end{equation}
we get that the square \eqref{sqfordfiller} has a unique diagonal filler if and only if the left hand square of \eqref{decompos} has a unique diagonal filler.
The class $ \Delta^{-1}(\cA^\fperp)$ is closed under base change, since the class $\cA^\fperp$ is closed under base change. 
Thus, $p_1\in \Delta^{-1}(\cA^\fperp)$ since $f\in  \Delta^{-1}(\cA^\fperp)$ by assumption.
For simplicity, let us put $Z:=B\times_Y X$ and $g=p_1$. 
We wish to show that the following square has a diagonal filler.
\begin{equation}
\label{sqfordfiller2}
\begin{tikzcd}
A\ar[d, "u"']\ar[rr,"x"] && Z \ar[d,"g"] \\
B \ar[rr,"{1_B}"]&& B
\end{tikzcd}
\end{equation}
By hypothesis, $g\in \Delta^{-1}(\cA^\fperp)$.  
The square \eqref{sqfordfiller2} can be viewed as a map $z:(A,u)\to (Z,g)$ between two objects of the category $\cE\slice{B}$.
The saturated orthogonal pair $\cA \perp \cA^\fperp$ induces a saturated orthogonal pair $\cA\slice B \perp \cA^\fperp\slice B$ in $\cE\slice B$ by \cref{inducedmodality}.
The diagonal $\Delta(X,g):(X,g)\to (X,g)\times (X,g)$ of the object $(X,g)$ of $\cE\slice{B}$  is defined by the diagonal $\Delta(g):X\to X\times_B X$ of the map $g:X\to B$. 
Hence the map $\Delta(X,g)$ belongs to $\cA^\fperp\slice B$, since the map $\Delta(g)$ belongs to $\cA^\fperp$.
 Similarly, the map $\Delta(A,u)$ belongs to $\cA\slice B$, since the map $\Delta(u)$ belongs to $\cA$.
Moreover, the map $(A,u)\to (B,1_B)$ in $\cE\slice{B}$ is surjective, since the map $u:A\to B$ is surjective.
Thus, $(A,u)\upvdash (Z,g)$ by the first part of the proof.
Hence the square \eqref{sqfordfiller2} has a unique diagonal filler.
It follows that the square \eqref{sqfordfiller} has a unique diagonal filler.
Since this square was arbitrary, this proves the orthogonality $u\upvdash f$.
\end{proof}

\begin{proof}[Proof of \cref{thm:suspension-acyclic}]
The second equality is \cref{rem:ortho-susp}. 
Let us prove that $\susp \cA = \Delta^{-1}(\cA)\cap\Surj = \dec\cA \cap \Surj$.
By \cref{lem:diag-surj-ortho}, we know that $\Delta^{-1}(\cA)\cap\ \Surj\subseteq \susp \cA$.
Since $\susp\cA \subseteq \cA$, it follows that
\begin{equation}
\label{suspensionthm}
\Delta^{-1}(\cA)\cap \Surj= \cA\cap (\Delta^{-1}(\cA)\cap \Surj)=
(\cA\cap \Delta^{-1}(\cA))\cap \Surj=\dec\cA\cap \Surj\,.
\end{equation} 
This shows that $\Delta^{-1}(\cA)\cap \Surj$ is an acyclic class, since both $\Surj$ and $\dec\cA$ are (\cref{lem:decalage}).
To show the inclusion $\susp \cA = \nabla(\cA)\ac\subseteq \Delta^{-1}(\cA)\cap \Surj$, it is thus sufficient to show that $\nabla(\cA) \subseteq \Delta^{-1}(\cA)\cap \Surj$.
But the codiagonal of any map is surjective, since it has a section.
And, by \cref{Deltanabla}, we have $\Delta \nabla(\cA)\subseteq \cA$ and thus $\nabla(\cA)\subseteq  \Delta^{-1}(\cA)$.
\end{proof}

Recall that an acyclic class $\cA$ is {\it epigenic} if $\cA\subseteq \Surj$.

\begin{corollary}[Suspension=decalage]
\label{corollarydescendantimcoimdecomp}
If an acyclic class $\cA$ is epigenic then $\susp \cA=\dec\cA$.
\end{corollary}
 
\begin{proof}
If $\cA \subseteq \Surj$, then $\dec\cA \subseteq \cA \subseteq \Surj$.
Hence $\susp \cA = \dec\cA \cap \Surj = \dec\cA$ by \cref{thm:suspension-acyclic}.
\end{proof}

A congruence $\cK$ is epigenic as an acyclic class if and only if $\cK\subseteq \decinfty\Surj= \Conn\infty$.
We have called such congruences {\it epic}.
We have seen in \cref{lem:decalage} that the fixed points of the decalage are congruences.
The following result characterizes the fixed points of the suspension as epic congruences.

\begin{lemma}
\label{coTCongrchar}
The following conditions on an acyclic class $\cA$ are equivalent:
\begin{lemenum}
\item\label{coTCongrchar:1} $\cA$ is epigenic and a congruence (\ie an epic congruence),
\item\label{coTCongrchar:3} $\cA = \susp\cA$.
\end{lemenum}
\end{lemma}

\begin{proof}
By \cref{lem:decalage}, an acyclic class $\cA$ is a congruence if and only if $\cA\subseteq\dec\cA$.
Thus, $\cA$ is epigenic and a congruence if and only if $\cA\subseteq \Surj\cap \dec\cA$.
But $\Surj\cap \dec\cA=\susp\cA$ by \Cref{thm:suspension-acyclic}.
Since the inclusion $\susp\cA\subseteq\cA$ is always true, this proves the equivalence.
\end{proof}

\begin{remark}
\label{rem:epic-coreflection}
In the same way that the iterated decalage provides a coreflection of acyclic classes into congruences (\cref{thm:adjCongAcyclic}), the iterated suspension provides a coreflection of acyclic classes into epic congruences.
\end{remark}

\medskip
We will now use \cref{thm:suspension-acyclic} to provide a new description of the decalage which makes clear how it compares to the suspension.

\begin{lemma}[Decalage from generators]
\label{lem:explicit-decalage}
For any class of maps $\Sigma$ in a logos, we have
\[
\dec{\Sigma\ac}
\ =\ 
\big(\nabla(\Sigma) \,\cup\,  \im \Sigma\big)\ac
\ =\ 
\big(\nabla\coim\Sigma \,\cup\, \im\Sigma\big)\ac\,.
\]
\end{lemma}
\begin{proof}
By \cref{prop:imcoimdecomp}, for any acyclic class we have $\cA = \coim \cA \vee \im\cA\ac$ in the poset of acyclic classes.
Applied to $\cA=\dec{\Sigma\ac}$, we get 
$\dec{\Sigma\ac} = \coim{\dec{\Sigma\ac}} \vee \im{\dec{\Sigma\ac}}\ac$.
We compute the coimage using \cref{thm:suspension-acyclic,lem:suspensiongenerator}:
\[
\coim{\dec{\Sigma\ac}}
\ =\ 
\dec{\Sigma\ac} \cap \Surj
\ =\ 
\susp{\Sigma\ac}
\ =\ 
\big(\nabla(\Sigma)\big)\ac\,.
\]
We compute the image using \cref{def:decalage}
\[
\im{\dec{\Sigma\ac}}
\ =\ 
\dec{\Sigma\ac}\cap \Mono
\ =\ 
\Sigma\ac \cap \Delta^{-1}(\Sigma\ac)\cap \Mono\,.
\]
The diagonal of any monomorphism is invertible, thus in $\Delta^{-1}(\Sigma\ac)$, hence $\Delta^{-1}(\Sigma\ac)\cap \Mono = \Mono$.
Then using \cref{prop:mono-part}, we have
\[
\im{\dec{\Sigma\ac}}\ac
\ =\ 
(\Sigma\ac \cap \Mono)\ac
\ =\ 
\im\Sigma\ac
\]
Putting everything together, we get the first equality:
\[
\dec{\Sigma\ac}
\ =\ 
\nabla(\Sigma)\ac \vee \im\Sigma\ac
\ =\ 
\big(\nabla(\Sigma) \cup \im\Sigma\big)\ac\,.
\]
For the second, we use $\Sigma\ac = \coim{\Sigma\ac} \vee \im\Sigma\ac$ (\cref{prop:imcoimdecomp})
and $\susp{\cA\vee \cB} = \susp\cA\vee \susp\cB$
to get 
\[
\nabla(\Sigma)\ac
\ =\ 
\susp{\Sigma\ac}
\ =\ 
\susp{\coim{\Sigma\ac} \vee \im\Sigma\ac}
\ =\ 
\susp{\coim{\Sigma\ac}} \vee \susp{\im\Sigma\ac}\,.
\]
Then we remark that $\susp{\im\Sigma\ac}\subseteq \im\Sigma\ac$ to get
\[
\dec{\Sigma\ac}
\ =\ 
\nabla(\Sigma)\ac \vee \im\Sigma\ac
\ =\ 
\susp{\coim{\Sigma\ac}} \vee \susp{\im\Sigma\ac} \vee \im\Sigma\ac
\ =\ 
\susp{\coim{\Sigma\ac}} \vee \im\Sigma\ac\,.\qedhere
\]
\end{proof}

\begin{lemma}
\label{lem:sg-dec}
For an acyclic class $\cA$ of small generation, 
the acyclic class $\dec \cA$ and
the congruence $\decinfty \cA$ are of small generation.
\end{lemma}

\begin{proof}
If $\cA=\Sigma\ac$, then $\susp\cA =(\im \Sigma \cup \nabla(\Sigma))\ac$ by \cref{lem:explicit-decalage}. 
Iterating this, we get that all acyclic classes $\decn n \cA$ are of small generation.
It follows by \cref{smallintersmallgen} that the acyclic class $\decinfty \cA$ is of small generation.
Then so is it as a congruence by \cref{smallintersmallgencong}.
\end{proof}

\medskip

\begin{theorem}[Decalage]
\label{thm:explicit-decalage}
If $\cA$ is an acyclic class in a logos, then 
\begin{align*}	
\dec\cA
&\ =\ 
\susp{\coim \cA}\vee \im \cA\ac
\ =\ 
\susp\cA\vee \cA\mono \,,\\
\dec\cA ^\fperp &\ =\  \Delta^{-1}(\cA^\perp) \cap \im\cA ^\perp
\end{align*}
(where the suprema are taken in the poset of acyclic classes).
\end{theorem}

\begin{proof}
The first equality is \cref{lem:explicit-decalage} applied to $\Sigma=\cA$.
We get also that 
$\dec\cA = (\nabla(\cA) \cup \im\cA)\ac = \nabla(\cA)\ac \vee \im\cA\ac$,
and the second equality follows by using $\susp\cA = \nabla(\cA)\ac$ and $\cA\mono = \im\cA\ac$.
For the last equality, we use the second one to get $\dec\cA^\fperp = \susp\cA^\fperp \cap (\im\cA\ac)^\fperp$, and \cref{thm:suspension-acyclic}
to get $\dec\cA ^\fperp =\Delta^{-1}(\cA^\fperp) \cap \im\cA ^\fperp$.
Then we remark that, since $\cA$, $\dec\cA$, and $\im\cA$ are all closed under base change, 
we have an equality 
$\Delta^{-1}(\cA^\fperp) \cap \im\cA ^\fperp=\Delta^{-1}(\cA^\perp) \cap \im\cA ^\perp$.
\end{proof}

\begin{remark}[Comparison decalage and suspension]
\label{rem:decalage-v-suspension}
Decalage and suspension are two way to ``shift'' an acyclic class.
The formula $\dec\cA = \susp{\coim \cA}\vee \im \cA\ac$ of \cref{thm:explicit-decalage} shows how they differ:
whereas the suspension $\susp\cA$ is generated by all suspensions of maps in $\cA$, 
the decalage $\dec\cA$ is generated by suspending only the epigenic part of $\cA$.
This has the consequence that the decalage preserves the monogenic part of $\cA$:
$\dec \cA\mono = \cA\mono$, whereas $\susp \cA\mono = \Iso$, since $\susp\cA\subseteq\Surj$.
\end{remark}
 
\begin{remark}[Higher decalage and suspension]
\label{rem:higher-dec-susp}
From \cref{thm:suspension-acyclic,thm:explicit-decalage}, one can see that
\begin{align*}	
\suspn n\cA &\ =\ \decn n \cA\cap \Conn {n-2} &&\hspace{-2cm}\text{for $0\leq n\leq \infty$, }\\
\decn n\cA &\ =\ \suspn n\cA\vee \cA\mono  &&\hspace{-2cm}\text{for $0\leq n< \infty$,}\\
	&\ =\ \suspn n \cA \vee \decn\infty\cA  &&\hspace{-2cm}\text{for $0\leq n\leq \infty$.}
\end{align*}
The last formula follows from the second one and $\cA\mono\subseteq\decn\infty\cA\subseteq \decn n\cA$.
Notice that the formula $\suspn\infty\cA = \bigcap_n \suspn n\cA = \decn\infty\cA\cap \Conn\infty$ holds, 
but that the formula $\decn\infty\cA = \suspn\infty\cA\vee\cA\mono$ is false: if $\cA$ is a congruence in a presheaf category, we have $\decn\infty\cA=\cA$ and $\suspn\infty \cA\subseteq\Conn\infty=\Iso$, so ${\decn\infty\cA = \suspn\infty\cA\vee\cA\mono }\iff {\cA=\cA\mono}$ which is false since not every congruence is monogenic (\ie a Grothendieck topology) in a presheaf category \cite[Example 4.1.19~(a)]{ABFJ:GT}.
\end{remark}

\subsection{Lex generators of congruences}

In this section, we recall the main theorem of \cite{ABFJ:HS} giving a necessary and sufficient condition for an acyclic class to be a congruence, and we introduce the notion of lex generator of a congruence.
This notion will be our main tool to compute acyclic products of congruences in practice (see \cref{boxprodlexgen}).
The last part recalls results about transport of congruences along morphisms of logoi.

\subsubsection{Diagonal criteria}
\label{sec:recognition}

Acyclic classes and congruences are intimately related.
We recall some results from \cite{ABFJ:HS}.
For a class of maps $\Sigma$ in a logos, we put
$\Delta(\Sigma) :=\{ \Delta u \, | \, u \in \Sigma \}$
and we define the {\it diagonal completion} of $\Sigma$ by
\[
\Sigma\diag
\ :=\ 
\big\{ \Delta^k u \, | \, u \in \Sigma,\, k\geq 0 \big\}
\,.
\]

\begin{proposition}[Recognition of congruences {\cite[Theorem~4.1.8(3)]{ABFJ:HS}}]
\label{lem:caraccong}
An acyclic class $\cA$ is a congruence if and only if $\Delta(\cA)\subseteq \cA$ if and only if $\cA\diag = \cA$. 
\end{proposition}

\begin{theorem}[Generation of congruences {\cite[Theorem 4.2.12 and Proposition 4.3.6]{ABFJ:HS}}]
\label{sigmaac=sigmacong}
If $\Sigma$ is a class of maps in a logos, then $\Sigma\cong=(\Sigma\diag)\ac$.
\end{theorem}

\begin{lemma}
\label{smallintersmallgencong} 
A congruence $\cK$ in a logos is of small generation if and only if it is of small generation as an acyclic class.
The intersection of a set of congruences of small generation is of small generation.
\end{lemma}

\begin{proof} 
If a congruence $\cK$ is of small generation as an acyclic class, then $\cK=\Sigma\ac$ for a set of maps $\Sigma$ and hence $\cK=(\Sigma\ac)\cong=\Sigma\cong$.
This shows that $\cK$ is also of small generation as a congruence.
Conversely, if $\cK=\Sigma\cong$ for a set of maps $\Sigma$,
then $\Sigma\cong =(\Sigma\diag)\ac$ by \cref{sigmaac=sigmacong}.
Thus, $\cK$ is of small generation as an acyclic class, since the set $\Sigma\diag$ is small.
So a congruence is of small generation if and only if it is of small generation as an acyclic class.
It then follows by  \Cref{smallintersmallgen}
that the intersection
of a set of congruences of small generation
is of small generation.
\end{proof}

Let $\Class\cE$ be the poset of classes of maps in a logos $\cE$ (\ie the poset of full subcategories of the arrow category $\cE\arr$).
Let $\Acyclic\cE$ be the subposet of acyclic classes and $\Cong\cE$ the subposet of congruences.
All the inclusions have left adjoints.
\[
\begin{tikzcd}
\Class\cE
\ar[rr,shift left = 1.6, "(-)\ac"]
\ar[from=rr,shift left = 1.6, hook']
\ar[rrrr,bend left = 30, "(-)\cong = \left((-)\diag\right)\ac"]
&
&
\Acyclic\cE
\ar[rr,shift left = 1.6, "(-)\cong"]
\ar[from=rr,shift left = 1.6, hook']
&
&
\Cong\cE
\end{tikzcd}
\]   
We have seen in \cref{thm:adjCongAcyclic} that the inclusion $\Cong\cE\to \Acyclic\cE$ has also a right adjoint $\decinfty-$.

\subsubsection{Lex generators}
\label{sec:lex-gen}

\begin{definition}
\label{def:lex-generator}
We will say that a class of maps $\Sigma$ in a logos is a {\it lex generator} if the class $\Sigma\ac$ is a congruence (\ie if $\Sigma\ac=\Sigma\cong$).
\end{definition}

\begin{lemma}
\label{caractlexgenerator}
A class of maps $\Sigma$ in a logos is a lex generator if and only if $\Delta(\Sigma)\subseteq \Sigma\ac$.
\end{lemma}

\begin{proof}
We put $\cA=\Sigma\ac$ and consider the decalage $\dec\cA = \cA\cap \Delta^{-1}(\cA) \subseteq \cA$.
The condition $\Delta(\Sigma)\subseteq \Sigma\ac$ is equivalent to $\Sigma \subseteq \dec\cA$.
Since $\dec\cA$ is acyclic by \cref{lem:decalage:1}, this is also equivalent to $\cA\subseteq \dec\cA$ and to $\cA=\dec\cA$, which is equivalent to $\cA$ being a congruence by \cref{lem:decalage:0}.
\end{proof}

\begin{examples}
\label{ex:lexgen}
\begin{exmpenum}

\item\label{ex:lexgen:diag}
Any class of maps $\Sigma$ such that $\Delta(\Sigma)\subset\Sigma$ is a lex generator.
In particular, any congruence is a lex generator since $\Delta(\cK)\subseteq\cK$.

\item\label{ex:lexgen:diag-comp}
For any class of maps $\Sigma$, the class $\Sigma\diag$ is a lex generator, 
since by construction $\Delta(\Sigma\diag) \subseteq \Sigma\diag$.
This can also be seen using the formula $(\Sigma\diag)\ac = \Sigma\cong$ of \cref{sigmaac=sigmacong}.

\item\label{ex:lexgen:mono}
Any class $\cM$ of monomorphisms is a lex generator, since $\Delta(\cM)\subseteq \Iso\subseteq \cM\ac$.

\item\label{ex:lexgen:superlexgen}
Since we have always $\Sigma\sat \subseteq \Sigma\sat\subseteq \Sigma\cong$,
any class of maps $\Sigma$ such that $\Sigma\sat = \Sigma\cong$ is a lex generator.
\end{exmpenum}
\end{examples}

The next lemma provides more examples of lex generators.
It can be applied in particular to $\Sigma=\cK$ a congruence.

\begin{lemma}
\label{lem:lex-gen}
If $\phi:\cE\to \cF$ is a morphism of logoi and $\Sigma$ a lex generator in $\cE$, then $\phi(\Sigma)$ is a lex generator in $\cF$.
\end{lemma}
\begin{proof}
The logos morphism $\phi$ preserves diagonals, since it is left-exact.
Thus $\phi(\Delta(\Sigma)) = \Delta(\phi(\Sigma))$ and
\[
\Delta(\phi(\Sigma))
\ =\ 
\phi(\Delta(\Sigma))
\ \subseteq\ 
\phi(\Sigma\ac)
\ \subseteq\ 
\phi(\Sigma\ac)\ac
\ =\ 
\phi(\Sigma)\ac
\]
where the last equality is from \cref{prop:transport-acyclic:image-ac}.
\end{proof}

\subsubsection{Transport of congruences}

Let $\phi:\cE\to \cF$ be a morphism of logoi.
Recall the morphism of suplattices $\phi(-)\ac:\Acyclic\cE \to \Acyclic\cF$ of \cref{prop:transport-acyclic}.

\begin{proposition}[Acyclic transport of congruences]
\label{prop:transport-cong}
Let $\phi:\cE\to \cF$ be a morphism of logoi, with associated congruence $\cK_\phi := \phi^{-1}(\Iso)$.
\begin{propenum}
\item\label{prop:transport-cong:inv-image}
For any congruence $\cJ$ in $\cF$ its inverse image $\phi^{-1}(\cJ)$ is a congruence in $\cE$ containing $\cK_\phi$.
Moreover, if $\cJ$ is of small generation, then so is $\phi^{-1}(\cJ)$.

\item\label{prop:transport-cong:image-cong}
For any congruence $\cI$ in $\cE$ its acyclic image $\phi(\cI)\ac$ is a congruence,
and $\phi(-)\ac$ restrict into a morphism of suplattices $\Cong\cE \to \Cong\cF$ which is left adjoint to $\phi^{-1}$.
Moreover, if $\cI$ is of small generation, then so is $\phi(\cI)\ac$.

\item\label{prop:transport-cong:image-mcong}
For any monogenic congruence $\cI$ in $\cE$ its acyclic image $\phi(\cI)\ac$ is a monogenic congruence,
and $\phi(-)\ac$ restrict into a morphism of suplattices $\Cong\cE \to \Cong\cF$ whose right adjoint is $\phi^{-1}(-)\mono$.

\item\label{prop:transport-cong:quotient-cong}
When $\phi$ is a quotient of logoi, the adjunction $\phi(-)\ac\dashv \phi^{-1}$ restricts into an isomorphism between $\Cong\cF$ and the poset $\Cong\cE\upslice {\cK_\phi}$ of congruences in $\cE$ containing $\cK_\phi$.

\item\label{prop:transport-cong:quotient-gtop}
When $\phi$ is a quotient of logoi, the adjunction $\phi(-)\ac\dashv \phi^{-1}(-)\mono$ restricts into an isomorphism between $\MCong\cF$ and the poset $\MCong\cE\upslice {\cK_\phi}$ of congruences in $\cE$ containing $\cK_\phi\mono$.

\end{propenum}
\end{proposition}
\begin{proof}
\noindent(1) 
We already saw in \cref{exmp:congruence2} that the inverse image of an acyclic class is acyclic.
The inclusion $\cK_\phi\subseteq\phi^{-1}(\cJ)$ comes from $\Iso\subseteq \cJ$.
If $\cJ$ is of small generation the quotient $\cF\quotient \cJ$ exists and $\phi^{-1}(\cJ)$ is then the congruence of the morphism $\cE\to \cF\to \cF\quotient\cJ$.
It is of small generation by \cref{exmp:congruence2}.

\smallskip
\noindent(2) 
Let $\cI$ be a congruence in $\cE$. 
The class $\phi(\cI)$ is a lex generator by \cref{lem:lex-gen}, thus $\phi(\cI)\ac$ is a congruence.
This proves that $\phi(-)\ac:\Cong\cE\to \Cong\cF$ is well defined.
Both functors of the adjunction $\phi(-)\ac\dashv \phi^{-1}$ restricts to congruences, thus their restrictions are still adjoint.
If $\cI=\Sigma\cong = (\Sigma\diag)\ac$ then $\phi(\cI)\ac = \phi(\Sigma\diag)\ac$ by \cref{prop:transport-acyclic:image-ac}.
If $\Sigma$ is a set, then so is $\Sigma\diag$. 
This shows that $\phi(\cI)\ac$ is of small generation if $\cI$ is.

\smallskip
\noindent(3) 
The morphism $\phi$ preserves monomorphisms since it is left-exact.
If $\Sigma$ is a class of monomorphisms, then $\phi(\Sigma)$ is a class of monomorphism.
Then the formula $\phi(\Sigma\ac)\ac = \phi(\Sigma)\ac$ of \cref{prop:transport-acyclic:image-ac} shows that $\phi(-)\ac$ preserve monogenic acyclic classes.
The verification that the right adjoint of $\phi^{-1}(-)\mono$ is straightforward.

\smallskip
\noindent(4) 
is proved in \cite[Proposition 2.2.32]{ABFJ:GT}.
(5) 
is proved similarly.
\end{proof}

\begin{lemma}
\label{lem:pushout-cong}
Let $\phi:\cE\to \cF$ be a morphism of logoi and $\cK$ a congruence of small generation in $\cE$. 
Then the square
\[
\begin{tikzcd}
\cE\ar[r,"\psi"]\ar[d,"\phi"'] & \cE\quotient \cK\ar[d,"\phi'"]\\
\cF \ar[r,"\psi'"]& \cF\quotient\phi(\cK)\ac    
\end{tikzcd}
\]
is a pushout in $\Logos$.
\end{lemma}
\begin{proof}
We have seen in \cref{prop:transport-cong:image-cong} that $\phi(\cK)\ac$ is of small generation if $\cK$ is. 
Thus the quotient $\cF\quotient\phi(\cK)\ac$ exists.
We need to show that the following diagram of mapping spaces in $\Logos$ is cartesian for every logos $\cG$:
\[
\begin{tikzcd}
\Map{\cF\quotient \phi(\cK)\ac}\cG\ar[r,"-\circ\psi'", hook]\ar[d,"-\circ\phi'"'] & \Map\cF\cG\ar[d,"-\circ\phi"]\\
\Map{\cE\quotient \cK}\cG \ar[r,"-\circ\psi", hook]& \Map\cE\cG\,.
\end{tikzcd}
\]
By definition of localizations, the horizontal morphism are inclusion of full subgroupoids, with respective images the subgroupoids 
of functors $g:\cE\to\cG$ sending $\cK$ to $\Iso$
and
of functors $g':\cF\to\cG$ sending $\phi(\cK)\ac$ to $\Iso$, or equivalently $\phi(\cK)$ to $\Iso$.
With this identification, the square is cartesian if and only if 
${g'(\phi(\cK))\subseteq \Iso} \iff {(g'\phi)(\cK)\subseteq\Iso}$, which is obvious.
\end{proof}

\section{The acyclic product}
\label{sec:product}

The goal of this section is to define monoidal structures on several types of classes of maps in a logos $\cE$:
\begin{enumerate}
\item on the poset $\Class\cE$ of all classes of maps (\cref{prop:mon-all}),
\item on the poset $\Acyclic\cE$ of acyclic classes (\cref{thm:aprod}),
\item on the poset $\Cong\cE$ of congruences (\cref{thm:cprod}), and
\item on the poset $\MCong\cE$ of monogenic congruences/Grothendieck topologies (\cref{thm:gtop}).
\end{enumerate}
This will turn these posets into $\otimes$\=/frames (\cref{def:tensor-frame}), and the latter will even be a frame (\cref{framearequantal}).
We then prove that several of the canonical functors between these posets are morphisms of $\otimes$\=/frames (see diagram~\eqref{monoidal-morphisms}).
In particular, we shall see that all the inclusions
\[
\MCong\cE
\subtto
\Cong\cE
\subtto
\Acyclic\cE\,,
\]
as well as the functor $(-)\mono:\Acyclic\cE\to \MCong\cE$ extracting the monogenic part of an acyclic class, are monoidal.

\subsection{Product and division of classes of maps}
\label{sec:mon-all}

Recall from \cref{classical:boxmonoidal} that the pushout product functor in a logos $\cE$
\[
-\pp - : \Arr\cE\times  \Arr\cE \tto  \Arr\cE
\]
is the tensor product of a symmetric monoidal closed structure on the arrow category $\cE\arr$. 
The unit object is the map $i:0\to 1$.

We denote $\Class\cE$ the (large) poset of (replete) classes of maps in $\cE$, it is a (large) complete lattice.
If $\cA$ and $\cB$ are two classes of maps, we define
\[
\cA\pp\cB := \{u\pp v\ | \ u\in \cA \ {\rm and} \ v\in \cB\}
\qquad \mathrm{and} \qquad 
\cA\div\cB :=\{f\in \cE \ | \ \cA \pp f \subseteq \cB\}\,.
\]
(In the equalities above, it is implicit that we are taking the replete closure of the right hand side.)
We shall say that $\cA\pp\cB$ is the {\it  product} of $\cA$ and $\cB$ 
and say that $\cA\div\cB$ is the {\it division} of $\cB$ by $\cA$.
Obviously,
\begin{equation}
\label{adjboxclasses}
\cA\pp \cB\subseteq \cC
\quad\iff\quad 
\cB\subseteq \cA \div \cC
\end{equation}

We shall call a {\it $\otimes$\=/lattice} a suplattice equipped with a symmetric monoidal structure (see \cref{sec:mframe}).

\begin{proposition}
\label{prop:mon-all}
The product $\pp$ defines a $\otimes$\=/lattice structure on $\Class\cE$, whose unit is the class $\{i:0\to 1\}$ consisting of a single map.
\end{proposition}
\begin{proof}
This is essentially \cref{quantfrommonoid}.
The collection of isomorphism classes of objects of $\Arr\cE$ has the structure of a commutative monoid $\cM$ with the product $\pp:\cM\times \cM\to \cM$ induced by the pushout product on $\Arr\cE$.
The poset $\Class\cE$ is then isomorphic to the powerset $\powerset \cM$, since every class $\cA\in \Class\cE$ is a disjoint union of isomorphism classes.
\end{proof}

Since the pushout product of any map with an isomorphism is an isomorphism, the class $\Iso$ of all isomorphisms in $\cE$ is an absorbing element in $(\Class\cE,\pp,i)$.
We denote $\All$ the class of all maps in $\cE$. 
It is the maximal element of $\Class\cE$.

\bigskip
The following lemma lists a number of compatibilities of the division operation with various operations on classes of maps.
It will be used in \cref{thm:aprod} to construct the monoidal structure on $\Acyclic\cE$.

\begin{lemma}
\label{lem:closure-pp}
We fix two arbitrary classes of maps $\cA$ and $\cB$ in a logos.
\begin{enumerate}
\item \label{lem:closure-pp:iso}
If $\Iso \subseteq \cA$, then 
$\Iso \subseteq \cB\div \cA$.

\item \label{lem:closure-pp:bc}
If $\cA$ is closed under base change, then so is $\cB\div\cA$.

\item \label{lem:closure-pp:cbc}
If $\cA$ is closed under cobase change, then so is $\cB\div\cA$.

\item \label{lem:closure-pp:compo}
If $\cA$ is closed under cobase change and composition, then $\cB\div\cA$ is closed under composition.

\item \label{lem:closure-pp:colim}
If $\cA$ is closed under colimits, then so is $\cB\div\cA$.

\end{enumerate}

In particular, if the class $\cA$ is acyclic, then so is $\cB\div\cA$ for any class $\cB$.

\end{lemma}

\begin{proof}
All the properties of the lemma are stable by intersection.
Thus, the formula 
\[
\cB\div\cA=\bigcap_{u\in \cB} u\div \cC
\]
allows to reduce the proofs to the case where $\cB$ is a single map $u$.

\noindent(1) 
For any $u$, we have $\Iso \pp u \subseteq \Iso$. 
The inclusion $\Iso \subseteq \cA$ implies $\Iso \pp u \subset\cA\pp u$.
The inclusion $\Iso \pp u \subseteq \Iso \subseteq \cA$ implies $\Iso\subseteq u\div \cA$.

\smallskip
\noindent(2) 
Let $f$ be a map in $\cB\div\cA$ and let $f'\to f$ be a cartesian morphism.
By \cref{basechangebox}, the morphism $u\pp f'\to u\pp f$ is cartesian.
By hypothesis, $u\pp f$ is in $\cA$ and $\cA$ is closed under base change, thus the map $u\pp f'$ is in $\cA$ and $f$ is in $u\div \cA$.

\smallskip
\noindent(3) 
Same as (2) 
but using \cref{cobasechangebox}.

\smallskip
\noindent(4) 
Let $f$ and $g$ be two maps in $\cB\div\cA$ such that the composite $fg$ exists.
\Cref{box-compo} says that $u\pp (fg) = (u\pp f)(u\pp g)'$ where $(u\pp g)'$ is a cobase change of $a\pp g$.
The hypothesis on $\cA$ imply that $u\pp (fg)$ in $\cA$ and that $fg$ is in $u\div \cA$.

\smallskip
\noindent(5) 
Let $f_i$ be a diagram in $\cB\div\cA$.
The functor $u\pp -$ preserves colimits, hence $u\pp\colim f_i = \colim u\pp f_i$.
Since $\cA$ is assumed closed under colimits, the map $u\pp\colim f_i$ is in $\cA$,
and $\colim f_i$ is in $u\div \cA$.

\smallskip
\noindent
The last assertion follows from the above properties and \cref{def:acyclic}.
\end{proof}

\subsection{The algebra of acyclic classes}
\label{subsec:quant-locclass}

Recall that every class of maps $\Sigma$ in a logos is contained in a smallest acyclic class $\Sigma\ac$, called the {\it acyclic closure} of $\Sigma$. 
For example, if $i$ is the map $0\to 1$, then $\{i\}\ac$ is the class of all maps $\All$ by \cref{ex:acyclicgene:-1}.

\begin{definition}[Acyclic product]
\label{deflocalbox}
If $\cA$ and $\cB$ are two acyclic classes of maps in a logos we shall say that the class
\[
\cA\aprod \cB:= (\cA\pp\cB)\ac
\]
is the {\it acyclic product} of $\cA$ and $\cB$.
\end{definition}

\begin{theorem}[Acyclic product]
\label{thm:aprod}
The acyclic product $\cA\aprod\cB$ is a $\otimes$\=/frame structure on the poset $\Acyclic\cE$ of acyclic classes of a logos $\cE$.
The reflection $(-)\ac:\Class\cE\to \Acyclic\cE$ is monoidal:
\begin{equation}
\label{eq:ac=monoidal}
(\Sigma \pp \Tau)\ac = \Sigma\ac \aprod \Tau\ac \,.	
\end{equation}
The subposet of acyclic classes of small generation is closed under products and small suprema, turning it into a sub-$\otimes$\=/frame. 
Moreover, the division of acyclic classes of small generation is of small generation.
\end{theorem}
\begin{proof}
We use \cref{localizationquantale} on the reflector $R=(-)\ac:\Class\cE\to \Class\cE$.
We need to show that for any class $\cB$ and any acyclic class $\cA$, the class $\cB\div\cA$ is acyclic.
A class is acyclic if it contains $\Iso$, is closed under base change and colimits.
The result follows from \cref{lem:closure-pp}~\eqref{lem:closure-pp:iso}, \eqref{lem:closure-pp:bc}, and \eqref{lem:closure-pp:colim}.
This provides a $\otimes$\=/frame structure on $\Acyclic\cE$.
The unit element is $\{0\to 1\}\ac = \All$, which is also the maximal element, hence the fact that it is a $\otimes$\=/frame.

The formula $(\Sigma \pp \Tau)\ac = \Sigma\ac \aprod \Tau\ac$ is a consequence of \cref{localizationquantale}.
We deduce from it that the product $\cA\aprod\cB$ of two acyclic classes of small generation is of small generation.
Let us see that the division $\cA\div \cB$ is also of small generation.
For $\cA=\Sigma\ac$, we have by \cref{localizationquantale}
\[
\cA\div \cB=\Sigma\ac \div \cB=\Sigma \div \cB=\bigcap_{u\in \Sigma} u\div \cB\,.
\]
The intersection of a set of acyclic classes of small generation is of small generation by \Cref{smallintersmallgen}.
So this reduces the problem to show that, for any map $u$, the acyclic class $u\div \cB$ is of small generation.
For simplicity, let us put $\cC:=u\div \cB$.
By \Cref{smallgenpres}, it suffices to show that the full subcategory $\underline{\cC} \subseteq \cE\arr$ spanned by the maps in $\cC$ is presentable.
By construction, $\underline{\cC}$ is the inverse image of the subcategory $\underline{\cB}$ by the functor $\phi:=u\pp -: \cE\arr\to \cE\arr$, which is cocontinuous.
The category $\underline{\cB}$ is presentable by \cref{smallgenpres}, since the acyclic class $\cB$ is of small generation.
It follows by \cite[Proposition 5.5.3.12]{Lurie:HTT} that the category $\underline{\cC}=\phi^{-1}(\underline{\cB})$ is presentable. 
Thus, the acyclic class $\cC:=u\div \cB$ is of small generation by \cref{smallgenpres}.
\end{proof}

\begin{remark}[Acyclic product in intersection]
\label{rem:prod-intersection}
The class $\All$ being the unit, we always have $\cA \aprod \All = \cA$ for every acyclic class $\cA$.
For any other acyclic class $\cB$ the inclusion $\cB\subseteq \All$ induces an inclusion $\cA\aprod \cB \subseteq \cA$.
By symmetry, this shows that we always have $\cA\aprod \cB \subseteq \cA\cap \cB$.
\end{remark}


\begin{examples}
\label{ex:aprod}
The formula $\Sigma\ac\aprod \Tau\ac = (\Sigma \pp \Tau)\ac$ lets us compute a number of examples.
\begin{exmpenum}

\item\label{ex:aprod:unit}\label{ex:aprod:power}
We saw already that the class $\All$ is the unit, hence $\All \aprod \cA = \cA$. 
Since $\All$ is also the top element, this implies the relation $\cA\aprod\cB\subseteq \cA \aprod \All = \cA$.
As a consequence, for any acyclic class $\cA$, the powers of $\cA$ define a decreasing sequence
\[
\dots
\ \subseteq\ 
\cA ^{3}
\ \subseteq\ 
\cA \aprod \cA
\ \subseteq\ 
\cA\,.
\]

\item\label{ex:aprod:iso}
The class $\Iso$ is an absorbing element: $\Iso \aprod \cA = (\Iso \pp \cA)\ac = \Iso$.

\item\label{ex:aprod:susp} 
Recall from \cref{examplefperp:0} that the acyclic class $\Surj$ of surjections is generated by the map $s^0:S^0\to 1$. Then by definition of the suspension of an acyclic class (\cref{suspdefinition}), we have 
\[
\susp \cA
\ =\ 
(s^0\pp\cA)\ac
\ =\ 
\Surj \aprod \cA
\,.
\]

\item\label{ex:aprod:n-conn}
Recall from \cref{examplefperp:n} that the acyclic class $\Conn n$ of $n$\=/connected maps is generated by ${s^{n+1}:S^{n+1}\to 1}$. 
Recall also that the join of two spheres is a sphere $S^m \join S^n = S^{m+n}$, that is $s^{m}\pp s^{n} = s^{m+n}$.
Thus, we can generalize \cref{examplesusp:1} into
\[
\Conn m \aprod \Conn n = (s^{m+1}\pp s^{n+1})\ac = (s^{m+n+3})\ac = \Conn {m+n+2}\,.
\]

\item\label{ex:aprod:acyclic-maps} 
Recall from \cref{examplemodality:acyclic}, the acyclic class $\Acmap\cE$ of acyclic maps in $\cE$, that is of map $f$ such that $\nabla f$ is invertible.
Using the acyclic product and division operations, we get:
\[
\text{$f$ is acyclic}
\ \iff \ 
s^0 \pp f \in \Iso
\ \iff \ 
f\in \{s^0\}\div \Iso
\ \iff \ 
f\in \Surj\div \Iso
\]
(where the last equivalence is \cref{localizationquantale}). 
Hence $\Acmap\cE = \Surj\div \Iso$, and this recovers that it is an acyclic class.
In fact, the more general formula $\Acmap\cE = \Conn n\div \Iso$ is also true (for every $n\geq -1$).
The next result shows that acyclic maps are nilpotent for the pushout product.

\end{exmpenum}	
\end{examples}

\begin{proposition}
\label{prop:acyclic-map-nilp}
If $\cE$ is hyper-reduced, we have $\Acmap\cE\aprod \Acmap \cE = \Iso$.
\end{proposition}
\begin{proof}
For simplicity we put $\cA:=\Acmap\cE$.
Any acyclic map $f$ is always 0-connected by \cite[Lemma 3]{Hoyois:acyclic}, 
so $\cA\aprod\cA\subseteq \Conn 0\aprod\Conn 0 = \Conn 2$.
Moreover, we always have $\cA\aprod \cA\subseteq \cA$ since $\cA$ is acyclic.
We deduce that $\cA\aprod \cA\subseteq \cA\cap \Conn 2 \subseteq \cA\cap \Conn 1$.
By \cite[Lemma 3]{Hoyois:acyclic}, any simply connected acyclic map is \oo connected, thus an isomorphism.
This shows that $\cA\cap \Conn 1=\Iso$ and implies that $\cA\aprod \cA = \Iso$.
\end{proof}

\subsection{The algebra of congruences}
\label{sec:push-prod-cong}

In this section, we build a monoidal structure on congruences.
In view of the analogy between the completion $\Sigma\ac$ and $\Sigma\cong$,
it seems natural to define the product of two congruences $\cK$ and $\cL$ as $\cK\cprod\cL :=(\cK\pp\cL)\cong$, the congruence generated by all pushout products of maps in $\cK$ and $\cL$.
This formula will end up being the correct one (see \cref{rem:cprod}) but somehow by accident, because the functor $(-)\cong:\Class\cE\to \Cong\cE$ will not be monoidal
(nor will be its restriction $(-)\cong:\Acyclic\cE\to \Cong\cE$, see \cref{rem:not-monoidal}).
The proof of \cref{thm:aprod} is actually impossible to repeat for congruences because of the insufficient compatibility of the pushout product with finite limits.
Precisely, to apply \cref{localizationquantale} to congruences, we would need a missing relation between $\Delta (u\pp f)$ and $u \pp \Delta f$.

We will thus construct the product of congruences with a very different idea.
We will actually prove in \cref{mainthmCongruence} that the acyclic product of two congruences is again a congruence.
In other terms, we will show that the subposet $\Cong\cE\subseteq \Acyclic\cE$ is closed under the monoidal structure of \cref{thm:aprod}.
The proof uses the decalage $\decname$ of acyclic classes (\cref{def:decalage}). 
The main step is to show that $\decname$ is a lax monoidal endofunctor on $\Acyclic\cE$ (\cref{boxofsatellites}), which we deduce easily from the description of the decalage obtained in \cref{thm:explicit-decalage}.
From there the result follows from the fact that congruences are the fixed points of $\decname$ (\cref{lem:decalage}).

\bigskip

Recall from \cref{sec:image} the image $\im\cA = \cA\cap \Mono$ of an acyclic class and the {\it monogenic part} $\cA\mono=(\cA\cap \Mono)\ac =\im \cA\ac$ of an acyclic class $\cA$.
The class $\cA\cap \Mono$ is an extended Grothendieck topology (\cref{def:GT}) and satisfies the hypothesis of \cref{lem:square-box-mono}.

\begin{proposition}[Monogenic idempotence]
\label{prop:idem-mono-cong}
If $\cA$ and $\cB$ are acyclic classes in a logos (\eg congruences), then 
$\cA\mono\aprod \cB\mono = (\cA\aprod \cB)\mono$.
Moreover, $\cA\mono\aprod \cA\mono = \cA\mono$.
\end{proposition}
\begin{proof}
We will use the formulas $\cA\mono\aprod\cB\mono = \big(\im\cA\pp\im\cB\big)\ac$  and $(\cA\aprod \cB)\mono = \big((\cA\aprod \cB) \cap \Mono\big)\ac$.

Let us first prove $\cA\mono\aprod \cB\mono \subseteq (\cA\aprod \cB)\mono$.
We always have $\im\cA\pp \im\cB \subseteq \cA\pp \cB$.
By \cref{boxprodmono}, we have also $\im\cA\pp \im\cB \subseteq \Mono$.
Hence
\[
\im\cA\pp \im\cB
\ \subseteq\ 
\big(\cA\pp \cB\big)\cap \Mono
\ =\ 
\im {\cA\aprod \cB}\,,
\]
and $\cA\mono\aprod \cB\mono \subseteq (\cA\aprod \cB)\mono$.

We prove now $(\cA\aprod \cB)\mono \subseteq \cA\mono\aprod \cB\mono$.
The class $\im{\cA\aprod \cB}$ is closed under base change.
Applying \cref{lem:square-box-mono} gives
\[
\im{\cA\aprod \cB}
\ =\ 
\big(\im{\cA\aprod \cB}\ \pp\ \im{\cA\aprod \cB}\big)\bc
\ \subseteq\ 
\big(\im{\cA\aprod \cB}\ \pp\ \im{\cA\aprod \cB}\big)\ac\,,
\]
since, for any class $\Sigma$, we always have $\Sigma\bc\subseteq \Sigma\ac$.
The inclusion $\cA\aprod \cB\subseteq \cA$ implies $\im{\cA\aprod \cB}\subseteq \im\cA$.
We use this and the same formula with $\cA$ and $\cB$ exchanged to get
\[
\im{\cA\aprod \cB}\ \pp\ \im{\cA\aprod \cB}
\quad \subseteq\quad 
\im\cA\ \pp\ \im\cB\,.
\]
Finally, 
\[
\im{\cA\aprod \cB}
\ \subseteq\ 
\big(\im\cA \pp \im\cB\big)\ac
\ =\ 
\cA\mono\aprod \cB\mono\,,
\]
and $(\cA\aprod \cB)\mono \subseteq \cA\mono\aprod \cB\mono$.

Let us see now that $\cA\mono\aprod \cA\mono = \cA\mono$.
The class $\im\cA = \cA\cap \Mono$ is always closed under base change.
Using \cref{lem:square-box-mono} again, we have $(\im\cA\pp \im\cA)\bc = \im\cA$.
Therefore $(\im\cA\pp \im\cA)\ac = \im\cA\ac$, 
since for any class $\Sigma$ we always have $(\Sigma\bc)\ac= \Sigma\ac$.
This shows that $\cA\mono\aprod \cA\mono = \cA\mono$.
\end{proof}

\begin{lemma}[Lax monoidal $\decname$]
\label{boxofsatellites} 
If $\cA$ and $\cB$ are acyclic classes in a logos, then 
\[
\dec\cA\aprod \dec\cB
\ \subseteq\ 
\dec{\cA\aprod \cB}\,.
\]
\end{lemma}

\begin{proof}
By \Cref{thm:explicit-decalage}, we have
\begin{enumerate}[label=(\roman*)]
\item\label{pf:boxofsatellites:a} $\dec\cA= \susp\cA \vee \cA\mono$
\item\label{pf:boxofsatellites:b} $\dec\cB= \susp\cB\vee \cB\mono$
\item\label{pf:boxofsatellites:c} $\dec{\cA\aprod \cB}= \susp{\cA\aprod \cB}\vee (\cA\aprod \cB)\mono$
\end{enumerate}
where the suprema are taken in the lattice $\Acyclic\cE$.
The acyclic product $\Acyclic\cE\times \Acyclic\cE\to \Acyclic\cE$ preserves suprema in each variable, since it is the multiplication of a $\otimes$\=/frame.
It follows from \ref{pf:boxofsatellites:a} and \ref{pf:boxofsatellites:b} that 
\[
\dec\cA\aprod \dec\cB
\ \ =\ \ 
\susp\cA\aprod \susp\cB  \  \ \vee \  \
\cA\mono \aprod  \susp\cB\ \ \vee  \ \
\susp\cA \aprod \cB\mono \  \ \vee \  \
\cA\mono \aprod \cB\mono \,.
\]
Let us verify the relations
\begin{enumerate}[label=(\Alph*)]
\item\label{pf:boxofsatellites:1}
$\susp\cA\aprod\susp\cB\ \subseteq\ \susp{\cA\aprod\cB}$,
\item\label{pf:boxofsatellites:2}
$\cA\mono\aprod\susp\cB\ \subseteq\ \susp{\cA\aprod\cB}$,
\item\label{pf:boxofsatellites:3}
$\susp\cA\aprod\cB\mono\ \subseteq\ \susp{\cA\aprod\cB}$, 
\item\label{pf:boxofsatellites:4}
$\cA\mono\aprod\cB\mono\ \subseteq\ (\cA\aprod\cB)\mono$.
\end{enumerate}
The relation \ref{pf:boxofsatellites:1} follows from the inclusion $\Surj\aprod \Surj \subseteq \Surj$. 
We have
\[
\susp\cA\aprod \susp\cB
\ =\ 
\Surj\aprod\cA\aprod \Surj\aprod\cB
\ \subseteq\ 
\Surj\aprod\cA\aprod \cB
\ =\ 
\susp{\cA\aprod \cB}\,.
\]
The relation \ref{pf:boxofsatellites:2} follows from the inclusion  $\cA\mono\subseteq \cA$.
We have 
\[
\cA\mono \aprod \susp\cB
\ =\ 
\cA\mono \aprod \Surj\aprod\cB
\ \subseteq\ 
\cA\aprod\Surj\aprod\cB
\ =\ 
\susp{\cA \aprod \cB}\,.
\]
The relation \ref{pf:boxofsatellites:3} is proved similarly.
The relation \ref{pf:boxofsatellites:4} follows from \cref{prop:idem-mono-cong}.
Taken together, these relations imply that 
\[
\dec\cA\aprod\dec\cB
\ \subseteq\ 
\susp{\cA\aprod \cB}\vee (\cA\aprod \cB)\mono
\ =\ 
\dec{\cA\aprod \cB}\,.\qedhere
\]
\end{proof}

We are now ready to prove the main result of this paper.

\begin{theorem}[Product of congruences]
\label{mainthmCongruence}
\label{thm:cprod}
The acyclic product of two congruences on a logos is a congruence.
The poset of congruences $\Cong\cE$ is a sub-$\otimes$\=/frame of the $\otimes$\=/frame of acyclic classes $\Acyclic\cE$ of \cref{thm:aprod}.
\end{theorem}

\begin{proof} 
Let us show that the acyclic product $\cJ\aprod \cK$ of two congruences $\cJ$ and $\cK$ is a congruence.
By \cref{lem:decalage} we have $\cJ = \dec\cJ$ and $\cK = \dec\cK$.
It then follows by \cref{boxofsatellites} that
\[
\cJ\aprod  \cK
\ =\ 
\dec\cJ\aprod \dec\cK
\ \subseteq\ 
\dec{\cJ\aprod \cK}
\]
By \cref{lem:decalage}, this shows that the acyclic class $\cJ\aprod \cK$ is a congruence.
The inclusion $i:\Cong\cE\subseteq \Acyclic\cE$ preserves suprema, since it has a right adjoint $\decinfty-$ by \cref{thm:adjCongAcyclic}.
The monoidal unit $\All$ belongs to $\Cong\cE$, since it is a congruence.
We have proved that the poset $\Cong\cE$ is a sub-$\otimes$\=/frame of $\Acyclic\cE$.
\end{proof}

We shall see in \cref{acyclicproductofsmallcong} that the congruences of small generation also define a $\otimes$\=/frame.

\begin{remark}[Non-monoidal adjoints]
\label{rem:not-monoidal}
Neither adjoint to the monoidal morphism $i:\Cong\cE\subseteq \Acyclic\cE$ of \cref{thm:cprod} is monoidal.
The left adjoint $(-)\cong$ is always oplax monoidal: for any two acyclic classes $\cA$ and $\cB$, we have $(\cA\aprod\cB)\cong \subseteq \cA\cong \aprod \cB\cong$.
A counter-example to the equality is given in $\cS$ using the class $\Acmap\cS$ of acyclic maps of \cref{prop:acyclic-map-nilp}.
We have seen there that $\Acmap\cS\aprod \Acmap\cS = \Iso$, 
thus $(\Acmap\cS\aprod \Acmap\cS)\cong = \Iso$.
On the other hand, $\Acmap\cS\cong = \All$, since the only two congruences of $\cS$ are $\Iso$ and $\All$ and $\Acmap\cS\subsetneq \Iso$.
This shows that the inclusion $(\cA\aprod\cB)\cong  \subseteq \cA\cong \aprod \cB\cong$ can be strict.
(The same example can be used to show that $(-)\cong:\Class\cE\to \Cong\cE$ is not be monoidal.)

The right adjoint $\decinftyname$ is always lax monoidal:
we have $\decinfty\cA \aprod \decinfty\cB \subseteq \decinfty{\cA \aprod \cB}$ for any two acyclic classes $\cA$ and $\cB$.
A counter-example to the equality is given by the class $\Surj$ in the logos of parametrized spectra $\PSp$ (the category of spectra bundles over varying spaces).
There we have $\decinfty\Surj \aprod \decinfty\Surj = \Conn\infty \aprod \Conn\infty = \Iso$
(use \cref{rem:nilpotent-object:PSp,prop:nilrad-in-hrad} below).
On the other hand, we have $\decinfty {\Surj\aprod \Surj} = \decinfty {\Conn 0} = \Conn\infty \not=\Iso$.
This proves that the inclusion $\decinfty \cA\aprod\decinfty\cB \subseteq \decinfty{\cA \aprod \cB}$ can be strict.
\end{remark}

\begin{remark}
\label{rem:cprod}
A consequence of \cref{thm:cprod} is the equality $(\cJ\pp\cK)\ac=(\cJ\pp\cK)\cong$ for any two congruences, showing that the product of congruences can be defined by the naive formula (even though, as we have just seen in \cref{rem:not-monoidal}, the functors $(-)\cong:\Class\cE\to \Cong\cE$ and $(-)\cong:\Acyclic\cE\to \Cong\cE$ are not monoidal).
The inclusion $(\cJ\pp\cK)\ac\subseteq(\cJ\pp\cK)\cong$ is always true and the converse is given by the fact that $(\cJ\pp\cK)\ac$ is a congruence containing $\cJ\pp\cK$.
\end{remark}

\begin{remark}
\label{remarkdivisioncongruence}
Beware that the inclusion $\Cong\cE \subseteq \Acyclic\cE$ does not preserve the division operation.
More precisely, if $\cJ$ and $\cK$ are congruences  on $\cE$, then the {\it acyclic division} $\cJ\div \cK$ may not be a congruence.
The {\it congruence division} of $\cK$ by $\cJ$ is the congruence $\decinfty{\cJ\div \cK}$.
This is a general formula for monoidal functors, but let us verify it in our case:
if $\cI$ is another congruence, the condition $\cI\subseteq \decinfty{\cJ\div \cK}$ is equivalent to the condition $\cI\subseteq \cJ\div \cK$ by \cref{thm:adjCongAcyclic}, which is equivalent to the condition $\cI\aprod \cJ\subseteq \cK$.
\end{remark}

\begin{corollary}[Computation formula]
\label{cor:computation-prod-cong}
For two classes of maps $\Sigma$ and $\Tau$, the following formula holds:
\[
\Sigma\cong\aprod \Tau\cong
\ =\ 
(\Sigma\diag\pp \Tau\diag)\ac\,.
\]
\end{corollary}
\begin{proof}
Direct from $\Sigma\cong =(\Sigma\diag)\ac$ (\cref{sigmaac=sigmacong}) and formula \eqref{eq:ac=monoidal}.
\end{proof}

Recall that a class of maps $\Sigma$ in a logos is said to be a lex generator (\cref{def:lex-generator}) if the class $\Sigma\ac$ is a congruence if and only if $\Delta(\Sigma)\subseteq \Sigma\ac$ (\cref{caractlexgenerator}).
\Cref{cor:computation-prod-cong} shows that any class $\Sigma\diag$ is a lex generator.

\begin{corollary}[Lex generators]
\label{boxprodlexgen}
The product $\Sigma\pp \Tau$ of two lex generators $\Sigma$ and $\Tau$ in a logos is again a lex generator.
\end{corollary}
\begin{proof} 
By formula \eqref{eq:ac=monoidal}, we have 
$(\Sigma\pp\Tau)\ac =\Sigma\ac\aprod\Tau\ac$.
The acyclic class $\Sigma\ac\aprod \Tau\ac$ is a congruence by \cref{mainthmCongruence}, since the classes $\Sigma\ac$ and $\Tau\ac$ are congruences by hypothesis.
\end{proof}

The acyclic product and acyclic division restrict to congruences of small generation.
\begin{proposition}
\label{acyclicproductofsmallcong}
If $\cJ$ and $\cK$ are congruences of small generation on a logos, then so are the congruences $\cJ\aprod \cK$ and $\decinfty{\cJ\div \cK}$.
In particular, the subposet ${\sf Cong_{sg}}(\cE)\subseteq \Cong\cE$ is a sub-$\otimes$\=/frame.
\end{proposition}
\begin{proof}
By \Cref{smallintersmallgencong}, it is enough to show that $\cJ\aprod \cK$ and $\decinfty{\cJ\div \cK}$ are of small generation as acyclic classes.
The congruence $\cJ\aprod \cK$ is of small generation by \cref{cor:computation-prod-cong}.
The acyclic division $\cJ\div \cK$ is of small generation by \cref{thm:aprod}, and the congruence division $\decinfty{\cJ\div \cK}$ is thus of small generation by \cref{lem:sg-dec}.
\end{proof}

\begin{examples}
\label{ex:cprod}
\begin{exmpenum}

\item We saw already that the class $\All$ is the unit, hence $\All \aprod \cK = \cK$, fro every congruence $\cK$.

\item The class $\Iso$ is an absorbing element in the monoid $\Cong\cE$: $\Iso \aprod \cK = (\Iso \pp \cK)\ac = \Iso$.

\item\label{ex:cprod:power}
Since $\All$ is the unit of $\Cong\cE$, the power of any congruence $\cK$ define a decreasing sequence
\[
\dots
\ \subseteq\ 
\cK ^{\aprod\, 3}
\ \subseteq\ 
\cK \aprod \cK
\ \subseteq\ 
\cK\,.
\]

\item\label{ex:cprod:power-object}
If $X$ is an object in a logos, \cref{cor:computation-prod-cong} implies that the $k$\=/th power of the congruence $\idl X := \{X\to 1\}\cong$ is generated by the family
\[
\big\{ \Delta^{n_1} X \pp \dots \pp \Delta^{n_k} X \ |\ (n_1,\dots,n_k)\in \NN^k
\big\}\,.
\]
This will be useful in our definition of nilpotent objects \cref{def:n-nilpotent}.

\item\label{ex:cprod:epic-cong}
If $\cK$ is an epic congruence on a logos, then so is the class $\cK\aprod \cA$ for every acyclic class $\cA$.
To see it, we use that an acyclic class $\cC$ is an epic congruence if and only if $\cC=\susp\cC = \Surj\aprod \cC$ (\cref{coTCongrchar}).
Then the condition $\cK = \Surj\aprod \cK$ implies
$\cK \aprod \cA =(\Surj\aprod \cK)\aprod \cA = \Surj \aprod (\cK \aprod \cA)$.

In particular, $\Conn\infty\aprod \cA$ is always an epic congruence.
This congruence is closely related to the infinite suspension of $\cA$, which is $\suspinfty\cA := \bigcap_n\suspn n \cA = \bigcap_n (\Conn n\aprod \cA)$.
Since $\suspn n \cA\subseteq\Conn n$, we have $\suspinfty\cA\subseteq \Conn\infty$ and $\suspinfty\cA$ is also a fixed point of the suspension.
We have an inclusion
\[
\Conn\infty\aprod \cA
\ =\ 
\big(\bigcap_n \Conn n\big)\aprod \cA)
\ \subseteq\ 
\bigcap_n (\Conn n\aprod \cA) 
\ =\ 
\suspinfty \cA\,.
\]
We do not know if this is an equality.

\end{exmpenum}	
\end{examples}

\subsection{The frame of Grothendieck topologies}

In this section, we construct a frame structure on the poset $\MCong\cE$ of monogenic congruences of a logos $\cE$.
In \cite{ABFJ:GT}, we have shown that the poset $\MCong\cE$ is equivalent to the poset of extended Grothendieck topologies on $\cE$ (\cref{def:GT}).
For the results of this section, it will be more convenient to work with monogenic congruences rather than Grothendieck topologies.
However, it can also be useful to think about the results in terms of extended Grothendieck topologies, particularly on a presheaf logos $\cE=\PSh C$ where they are in bijection with Grothendieck topologies on the category $C$

The frame structure was proved to exist in \cite[Theorem 3.0.1]{ABFJ:GT} using classical results of 1-topos theory.
We provide here an independent proof.
We are going to construct the frame structure explicitly by embedding $\MCong\cE$ into the $\otimes$\=/frame $\Acyclic\cE$ of acyclic classes constructed in \cref{thm:aprod}.
We will see then that the embedding $\MCong\cE\subseteq\Cong\cE$ is also a morphism of $\otimes$\=/frames.

Recall from \cref{rem:triple-adjunction}
the triple adjunction relating monogenic congruences and acyclic classes
\[
\begin{tikzcd}
\MCong\cE
\ar[from=rrr,"{(-)\mono}" description]
\ar[rrr, shift right = 3,"{(-)\vee \Surj}"',hook]
\ar[rrr, shift left = 3,"{\sf can.}", hook]
&&& \Acyclic\cE\,.
\end{tikzcd}
\]
The right adjoint to the inclusion extract the monogenic part of an acyclic class, which is defined by the same formula as the monogenic part of congruences: $\cA\mono = (\cA\cap \Mono)\ac$.
An acyclic class $\cA$ is a monogenic congruence if and only if $\cA=\cA\mono$.

\begin{theorem}
\label{thm:gtop} 
The poset of monogenic congruences $\MCong\cE$ is a frame. 
The inclusion $ \MCong \cE \subseteq  \Acyclic\cE$ and its right adjoint
$(-)\mono$ are morphisms of $\otimes$\=/frames.
\[
\begin{tikzcd}
\MCong\cE
\ar[from=rrr,"{(-\cap \Mono)\ac}", shift left = 1.5]
\ar[rrr, shift left = 1.5,"{\sf can.}", hook]
&&& \Acyclic\cE
\end{tikzcd}
\]
\end{theorem}
 
\begin{proof}
The canonical inclusion $\MCong\cE\subseteq  \Acyclic\cE$ preserves suprema, since it has a right adjoint $(-)\mono$.
Since an acyclic class $\cA$ is a monogenic congruence if $\cA=\cA\mono$, we can use \cref{prop:idem-mono-cong} to show that the inclusion preserves the acyclic product.
The class of all maps $\All$ is a monogenic congruence, since $\All =\{i:0\to 1\}\ac$.
Thus the poset $\MCong\cE$ is a sub-$\otimes$\=/frame of the $\otimes$\=/frame $\Acyclic\cE$.
By \cref{prop:idem-mono-cong}, we see also that $\cA\aprod \cA = \cA$ for every $\cA \in \MCong\cE$ (using that $\cA=\cA\mono$).
By \cref{cor:charac-frames}, this proves that the $\otimes$\=/frame $\MCong\cE$ is a frame. 
Let us see that the map $(-)\mono:\Acyclic\cE\to\MCong\cE$ is a morphism of $\otimes$\=/frames.
It preserves suprema, since it has a right adjoint.
It preserves the product by \cref{prop:idem-mono-cong}, and it preserves the unit $\All\mono=\All$, since the class $\All$ is generated by the monomorphism $0\to 1$.
\end{proof}

\begin{remark}
\label{rem:idempotent-cong}
Monogenic congruences are thus idempotent for the acyclic product.
The existence of idempotent ideals in commutative algebra that are not generated by idempotent elements suggest to us that there should be idempotent congruences that are not monogenic, 
but we do not have a counter-example.
\end{remark}

\begin{remark}
\label{rem:idem-ideal}
\Cref{thm:gtop} suggests to compare monogenic congruences to idempotent ideals in commutative algebra.
The simplest examples of such ideals are those generated by idempotent elements (corresponding geometrically to the inclusion of some connected components).
Their logos analogues seem to be the congruences generated by subterminal objects $U\to 1$ (\cref{exmp:mono-congruence2}) since such maps are idempotent for the pushout product.
There are more sophisticated examples of idempotent ideals though,
typically the maximal ideals in rings of continuous functions with real values on a manifold.
\end{remark}

The following diagram~\eqref{monoidal-morphisms} recalls the various adjoint functors between acyclic classes, congruences, and monogenic congruences (aka Grothendieck topologies).
The convention is the following.
The left adjoints are put above their right adjoint.
The top, middle and bottom triangles all commute.
The functors indicated in solid lines are the monoidal ones (the others can be proven not to be).
The top ${\sf can.}$ functor is monoidal by \cref{thm:cprod}, the one on the right side is by \cref{thm:gtop}.
The monoidal structures for the functors on the left side are the content of \cref{topcongquantale}.

\begin{equation}
\label{monoidal-morphisms}
\begin{tikzcd}
\Cong\cE 
\ar[from=rrrrrr, shift right = 3,"(-)\cong"', dashed]
\ar[rrrrrr, hook, "{\sf can.}" description]
\ar[from=rrrrrr, shift left = 3,"\decinfty-", dashed] 
&&&&&&
\Acyclic\cE
\\
\\
\\
&&&\MCong\cE
\ar[rrruuu, "{\sf can.}", hook, shift left = 3, bend right]
\ar[from=rrruuu, "(-)\mono" description, bend left] 
\ar[rrruuu, "\hcrad-\,=\,-\vee \Surj"', hook, shift right = 3, bend right, dashed]
\ar[llluuu, "{\sf can.}"', hook', shift right = 3, bend left]
\ar[from=llluuu, "(-)\mono" description, bend right] 
\ar[llluuu, "\hrad-\,=\,\decinfty{-\vee \Surj}", hook', shift left = 3, bend left, dashed]
\end{tikzcd}	
\end{equation}

\begin{corollary}
\label{topcongquantale}
Then the inclusion $\MCong\cE\to\Cong\cE$ and its right adjoint $(-)\mono$ are morphisms of $\otimes$\=/frames.
\end{corollary}
\begin{proof}
The morphism $\MCong\cE\to \Cong\cE$ is a morphism of $\otimes$\=/frames, since the following triangle commutes 
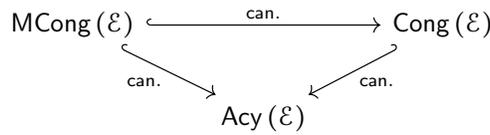
\[
\begin{tikzcd}
\MCong\cE \ar[rr, "{\sf can.}", hook] \ar[dr,"{\sf can.}"', hook]
&&\Cong\cE \ar[dl,"{\sf can.}", hook']  \\
& \Acyclic\cE
\end{tikzcd}
\]
and both the left and right maps are morphisms of $\otimes$\=/frames by 
\cref{thm:gtop,mainthmCongruence}. 
Similarly, the map $(-)\mono:\Cong\cE\to \MCong\cE$ is a morphism of $\otimes$\=/frames, since the following triangle commutes
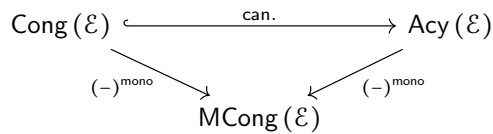
\[
\begin{tikzcd}
\Cong\cE \ar[rr, "{\sf can.}", hook] \ar[dr,"(-)\mono"']
&&\Acyclic\cE \ar[dl,"(-)\mono"]  \\
&\MCong\cE
\end{tikzcd}
\]
and the map $(-)\mono:\Acyclic\cE\to \MCong\cE$ is a morphism of $\otimes$\=/frames by \cref{thm:gtop}.
\end{proof}

\begin{remark}
\label{rem:formula-gtop}
In \cite{ABFJ:GT} we established isomorphisms between 
the poset $\MCong\cE$ of monogenic congruences (or acyclic classes), 
the poset $\RCong\cE$ of hyper-radical congruences, 
the poset $\CTop(\cE)$ covering topologies,
and the poset $\GTop(\cE)$ of Grothendieck topologies.
Each of these inherits a frame structure from \cref{thm:gtop}.
We make the product structure explicit in each case.

\begin{remenum}

\item\label{rem:formula-gtop:hrad} 
In hyper-radical congruences,
the unit is the class $\All$,
and the product of two hyper-radical congruences $\cK$ and $\cK'$ is their intersection as classes of maps.
It satisfies the relations
\[
\cK\cap \cK'
\ =\ 
\hrad{(\cK\mono \aprod (\cK')\mono)}
\ =\ 
\hrad{(\im\cK \pp \im{\cK'})\ac}
\,.
\]
In particular, for two arbitrary congruences $\cJ$ and $\cJ'$, we can derive the classical formula
\[
\hrad\cJ \cap \hrad{\cJ'} = \hrad{\cJ \aprod \cJ'}\,.
\]

\item\label{rem:formula-gtop:gtop}
In (extended) Grothendieck topologies,
the unit is the class $\Mono$,
and the product of two topologies $\cG$ and $\cG'$ is their intersection as classes of maps.
It satisfies the relations
\[
\cG\cap \cG'
\ =\ 
(\cG\ac \aprod (\cG')\ac)\cap \Mono
\ =\ 
(\cG \pp \cG')\ac\cap \Mono
\ =\ 
(\cG \pp \cG')\gtop\,,
\]
where, for a class of monomorphisms $\Sigma$, $\Sigma\gtop=\Sigma\ac\cap \Mono$ is the smallest Grothendieck topology containing it \cite[Corollary 3.3.9]{ABFJ:GT}.
Moreover, for any two classes of monomorphisms $\Sigma$ and $\Tau$, we have
\[
\Sigma\gtop\cap \Tau\gtop
\ =\ 
(\Sigma \pp \Tau)\gtop\,.
\]

\item\label{rem:formula-gtop:covt} 
In covering topologies,
the unit is the class $\All$,
and the product of two covering topologies $\cC$ and $\cC'$ is their intersection as classes of maps.
It satisfies the relations
\[
\cC\cap \cC'
\ =\ 
(\cC\mono \aprod (\cC')\mono)\vee \Surj
\ =\ 
(\im\cC \pp \im{\cC'})\ac\vee \Surj
\,.
\]
\end{remenum}
\end{remark}

\begin{remark}
\label{rem:prod=inter}
In all examples of \cref{rem:formula-gtop}, the product structure is always the intersection of classes of maps.
This is not a priori true for monogenic congruences.
For two monogenic congruences $\cA$ and $\cB$, their infimum in $\MCong\cE$ can computed as $\cA\wedge\cB = (\cA\cap \cB)\mono$, where the intersection is taken in $\Acyclic\cE$ (or $\Cong\cE$).
We do not know if the intersection of two monogenic congruences is again monogenic (\ie whether the simpler formula $\cA\wedge\cB = \cA\cap \cB$ holds).
But this holds in some examples, see below.
\end{remark}

\begin{examples}
\label{ex:aprod-gtop}
Let $U$ and $V$ be subterminal objects of a logos $\cE$. 
We denote by $U\cup V$ and $U\cap V$ their union and intersection.
Let $\tX$ be the topos dual to $\cE$. 
\begin{exmpenum}
\item\label{ex:aprod-gtop:open}
The monogenic congruence $\{U\to 1\}\ac$ corresponds to the open quotient $\cE\to \cE\slice U$. 
We denote by $\tU\to \tX$ the corresponding open inclusion of topoi.

Using 
\cref{lem:cocartgapmapmono,rem:formula-gtop:gtop} we see that $\{U\to 1\}\ac \aprod \{V\to 1\}\ac = \{U\cup V \to 1\}\ac$ is the infimum of the monogenic quotients $\cE\to \cE\slice U$ and $\cE\to \cE\slice V$. 
Geometrically, this corresponds to the union $\tU \cup \tV$ of the two open subtopoi.
Using the fact that the quotient $\cE\to \cE\quotient\{U\to 1\}\ac$ has an explicit reflector $(-)^U:\cE\to \cE$, it is possible to verify that $\{U\to 1\}\ac \aprod \{V\to 1\}\ac = \{U\to 1\}\ac \cap \{V\to 1\}\ac$.

\item\label{ex:aprod-gtop:closed}
The monogenic congruence $\{0\to U\}\ac$ corresponds to the closed quotient complementary to the open quotient $\cE\to \cE\slice U$.
We denote by $\tX\!\setminus\!\tU\to \tX$ the corresponding closed inclusion of topoi.
The product $\{0\to U\}\ac \aprod \{0\to V\}\ac = \{0\to U\cap V\}\ac$ is the infimum of the monogenic quotients and, geometrically, the union $\tX\!\setminus\!\tU\,\cup\,\tX\!\setminus\!\tV = \tX\setminus (\tU\cap\tV)$ of the two closed subtopoi of $\tX$.
The quotient $\cE\to \cE\quotient\{0\to U\}\ac$ has an explicit reflector $(-)\star U :\cE\to \cE$, which can be used to show that $\{0\to U\}\ac \aprod \{0\to V\}\ac = {\{0\to U\}\ac \cap \{0\to V\}\ac}$.
\end{exmpenum}
\end{examples}

\subsection{Nilpotent objects, nilradical, frame reflection}
\label{sec:nilradical}

The purpose of this section is to introduce the notion of a nilpotent map in a logos.
It will be used to describe the Goodwillie logos of $n$\=/excisive functors as being freely generated by a nilpotent pointed object (\cref{thm:Goodwillie:tower}~\eqref{thm:Goodwillie:tower:nilpotent}).
As it turns out, the naive definition of a nilpotent map that one can derive from the pushout product (some finite power is invertible) is not the one needed for the universal property of $n$\=/excisive functors.
We call this first notion {\it weak nilpotence}.
The correct notion of a nilpotent map $f$ is obtained by demanding that the family of all diagonals of $f$ be a weakly nilpotent family.

We apply the notion of nilpotence to define a notion of nilradical for congruences  (\cref{def:nilradical}).
We prove the nilradical of a congruence sits inside its hyper-radical (\cref{prop:nilrad-in-hrad}) and that the poset of nilradical congruences is a frame which is the frame reflection of the $\otimes$\=/frame $\Cong\cE$.
We also define a notion of nilradical for acyclic classes and show some connections with the notion of acyclic maps in the sense of algebraic topology (see \cref{prop:acyclic-map-nilp}).

\medskip
We will define nilpotency for a class of maps $\Sigma$ (rather than a single map) and relatively to another class of map $\Theta$.
The absolute case will be when $\Theta=\Iso$. 
The relative notion will be useful to define the nilradical of a congruence.

\begin{definition}[Weak $n$\=/nilpotence]
\label{def:weak-n-nilpotent}
We fix class of maps $\Theta$.
A class of maps $\Sigma$ is said to be {\it weakly $\Theta$\=/nilpotent of order $n$}
if $\Sigma^{\pp n+1} \subseteq \Theta$.

When $\Theta=\Iso$, we shall simply say {\it weakly $n$\=/nilpotent}.
A map $f$ is {\it weakly $\Theta$\=/nilpotent of order $n$} if $\Sigma=\{f\}$ is a weakly $\Theta$\=/nilpotent class of order $n$.
An object $X$ is {\it weakly $\Theta$\=/nilpotent of order $n$} if $f=X\to 1$ is a weakly $\Theta$\=/nilpotent map of order $n$.
A class of maps, a map or an object is called {\it weakly $\Theta$\=/nilpotent of finite order} if it is weakly $\Theta$\=/nilpotent of order $n$ for some $n$.
\end{definition}

\begin{remark}
\label{rem:nil-via-aprod}
When $\cA$ and $\cB$ are acyclic classes (in particular congruences), the following equivalences hold by definition of the acyclic product (\cref{deflocalbox}):
\[
\cB \pp \dots \pp \cB
\ \subseteq\ 
\cA
\quad\Leftrightarrow\quad
(\cB \pp \dots \pp \cB)\ac
\ \subseteq\ 
\cA
\quad\Leftrightarrow\quad
\cB \aprod \dots \aprod \cB
\ \subseteq\ 
\cA\,.
\]
This shows that $\cB$ is weakly $\cA$\=/nilpotent of order $n$ if and only if $\cB^{\aprod n+1}\subseteq \cA$ (\ie the pushout product can be replaced by the acyclic product).
\end{remark}

\begin{examples}
\begin{exmpenum}

\item\label{rem:weak-nilpotent-object}
Explicitly, an object $X$ is weakly $n$\=/nilpotent if its $(n+1)$\=/st join power ${X\join \dots \join X}$ is contractible.

\item In the logos of spaces $\cS$, the classifying space of a perfect group is weakly 1\=/nilpotent. More generally, in a a hyper-reduced logos $\cE$ (\eg $\cS$), \cref{prop:acyclic-map-nilp} says that the class $\Acmap\cE$ of acyclic maps of \cref{examplemodality:acyclic} is weakly 1\=/nilpotent.

\item\label{ex:weakly-nilpotent-object:PSp} In the logos $\PSp$ of parametrized spectra (see \cref{sec:Goodwillie-tower}), all spectra are weakly 1\=/nilpotent objects.
In fact something stronger is true: for any two spectra $E$ and $F$, 
we have $E\join F = 1$ since the cartesian square
\[
\begin{tikzcd}
E\times F \ar[d]\ar[r] & E\ar[d] \\
F\ar[r] & 1
\end{tikzcd}
\]
is always a pushout in $\Sp$ and also in $\PSp$.
\end{exmpenum}	
\end{examples}

\begin{definition}[$n$\=/nilpotence]
\label{def:n-nilpotent}
We fix class of map $\Theta$.
A class of maps $\Sigma$ is {\it $\Theta$\=/nilpotent of order $n$}
(when $\Theta=\Iso$, we shall simply say {\it $n$\=/nilpotent}) 
if $\Sigma\diag$ is weakly $\Theta$\=/nilpotent of order $n$.
A map $f$ is {\it $\Theta$\=/nilpotent of order $n$} if $\{f\}\diag = \{\Delta^k f\,|\, k\in \NN\}$ is $\Theta$\=/nilpotent of order $n$.
An object $X$ is {\it $\Theta$\=/nilpotent of order $n$} if $f=X\to 1$ is a $\Theta$\=/nilpotent map of order $n$.
A class of maps, a map or an object is called {\it $\Theta$\=/nilpotent of finite order} if it is $n$\=/nilpotent for some $n$.
\end{definition}

\begin{examples}
\begin{exmpenum}

\item\label{rem:nilpotent-object}
Explicitly, a 1\=/nilpotent object is an object $X$ such that all pushout products ${\Delta^m X\pp \Delta^n X}$ are isomorphisms.
In informal terms, this means that binary join products of arbitrary path spaces of $X$ are contractible: $\Omega^m X\join \Omega^n X = 1$.
The definition of nilpotent objects of order $n$ is similar with $n+1$ pushout products involved (see \cref{ex:cprod:power-object}).

\item In the logos of spaces $\cS$, the only nilpotent object is the point.
In fact, we shall see in \cref{cor:nilrad-surj} that all nilpotent objects in a hyper-reduced logos must be contractible.

\item\label{rem:nilpotent-object:PSp} 
In the logos $\PSp$ of parametrized spectra, the nilpotent objects are exactly the spectra. They are all 1\=/nilpotent.
In fact, something stronger is true: the pushout product of any two maps in $\Sp\subseteq\PSp$ is always an isomorphism.
To see this, we use that the inclusion $\Sp \subseteq \PSp$ preserve limits and pushouts, hence  pushout products.
Then in the stable category, the pushout product of any two maps is always invertible, because it is the cocartesian gap map of a cartesian square.

\end{exmpenum}	
\end{examples}

If $\Sigma\diag\subseteq \Sigma$, then $\Sigma$ is $\Theta$\=/nilpotent of order $n$ if and only if it is weakly $\Theta$\=/nilpotent of order $n$.
Since any congruence $\cK$ satisfies $\cK\diag\subseteq\cK$, this show that the two notions of nilpotence coincide on congruences.

\begin{lemma}
\label{lem:wnil=nil}
For any class of maps $\Theta$,
a congruence $\cK$ is $\Theta$\=/nilpotent of order $n$ if and only if it is weakly $\Theta$\=/nilpotent of order $n$.
\end{lemma}

\begin{remark}
When the class $\Sigma$ is acyclic or a congruence, the nilpotence conditions of \cref{def:weak-n-nilpotent,def:n-nilpotent} becomes a property of $\Sigma$ in the $\otimes$\=/frame of acyclic classes or congruences:
If $\Theta$ and $\cA$ are acyclic classes, $\cA$ is $\Theta$\=/nilpotent if and only if $\cA^{\pp\ac n+1}\subset\Theta$.
If $\Theta$ and $\cK$ are congruences, $\cK$ is $\Theta$\=/nilpotent if and only if $\cK^{\pp\ac n+1}\subset\Theta$.
\end{remark}

We now define the notions of nilradical of an acyclic class and of a congruence.
These are two different notions which do not coincide when applied to a congruence.
The definitions are a special case of \cref{def:radical-elt} applied to the $\otimes$\=/frames of acyclic classes and congruences.

\begin{definition}[Acyclic nilradical, weak nilpotence]
\label{def:acrad}
An acyclic class $\cA$ is said to be {\it nilradical} if, for every acyclic class $\cB$, it holds that $\cB^{\pp\ac n}\subseteq \acrad\cA \Ra \cB \subseteq \acrad\cA$
(or, equivalently, if, for every class of maps $\Sigma$, $\Sigma^{\pp n}\subseteq \cA \Ra \Sigma \subseteq \cA$).
The {\it acyclic nilradical} of an acyclic class $\cA$ is defined as the smallest nilradical acyclic class $\acrad \cA$ containing $\cA$.
A class, a map or an object is called {\it weakly $\cA$\=/nilpotent} if it belongs to $\acrad \cA$.
When $\cA=\Iso$, we shall simply say {\it weakly nilpotent} instead of weakly $\Iso$\=/nilpotent and call $\acrad\Iso$ the {\it acyclic radical of the logos} $\cE$ (we shall also use the notation $\acrad 1$ for the acyclic radical).
\end{definition}

\begin{definition}[Congruence nilradical, nilpotence]
\label{def:nrad}
\label{def:nilradical}
A congruence $\cK$ is said to be {\it nilradical} if, for every congruence $\cW$, it holds that 
$\cW^{\pp\ac n}\subseteq \nrad\cK \Ra \cW \subseteq \nrad\cK$ (or, equivalently, if, for every class of maps $\Sigma$, $(\Sigma\diag)^{\pp n}\subseteq \cK \Ra \Sigma \subseteq \cK$).
The {\it nilradical} of a congruence $\cK$ is defined as the smallest nilradical congruence  $\acrad \cK$ containing $\cK$.
A class, a map or an object is called {\it $\cK$\=/nilpotent} if it belong to $\nrad \cK$.
When $\cK=\Iso$, we shall simply say {\it nilpotent} instead of $\Iso$\=/nilpotent and call
$\nrad\Iso$ the {\it nilradical of the logos} $\cE$ (we shall also use the notation $\nrad 1$ for the nilradical).
\end{definition}

\begin{remark}
\label{rem:non-finite-weak-nilpotent}
Notice that not every weakly $\cA$\=/nilpotent map needs to be weakly $\cA$\=/nilpotent of finite order (infinite coproducts of weakly nilpotent maps of different order might not be weakly nilpotent of finite order).
The same remark is true for congruences.
This is a difference with rings where the nilradical is simply the subset of all nilpotent elements.
For acyclic classes and congruences, such a construction of the nilradical needs to be iterated transfinitely (see \cref{rem:radical-soa}).
This makes the nilradical difficult to describe concretely.
In particular, it is not immediately clear that the nilradical of a congruence (or an acyclic class) of small generation is again of small generation. 
\end{remark}

\begin{examples}
\begin{exmpenum}

\item The congruence nilradical of $\cS$ is the class $\nrad 1=\Iso$ (this will be a consequence of \cref{cor:nilrad-surj}).

\item\label{ex:nilradical:2} From \cref{prop:acyclic-map-nilp}, we know that the class of acyclic maps in $\cS$ sits inside the acyclic nilradical $\acrad 1$. (Could this inclusion be an equality?)
Since we just saw that $\nrad  1=\Iso$, this shows that the congruence nilradical can be strictly smaller than the acyclic radical, $\nrad  1\subsetneq \acrad 1$.
\end{exmpenum}	
\end{examples}

The inclusion of $\otimes$\=/frames $\Cong\cE\to \Acyclic\cE$ provides some relations between the two notions.
Recall that its right adjoint is $\decn\infty-:\Acyclic\cE\to \Cong\cE$.

\begin{proposition}
\label{prop:nilradical}
There exists a natural commutative square of $\otimes$\=/frame morphisms (in solid arrows, dashed arrows are the associated right adjoints):
\[
\begin{tikzcd}
\sqrt{\Cong\cE}
\ar[rrr, "{\sf can.}"', hook, shift right = 1.5, dashed]
\ar[from=rrr, "\nrad-"', shift right = 1.5] 
\ar[dd, "{\acrad-}"']
&&&
\Cong\cE
\ar[dd, hook', shift right=1.5]
\\
\\
\sqrt{\Acyclic\cE}
\ar[rrr, "{\sf can.}"', hook, shift right = 1.5, dashed]
\ar[from=rrr, "\acrad-"', shift right = 1.5] 
\ar[uu, shift right=4, "\decinfty-"', dashed]
&&&
\Acyclic\cE\,.
\ar[uu, shift right=1.5, "\decinfty-"', dashed]
\end{tikzcd}
\]
This diagram provides the following relations:
\begin{enumerate}
\item\label{prop:nilradical:commut-right} For every nilradical acyclic class $\cA$, the congruence $\decn\infty\cA$ is nilradical.
\item\label{prop:nilradical:commut-left} For every congruence $\cK$, $\acrad{\nrad\cK}=\acrad\cK$.
\item\label{prop:nilradical:mate-right} For every acyclic class $\cA$, there is a canonical inclusion $\nrad{\decn\infty\cA}\subseteq\decn\infty{\acrad\cA}$.
\item\label{prop:nilradical:inclusion} For every congruence $\cK$, there is a canonical inclusion $\nrad{\cK}\subseteq \acrad\cK$ (which can be strict by \cref{ex:nilradical:2}).
\end{enumerate}
\end{proposition}
\begin{proof}
The first statement and \eqref{prop:nilradical:commut-right} come from \cref{prop:rad:square-adjoints} applied to the morphism of $\otimes$\=/frames $\Cong\cE\to \Acyclic\cE$.
\eqref{prop:nilradical:commut-left} is the commutation of left adjoints in  \cref{prop:rad:square-adjoints}.
\eqref{prop:nilradical:mate-right} is the mate of that square.
And \eqref{prop:nilradical:inclusion} is deduced from \eqref{prop:nilradical:mate-right} by the  right adjoint property of $\decn\infty-$ and the relation $\decn\infty\cK=\cK$ for a congruence.
\end{proof}

\begin{remark}
\label{rem:formula-nil=dec-acnil}
The morphism $\acrad-:\sqrt{\Cong\cE}\to \sqrt{\Acyclic\cE}$ is induced by the inclusion $\Cong\cE\to \Acyclic\cE$, but it is not clear that it is itself an inclusion.
This is the case 
if and only if $\decn\infty{\acrad\cN}=\cN$ is true for every nilradical congruence $\cN$, 
if and only if $\decn\infty{\acrad\cK}=\nrad\cK$ for every congruence $\cK$.
A related problem is whether the inequality \eqref{prop:nilradical:mate-right} is in fact an 
equality $\nrad{\decn\infty\cA}=\decn\infty{\acrad\cA}$.
When $\cA=\cK$ is a congruence, this would give also the relation $\nrad{\cK}=\decn\infty{\acrad\cK}$.
\end{remark}

For an acyclic class $\cA$, let us say that a map $f$ is an {\it $\cA$\=/covering} if $\im f\in \cA$.
The class of $\cA$\=/covering maps is $\hcrad\cA:=\cA\vee \Surj$, it is the {\it covering topology} associated to $\cA$ in the sense of \cite[Theorem 3.3.11]{ABFJ:GT}.
For the purpose of this paper, we shall called $\hcrad\cA$ the {\it covering radical} of $\cA$.
Recall from \cref{def:hrad} that, for a congruence $\cK$, a map $f$ is a {\it $\cK$\=/hypercovering} if, for all $n$, $\im{\Delta^nf}\in \cK$.
The class of all $\cK$\=/hypercoverings is denoted $\hrad \cK$ and called the {\it hyper-radical} of $\cK$.

\begin{lemma}
\label{lem:power-image}
\begin{corenum}
\item Let $\cA$ be an acyclic class.
If $f^{\pp n} \in \cA$ for some $n$, then $f \in \hcrad\cA$.

\item Let $\cK$ be a congruence and $f$ a map.
If $\{\Delta^k f\,|\,k\in \NN\}^{\pp n} \in \cK$ for some $n$, then $f\in \hrad\cK$.    
\end{corenum}
\end{lemma}
\begin{proof}
We prove the first assertion.
Let $f$ be a map such that $f^{\pp n}\in \cA$.
Then $\im {f ^{\pp n}}\in \cA$ by \cref{lem:acyclic-image}.
From \cref{lem:image-box} we have $\im {f^{\pp n}} = \im f ^{\pp n}$, thus $\im f ^{\pp n}\in \cA$.
From \cref{lem:im-box-im} we have $\im f$ is a base change of $\im f ^{\pp n}$, thus $\im f\in \cA$ since $\cA$ is closed by base change.
This shows that $f\in \hcrad\cA$ by definition of $\hcrad\cA$.

We prove the second assertion.
Let $f$ be a map such that $(\{f\}\cong)^{\pp n}\subseteq \cK$.
The condition implies in particular that for every $k$, $(\Delta^k f)^{\pp n}$ is in $\cK$.
Applying the first part with $\cA=\cK$, we get that all images of diagonals $\im{\Delta^k f}$ are in $\cK$.
Thus $f\in \hrad\cK$ by definition of $\hrad\cK$.
\end{proof}

\begin{proposition}
\label{prop:nilrad-in-hrad} 
For every acyclic class $\cA$, the covering radical $\hcrad\cA=\Surj\vee\cA$ is a nilradical acyclic class.
In particular, we always have an inclusion
\[
\acrad\cA
\ \subseteq\,
\hcrad\cA\,.
\]
For every congruence $\cK$, the hyper-radical $\hrad\cK$ is a nilradical congruence.
In particular, we always have an inclusion
\[
\nrad \cK
\ \subseteq\,
\hrad \cK\,.
\]
\end{proposition}
\begin{proof}
\Cref{lem:power-image} shows that $\hcrad\cA$ is nilradical and contains $\cA$.
Then the result follows by definition of $\acrad\cA$.
The argument is the same for congruences.
\end{proof}

\begin{corollary}
\label{cor:nilrad-surj}
The acyclic class $\Surj$ is acyclic radical and the congruence $\Conn \infty$ is nilradical.
Every weakly nilpotent map is surjective and every nilpotent map is \oo connected.
\[
\acrad\Iso
\ \subseteq\ 
\Surj=\hcrad\Iso
\qquad\text{and}\qquad
\nrad\Iso
\ \subseteq\ 
\Conn\infty=\hrad\Iso
\]
\end{corollary}
\begin{proof}
The acyclic class $\Surj = \hcrad\Iso$ is the covering radical of the acyclic class $\Iso$.
The congruence $\Conn\infty = \hrad\Iso$ is the hyper-radical of the congruence $\Iso$.
Then the result follows \cref{prop:nilrad-in-hrad}.
\end{proof}

\begin{remark}
\label{rem:nilrad-in-hrad}
At the time of writing this, we do not have an example where the inclusions of \cref{cor:nilrad-surj} are strict.
In the congruence case, we believe that this should be the case in the logos $\S {X\connected\infty}$ classifying \oo connected objects of \cref{exmp:mono-congruence:oo-conn-obj}.
In \cref{exmp:rad-congruence2}, we have seen that $\hrad\Iso = \Conn \infty \not=\Iso$ (since $X\connected\infty$ is \oo connected) and we suspect that $\nrad\Iso = \Iso$ (there should be no non-trivial nilpotent map).
\end{remark}

The frame reflection of a $\otimes$\=/frame is defined in \cref{prop:frame-refl}.

\begin{theorem}
\label{thm:frame-reflection}
The poset of nilradical acyclic classes is the frame reflection $\sqrt{\Acyclic\cE}$ of $\Acyclic\cE$.
The poset of nilradical congruences is the frame reflection $\sqrt{\Cong\cE}$ of $\Cong\cE$.
\end{theorem}
\begin{proof}
Both statements are derived from \cref{prop:frame-refl}.
By definition, the weak nilradical $\acrad \cA$ of an acyclic class coincides with its radical $\sqrt \cA$ in the sense of the reflection \cref{prop:frame-refl}.
Same for the nilradical of a congruence.
\end{proof}

\begin{remark}
\label{rem:radical-relations}
The commutative square of \cref{prop:nilradical} can be extended to include the previous comparisons with the covering radical and hyper-radical.
Since $\Cong\cE\to\RCong\cE$ is a $\otimes$\=/frame morphism into an actual frame, it factors through the frame reflection, which is $\sqrt{\Cong\cE}$ by \cref{thm:frame-reflection}.
For simplicity, we have named $\hrad-$ the frame morphism $\sqrt{\Cong\cE}\to \RCong\cE$.
It has a left adjoint given by the ``nilpotent core'' $(-)\nil := \nrad{(-)\mono}$.

We do not know whether the frame reflection $\Cong\cE\to \sqrt{\Cong\cE}$ has a left adjoint.
If it is the case, by \cref{prop:idem-coreflection}, its image must the frame of idempotent congruences.
In keeping with \cref{rem:nilrad-in-hrad}, we expect $\sqrt{\Cong\cE}$ to be strictly larger than $\RCong\cE$, since the contrary would imply the equalities $\nrad\cK =\hrad\cK$.
It the adjoint exist, this would means that there are strictly more idempotent congruences than the monogenic ones.
(The problem is also related to \cref{rem:prod=inter}.)
Similar considerations hold for the acyclic classes.
Altogether, we get a diagram:
\[
\begin{tikzcd}
\RCong\cE
\ar[rrr, "(-)\nil", hook, shift left = 3]
\ar[rrr, "{\sf can.}"', hook, shift right = 3]
\ar[from=rrr, "\hrad-" description] 
\ar[dd, equal]
&&&
\sqrt{\Cong\cE}
\ar[rrr, "(-)\idem\,?", hook, shift left = 3, dotted]
\ar[rrr, "{\sf can.}"', hook, shift right = 3]
\ar[from=rrr, "\nrad-" description] 
\ar[dd, "{\acrad-}" description]
&&&
\Cong\cE
\ar[dd, hook']
\\
\\
\CTop(\cE)
\ar[rrr, "(-)\wnil", hook, shift left = 3]
\ar[rrr, "{\sf can.}"', hook, shift right = 3]
\ar[from=rrr, "\hcrad-" description] 
&&&
\sqrt{\Acyclic\cE}
\ar[rrr, "(-)\idem\,?", hook, shift left = 3, dotted]
\ar[rrr, "{\sf can.}"', hook, shift right = 3]
\ar[from=rrr, "\acrad-" description] 
\ar[uu, shift left=4, "?", dotted]
\ar[uu, shift right=4, "\decinfty-"']
&&&
\Acyclic\cE
\ar[uu, shift left=3, "(-)\cong"]
\ar[uu, shift right=3, "\decinfty-"']
\end{tikzcd}
\]
where left adjoints are vertically on the left of their right adjoint and horizontally on top of them.
The diagram of the rightmost adjoint commutes.
So does the diagram of their left adjoints.
The middle horizontal functors are quotients of $\otimes$\=/frames.
The middle vertical functors are morphism of $\otimes$\=/frames (as mentioned in \cref{rem:formula-nil=dec-acnil}, we do not know if they are both inclusions).
The existence of the dotted arrows is open.
\end{remark}

\subsection{Functoriality}
\label{sec:naturality}

In this section, we prove that the $\otimes$\=/frames $\Acyclic\cE$, $\Cong\cE$, and $\MCong\cE$ of \cref{thm:aprod,thm:cprod,thm:gtop} are natural in $\cE$ and define functors on the category $\Logos$.

\begin{proposition}
\label{inducedmorphismofacyclicclasses}
\label{prop:naturality}
Let $\phi:\cE\to \cF$ be a morphism of logoi.
\begin{propenum}
\item\label{prop:naturality:ac}
The morphism 
\[
\phi(-)\ac:\Acyclic\cE\stto \Acyclic\cF
\]
is a morphism of $\otimes$\=/frames, 
with right adjoint the map $\phi^{-1}:\Acyclic\cF\to \Acyclic\cE$.

\item\label{prop:naturality:cong}
The restriction of $\phi(-)\ac$ along $\Cong\cE\to \Acyclic\cE$ takes values in $\Cong\cF$ and defines a morphism of $\otimes$\=/frames
\[
\phi(-)\ac:\Cong\cE\stto \Cong\cF
\]
with right adjoint given by the restriction of $\phi^{-1}$.

\item\label{prop:naturality:gtop}
The restriction of $\phi(-)\ac$ along $\MCong\cE\to \Acyclic\cE$ takes values in $\MCong\cF$ and defines a morphism of frames
\[
\phi(-)\ac:\MCong\cE\stto \MCong\cF
\]
with right adjoint given by the restriction of $\phi^{-1}(-)\mono$.
	
\end{propenum}
Moreover all the morphisms of the diagram \eqref{monoidal-morphisms} (monoidal or not) commute with $\phi(-)\ac$.
\end{proposition}

\begin{proof}
\noindent (1)
The adjunction $\phi(-)\ac\dashv \phi^{-1}$ is from \cref{prop:transport-acyclic}.
Let us show that $\phi(-)\ac$ is a morphism of $\otimes$\=/frames.
The map preserves suprema, since it is a left adjoint.
Let us show that it preserves the product of acyclic classes.
If $u$ and $v$ are two maps in $\cE$, then we have $\phi(u\pp v)\simeq \phi(u)\pp \phi(v)$ since the functor $\phi$ preserves cartesian products and pushouts.
Thus, if $\cA$ and $\cB$ are acyclic classes in $\cE$, 
the acyclic closure of the class $\phi(\cA\pp \cB)$ is equal to the acyclic closure of the class $\phi(\cA)\pp \phi(\cB)$.
Thus, $\phi(\cA\pp \cB)\ac=(\phi(\cA)\pp \phi(\cB))\ac$.
It follows that 
\[
\phi(\cM\aprod \cN)\ac
\ =\ 
\phi((\cM\pp\cN)\ac)\ac
\ =\ 
\phi(\cM\pp\cN)\ac
\ =\ 
(\phi(\cM)\pp \phi(\cN))\ac
\ =\ 
\phi(\cM)\ac\aprod \phi(\cN)\ac.
\]
It remains to verify that $\phi(-)\ac$ preserves the $\otimes$\=/frame unit $\All$.
By \cref{ex:acyclicgene:-1}, we have $\All=\{i\}\ac$, where $i$ is the map $0\to 1$. 
Thus,
\[
\phi(\All)\ac=\phi(\{i\}\ac)\ac=\phi(\{i\})\ac=\{i\}\ac =\All\,.
\]

\smallskip
\noindent (2)
Using \cref{lem:lex-gen} (or \cref{prop:transport-cong}), we need only to prove that $\phi(-)\ac$ preserves the product of congruences. 
This follows from (1).

\smallskip
\noindent (3)
It is sufficient to show that $\phi(-)\ac$ commutes with the inclusion $\MCong\cE\to \Acyclic\cE$ and its left adjoint $(-)\mono:\Acyclic\cE\to \MCong\cE$.
If $\cA = \Sigma\ac$, then $\phi(\cA)\ac = \phi(\Sigma)\ac$ by \cref{inducedmorphismofacyclicclasses}. 
If $\Sigma$ a class of monomorphisms, then so is $\phi(\Sigma)$ since logos morphisms preserve monomorphisms.
Thus the acyclic class $\phi(\cA)\ac$ is monogenic and this shows that $\phi(-)\ac$ commutes with the inclusion $\MCong\cE\to \Acyclic\cE$.

Let us see now that $\phi(-)\ac$ commutes with $(-)\mono$.
We need to show that $\phi(\cA\mono)\ac = (\phi(\cA)\ac)\mono$.
For any class of maps $\Sigma$, we have $\im{\phi(\Sigma)} = \phi(\im\Sigma)$ since
the logos morphism $\phi$ preserves image factorizations.
Thus $\im{\phi(\cA)}\ac = \phi(\im\cA)\ac$.
By \cref{prop:mono-part}, the left side is 
$\im{\phi(\cA)}\ac=\im{\phi(\cA)\ac}\ac=(\phi(\cA)\ac)\mono$.
By \cref{inducedmorphismofacyclicclasses}, the right side is 
$\phi(\im\cA)\ac = \phi(\im\cA\ac)\ac= \phi(\cA\mono)\ac$.
Hence the equality $\phi(\cA\mono)\ac = (\phi(\cA)\ac)\mono$.
\end{proof}

We shall say that a morphism of $\otimes$\=/frames $\phi:P\to Q$ is a {\it closed quotient} $\phi:P\to Q$ if its right adjoint $\phi_*:Q\to P$ is fully faithful (= if it is a monomorphism) and if the image of the map $\phi_*:Q\to P$ is the subset $P\upslice {\phi_*(0)}=\{x\in P \ | \ x\geq \phi_*(\bot)\}$.
In this case, the map $\phi:P\to Q$ is isomorphic to the map $\phi_*(\bot)\vee(-):P\to P\upslice {\phi_*(\bot)}$.
The terminology is suggested by the fact that, applied to a morphism of frames, the notion recovers the frame quotients dual to the inclusion of a closed sublocale.

\begin{proposition}
\label{prop:naturality-facto}
Let $\phi:\cE \to \cF$ be a quotient of logoi. 
\begin{propenum}
\item\label{prop:naturality-facto:ac}
The morphism of $\otimes$\=/frames $\phi(-)\ac:\Acyclic\cE\to \Acyclic\cF$ is a closed quotient and it is isomorphic to the morphism $\cK_\phi \vee(-): \Acyclic\cE \to \Acyclic\cE\upslice{\cK_\phi }$.

\item\label{prop:naturality-facto:cong}
The morphism of $\otimes$\=/frames $\phi(-)\ac:\Cong\cE\to \Cong\cF$ is a closed quotient and it is isomorphic to the morphism $\cK_\phi \vee(-): \Cong\cE \to \Cong\cE\upslice{\cK_\phi }$.

\item\label{prop:naturality-facto:gtop}
The morphism of frames $\phi(-)\ac:\MCong\cE\to \MCong\cF$ is a closed quotient and it is isomorphic to the morphism $\cK_\phi \vee(-): \MCong\cE \to  \MCong\cE\upslice{\cK_\phi\mono}$.
\end{propenum}

\end{proposition}
\begin{proof}
We prove the first statement. 
The equivalence $\Acyclic\cF = \Acyclic\cE\upslice{\cK_\phi }$ is from \cref{prop:transport-acyclic}
and the description of the map follows from \cref{overquantale}.
The statements (2) and (3) are proved the same way using \cref{prop:transport-cong}.
\end{proof}

\begin{remark}
In \cref{prop:naturality-facto:gtop}, when $\cK_\phi$ is an epic congruence, we have $\cK_\phi\mono=\Iso$ and we get an isomorphism $\MCong\cE = \MCong\cF$.
This was already remarked in \cite[Proposition 4.3.4 and Remark 4.3.5]{ABFJ:GT}.
\end{remark}

\section{Application to Goodwillie Calculus}
\label{sec:Goodwillie}

In this section, we explain how to use the acyclic product of congruences to recover examples of Goodwillie towers \cite{G03}.
\Cref{sec:modality-tower} introduces the general setting of {\it towers of modalities} in a logos and lists some examples.
\Cref{sec:completion-tower} studies the {\it completion towers}, which are the towers of left-exact modalities generated by the powers of a given congruence for the acyclic product.
\Cref{sec:Goodwillie-tower} proves that Goodwillie towers are instances of completion towers.
Although we shall not prove this here, Weiss' orthogonal calculus \cite{Weiss:OC} also provides an example of a completion tower, see \cite{ABFJ:TM}.

\subsection{Modality towers}
\label{sec:modality-tower}

We fix a decreasing sequence of acyclic classes in a logos $\cE$: 
\[
\dots
\ \subseteq\ 
\cA_2
\ \subseteq\ 
\cA_1
\ \subseteq\ 
\cA_0
\]
We will complete this sequence on the left by considering the acyclic classes 
\[
\cA_{\infty} := \bigcap_n \cA_n\,,
\qquad\text{and}\qquad
\cA_{\infty}\mono := \Big(\bigcap_n \cA_n\Big)\mono\,.
\]
Altogether, we have a nested sequence of acyclic classes
\[
\cA_{\infty}\mono
\quad\subseteq\quad
\cA_{\infty}
\quad\subseteq\quad
\dots
\quad\subseteq\quad
\cA_2
\quad\subseteq\quad
\cA_1
\quad\subseteq\quad
\cA_0\,.
\]
If for all $n\in \NN$ the $\cA_n$ are of small generation, then so are $\cA_{\infty}$ (by \cref{smallintersmallgen}) and $\cA_{\infty}\mono$ (by \cref{prop:mono-part}).
Thus, by \cref{prop:acyclic2modality}, the pair $(\cA_n,\rforth \cA_n=\rorth \cA_n)$ is a modality on $\cE$ for every $n\in \NN\cup\{\infty\}$, as well as $\big(\cA_{\infty}\mono,\rforth {(\cA_{\infty}\mono)}=\rorth {(\cA_{\infty}\mono)}\big)$.
In this case, we denote by $\cE\quotient \cA_n$ the presentable category obtained by inverting all maps of $\cA_n$ in $\cE$ (see~\cref{def:quotient-presentable}).
This category is equivalent to the subcategory of $\cE$ spanned by the $\cA_n$\=/local objects $X$ ($X\to 1\in \rorth \cA_n$).
We obtain a tower of presentable categories
\begin{equation}
\label{eqn:modality-tower}
\cE
\stto
\cE\quotient\cA_{\infty}\mono
\stto
\cE\quotient\cA_{\infty}
\stto
\lim_n \cE\quotient\cA_n
\stto
\dots
\stto
\cE\quotient\cA_1
\stto
\cE\quotient\cA_0\, ,
\end{equation}
where we have added the limit of the finite stages. Except for the limit stage all stages of the tower are localizations of $\cE$. The functors $\cE\to\cE\quotient\cA_{\infty}\to\lim_n\cE\quotient\cA_n$ factorize the composite functor into a localization followed by a conservative functor.
If the $\cA_n$ are congruences, then the categories $\cE\quotient \cA_n$  are quotients of logoi, see~\cref{thm:bij-congruence-lexloc}.
This special case will be the focus of the following sections.

For a map $f:X\to Y$ in $\cE$ the inclusion relations between the $(\cA_n,\rorth \cA_n)$ induce a tower of factorizations:
\vspace{-1em}
\begin{equation}
\label{eq:tower-facto}
\begin{tikzcd}
X \ar[r]
	\ar[rrrrrrrr,bend left=40,"\in\cA_0"']
	\ar[rrrrrr,bend left=30,"\in\cA_1"]
	\ar[rrrr,bend left=20,"\in\cA_2"]
&{}
& \dots 
& \ar[r]
& P_2(f)
	\ar[rr,"\in\cA_1\quotient\cA_2"']
	\ar[rrrrrr, bend right=40,"\in \rorth \cA_2"]
&& P_1(f)
	\ar[rr,"\in\cA_0\quotient\cA_1"]
	\ar[rrrr, bend right=30,"\in \rorth \cA_1"]
&& P_0(f)
	\ar[rr,"\in\rorth \cA_0"]
&& Y\,.
\end{tikzcd}
\end{equation}
Moreover, such towers are stable under base change along a map $Y'\to Y$ since all $(\cA_n,\rorth \cA_n)$ are modalities.
The reflection $\cE \to \cE\quotient \cA_n$ is given by sending $X$ to the factorization $X\to P_nX\to 1$ with respect to the respective modality. 

\begin{definition}[Tower of modalities and layers]
\label{def:tower-modalities}
\label{def:layer}
We shall call a sequence $(\cA_n,\rorth \cA_n)_{n\ge 0}$ as above and by abuse the associated tower~\eqref{eqn:modality-tower}  a {\it tower of modalities} on the logos $\cE$. We put $\cA_{-1}:= \All$. For $n\geq 0$ we define $\cA_{n-1}\quotient\cA_n := \cA_{n-1}\cap \rorth \cA_n$, and call it the {\it $n$\=/th layer} of the tower. The 0-th layer, or {\it ground layer}, is $\cA_{-1}\quotient\cA_0 := \rorth \cA_0$. 
\end{definition}

Each layer $\cA_{n-1}\cap \rorth \cA_n$ is a local class of maps, since the classes $\cA_{n-1}$ and $\rorth \cA_n$ are local
\cite[Proposition 3.6.5]{ABFJ:GBM}.

\begin{examples}
Let us list here a few examples that are not given by congruences.
\begin{exmpenum}
\item\label{ex:tower:Postnikov}
(Postnikov tower) The tower of modalities associated to the sequence 
\[
\dots
\ \subseteq\ 
\Conn 1
\ \subseteq\ 
\Conn 0
\ \subseteq\ 
\Conn {-1}=\Surj
\]
is the Postnikov tower.
In that case, we have $\bigcap_n\Conn n = \Conn \infty$ (which is a congruence), and $(\bigcap_n\Conn n)\mono=\Iso$.
The quotient $\cE\quotient\Conn n$ is equivalent to the subcategory $\cE\truncated n$ of $n$\=/truncated objects.
The associated tower of quotients is the tower of ``Postnikov sections''  from~\cite[Section 6.4.5]{Lurie:HTT}.
None of the categories $\cE\quotient\Conn n$ are logoi, but their limit $\lim_n\cE\quotient\Conn n$ is a logos and the canonical morphism $\phi:\cE\to \lim_n\cE\quotient\Conn n$ is a logos morphism \cite[Proposition A.7.3.4]{Lurie:SAG}.
Moreover, the congruence $\bigcap_n\Conn n = \Conn \infty$ is the class of maps inverted by $\phi$, and the factorization
\[
\cE\stto \cE\quotient\Conn\infty \stto \lim_n\cE\quotient\Conn n
\]
is the quotient factorization of $\phi$ (\cref{def:quotient-facto}).
In other words, the hyper-reduction is the Postnikov separation.
We have also $\Conn\infty\mono = \Iso$ and the first morphism $\cE\to \cE\quotient\Conn\infty\mono$ is an equivalence.

\item\label{ex:tower:decalage}
(Decalage tower) 
The decalages of an acyclic class $\cA$~(\cref{def:decalage}) yields a sequence 
\[
\dots
\ \subseteq\ 
\decn 3 \cA
\ \subseteq\ 
\decn 2 \cA
\ \subseteq\ 
\dec \cA
\ \subseteq\ 
\cA\ 
\]
of acyclic classes.
In this case, $\bigcap_n \decn n\cA = \decinfty\cA $ is a congruence by \cref{lem:decalage}.
Moreover, because $(-)\mono$ is a right adjoint (\cref{thm:mono-cong}), we have  $\decinfty\cA\mono = \bigcap_n \decn n \cA\mono$. But $\decn n \cA\mono = \cA\mono$ by \cref{rem:decalage-v-suspension} and hence the monogenic part $\decinfty\cA\mono =\cA\mono$. It is possible to prove that the limit of the tower is always a logos.
The factorization 
\[
\cE
\stto
\cE\quotient\cA\mono
\stto
\cE\quotient\decinfty\cA
\stto
\lim_n\cE\quotient\decn n \cA
\]
is the image triple factorization of the canonical morphism $\cE\to \lim_n\cE\quotient\decn n \cA$ (\cref{def:quotient-facto}).

When $\cA=\Surj$ we recover the example of the Postnikov tower, since $\decn {n+1} \cA=\Conn n$, compare~\cref{ex:decalage:2}.
When $\cA=\cK$ is a congruence, the decalage tower is essentially constant since $\dec \cK=\cK$.
The only nontrivial stage is $\cE\quotient\cK\mono$ and all the other are equal to $\cE\quotient\cK$.

\item\label{ex:tower:suspension}
(Suspension tower) Any acyclic class $\cA$ has an associated tower of iterated suspensions 
\[
\suspn n \cA
\ =\ 
\Surj^{\pp n} \aprod \cA
\ =\ 
\Conn {n-2} \aprod \cA
\]
\[
\dots
\ \subseteq\ 
\suspn 3 \cA
\ \subseteq\ 
\suspn 2 \cA
\ \subseteq\ 
\susp \cA
\ \subseteq\ 
\cA\,.
\]
The intersection $\bigcap_n\suspn n \cA$ is hard to compute in general since $\aprod$ does not commute with intersections.
The monogenic part, however is $\bigl(\bigcap_n\suspn n \cA\bigr)\mono = \bigcap_n\suspn n \cA\mono =\Iso$, since $(-)\mono$ is a right adjoint and $\suspn n \cA\subseteq\Surj$. It is possible to prove that the limit of the tower is always a logos.

When $\cA\subseteq \Surj$ this tower coincides with the decalage tower by \cref{thm:suspension-acyclic}. For example, this happens in the logos $\cS$, since $\All$ is the only acyclic class which is not contained in $\Surj$.

The suspension tower is also considered in~\cite[Section 7]{bousfield1994localization} in the case $\cE=\cS$ and $\cA=\{W\to 1\}\ac$ for a space $W$.

\item\label{ex:tower:power}
(Power tower) We have seen in \cref{ex:aprod:power} that the acyclic powers of an acyclic class $\cA$ form a decreasing sequence $\dots \subseteq\cA^3\subseteq\cA^2\subseteq \cA$.
If $\cA$ is of small generation, so are all the $\cA^n$ (\cref{thm:aprod}), as well as the intersection $\bigcap_n\cA^n$  (\cref{smallintersmallgen}), and its monogenic part $(\bigcap_n\cA^n)\mono$ (\cref{prop:mono-part}).
Therefore we get a tower of modalities. 
We do not know if the limits of such towers are always logoi.
For $\cA=\Surj$, we recover the Postnikov tower of \cref{ex:tower:Postnikov} since $\Surj^{\pp n} =\Conn {n-2}$.
\end{exmpenum}
\end{examples}

\subsection{Completion towers}
\label{sec:completion-tower}

We now study the modality tower associated to the powers of a congruence $\cK$.
For a congruence $\cK$ we denote simply by $\cK^n$ its $n$\=/th power for the acyclic product.

\begin{proposition}
\label{prop:filt-mono}
For any congruence $\cK$ in a logos $\cE$ there is the sequence of nested congruences
\begin{equation}
\label{eqn:filtration}
\cK\mono
\quad\subseteq\quad
\bigcap_n\cK^n
\quad\subseteq\quad
\dots
\quad\subseteq\quad
\cK^3
\quad\subseteq\quad
\cK^2
\quad\subseteq\quad
\cK\, ,
\end{equation}
all sharing the same monogenic part.
Moreover, if $\cK$ is of small generation, then all these congruences are also of small generation.
\end{proposition}

\begin{proof}
All the $\cK^n$ have the same monogenic part by~\cref{prop:idem-mono-cong}.
The equalities 
$(\bigcap_n\cK^n)\mono\ =\ \bigcap_n(\cK^n)\mono\ =\ \cK\mono$ 
follow since the functor $(-)\mono$ is right adjoint to the inclusion of monogenic congruences into congruences.
Small generation for the powers has been proved in \cref{acyclicproductofsmallcong}, and for $\bigcap_n\cK^n$ this is \cref{smallintersmallgencong}. 
\end{proof}

\begin{definition}
\label{def:tower}
We shall call the sequence \eqref{eqn:filtration} the {\it $\cK$\=/adic filtration} of $\cE$.
If the congruence $\cK$ is of small generation, we obtain an associated tower of logoi:
\begin{equation}
\label{eqn:tower}
\cE
\stto
\cE\quotient\cK\mono
\stto
\cE\quotient\bigcap_n\cK^n
\stto
\lim_n \cE\quotient\cK^n
\stto
\dots
\stto
\cE\quotient\cK^3
\stto
\cE\quotient\cK^2
\stto
\cE\quotient\cK\,,
\end{equation}
where we have also introduced the limit logos of the tower of $\cE\quotient\cK^n$ (which can be computed in $\CAT$).
Except for the limit logos, all stages of the tower are quotients of the logos $\cE$ (see~\cref{prop:sep-comp-facto} for a statement about the limit logos).

We shall call the tower \eqref{eqn:tower} the {\it completion tower} of the congruence $\cK$ (or of the quotient $\cE\to \cE\quotient \cK$) and use the vocabulary of \cref{table:tower} to talk about the various stages of the tower.
There, we have added two ``underground" stages to the tower to account for the hyper-radical and nilradical of $\cK$ (\cref{def:nilradical,def:radical-logos}).
\end{definition}

\begin{table}[htbp]
\caption{Completion tower}
\label{table:tower}

\begin{center}
\renewcommand{\arraystretch}{1.6}
\begin{tabular}{|cc|cc|cc|}
\hline
\multicolumn{2}{|c|}{Inclusion of topoi $\tZ\subseteq \tX$} 
& \multicolumn{2}{c|}{Quotient of logoi $\cE\to \cE\quotient\cK$} 
& \multicolumn{2}{c|}{Congruence $\cK$} \\
\hline
& $\tX$ 
& $\cE$ 
&
& $\Iso$
& \\
amplification of $\tZ$ 
& $\tZ\ampli$
& $\cE\quotient\cK\mono$ 
& hypo-separation stage 
& $\cK\mono$
& monogenic part
\\
flat neighborhood of $\tZ$
& $\tZ^\flat$
& $\cE\quotient\bigcap_n\cK^n$
& separation stage 
& $\bigcap_n\cK^n$
& separation part
\\
formal neighborhood of $\tZ$
& $\tZ^\wedge$ 
& $\lim_n\cE\quotient\cK^n$ 
& completion stage 
&
& --
\\
\vdots & \vdots & \vdots & \vdots & \vdots &\\
second neighborhood of $\tZ$ 
& $\tZ\exc 2$ 
& $\cE\quotient\cK^3$ 
& second stage 
& $\cK^3$
& finite
\\
first neighborhood of $\tZ$
& $\tZ\exc 1$ 
& $\cE\quotient\cK^2$ 
& first stage 
& $\cK^2$
& powers
\\
& $\tZ$ 
& $\cE\quotient\cK$ 
& ground stage 
& $\cK$
&
\\
nilreduction of $\tZ$ 
& $\tZ^\nred$ 
& $\cE\quotient\nrad\cK$ 
& nilreduction stage 
& $\nrad \cK$
& nilradical 
\\
hyper-reduction of $\tZ$ 
& $\tZ\hred$ 
& $\cE\quotient\hrad\cK$ 
& hyper-reduction stage 
& $\hrad \cK$
& hyper-radical 
\\
\hline
\end{tabular}
\end{center}
\end{table}

\begin{remark}
\label{rem:lim-nil}
In general, there is no relation between the tower associated to $\cK$ and that of $\hrad\cK$.
A counter-example is given by $\cE=\S {X\connected\infty}$ (see \cref{exmp:rad-congruence2}), for which the tower of $\cK=\Iso$ is trivial but for which we shall see in \cref{sec:Goodwillie-tower} that the tower of $\hrad\Iso=\Conn\infty$ is the Goodwillie tower.
We believe that the towers of $\cK$ and that of $\nrad\cK$ have the same limit, but the question is open.
\end{remark}

Beware that the logos morphisms $\cE\to \lim_n \cE\quotient\cK^n$ is not a quotient in general.
More precisely, we have the following result.
\begin{proposition}[Separation--completion factorization]
\label{prop:sep-comp-facto}
The factorization
\[
\cE\stto \cE\quotient\cK\mono \stto \cE\quotient\bigcap_n\cK^n \stto \lim_n \cE\quotient\cK^n
\]
is the quotient triple factorization (\cref{sec:triple-facto}) of the logos morphism $\cE\to \lim_n \cE\quotient\cK^n$.
\end{proposition}
\begin{proof}
Limits of logoi are computed in $\CAT$.
Thus, a map in $\cE$ is inverted in $\lim \cE\quotient\cK^n$ if and only if it is inverted in all the $\cE\quotient\cK^n$.
This shows that the  congruence of $\cE\to \lim \cE\quotient\cK^n$ is $\bigcap_n\cK^n$ and that the factorization 
$\cE \to \cE\quotient\bigcap_n\cK^n \to \lim \cE\quotient\cK^n$
is the quotient factorization of \cref{def:quotient-facto}.
We computed in \cref{prop:filt-mono} that $(\bigcap_n\cK^n)\mono = \cK\mono$.
This implies the rest of the statement.
\end{proof}

\begin{remark}[Analogy with differential calculus]
\label{rem:triple-facto-tower}
In keeping with \cref{rem:triple-facto}, the triple factorization of \cref{prop:sep-comp-facto} compares nicely with what happens in differential calculus.
Let $A$ be the ring of differentiable functions $\RR\to \RR$, and $\mathfrak m$ the maximal ideal of functions vanishing at $0\in \RR$.
The completion of $A$ at $\mathfrak m$ is the ring $A^\wedge_{\mathfrak m} := \lim_n A/{\mathfrak m}^n = \RR\lsem X\rsem$ of formal power series in one variable.
The kernel of the morphism $A\to \lim_n A/{\mathfrak m}^n$ is the intersection $\bigcap_n\mathfrak m ^n$.
This is the ideal of flat functions at 0 (functions with zero Taylor tower at 0).
Let $\mathfrak m\loc\subseteq \bigcap_n\mathfrak m ^n$ be the ideal of functions vanishing in a neighborhood of 0.
Then the quotient $A/\mathfrak m\loc$ is the local ring $A_{\mathfrak m}$ at 0.
The factorization 
\[
A
\stto
A/\mathfrak m\loc = A_{\mathfrak m}
\stto
A/\bigcap_n\mathfrak m ^n
\stto
\lim_n A/{\mathfrak m}^n\,.
\]
is analogous to that of \cref{prop:sep-comp-facto}.
The rest of the completion tower is translated easily.
\end{remark}

\begin{remark}
The inclusion $\cK\mono\subseteq \bigcap_n\cK^n$ is a priori strict,
and the quotient $\cE\quotient\cK\mono$ is a priori different from $\cE\quotient\bigcap_n\cK^n$.
Recall that $\cK\mono$ and $\hrad\cK$ are dual constructions in the sense of the triple adjunction of \cref{thm:mono-cong}; this motivated our choice of terminology for the {\it hypo}-separation stage.
\end{remark}

\begin{remark}[Monogenic collapse]
If the congruence $\cK$ is monogenic, we have $\cK^n=\cK$ by \cref{thm:gtop} and the whole tower collapse into the single quotient map $\cE\to \cE\quotient\cK$.
Geometrically, this means the ample subtopoi (dual to monogenic quotients) contain their infinitesimal neighborhoods.
Since every congruence is monogenic in $n$\=/topoi, for all $n<\infty$, this also explain why completion towers are invisible in these theories.
\end{remark}

\medskip
Given a quotient $\cE\to \cE\quotient\cK$, recall the notation $\phi\mono:\cE\to \cE\quotient\cK\mono$ and $\cK\epi:=\phi\mono(\cK)\ac$ from \cref{def:epic,sec:triple-facto}, and that $\cE\quotient\cK = (\cE\quotient\cK\mono)\quotient \cK\epi$.
The following result shows that the whole tower only depends on $\cE\quotient\cK\mono$ and $\cK\epi$.

\begin{theorem}[Structure of completion towers]
\label{thm:tower}
Let $\cK$ be a congruence of small generation in a logos $\cE$.
We put $\cK^\infty:=\bigcap_n\cK^n$.
\begin{enumerate}
\item\label{thm:tower:1} 
For any $1\leq n\leq \infty$, the monogenic-epic factorization of $\cE\to\cE\quotient \cK^n$ is
\[
\cE \stto \cE\quotient\cK\mono \stto \cE\quotient\cK^n\,.
\]

\item\label{thm:tower:2}
For any $1\leq m\leq n\leq \infty$, the quotient $\cE\quotient \cK^n\to\cE\quotient\cK^m$ is epic.

\item\label{thm:tower:3} 
The tower of $\cE\to \cE\quotient\cK$ coincides with the tower of $\cE\quotient\cK\mono\to\cE\quotient\cK$. 
In other words, $\phi\mono$ induces equivalences 
\begin{align*}
\cE\quotient \cK^n &=	(\cE\quotient\cK\mono)\quotient(\cK\epi)^n\,,\quad\text{for all $n\geq 0$,}\\
\cE\quotient\nrad\cK &= (\cE\quotient\cK\mono)\quotient \nrad{\cK\epi}\,,\quad\text{and}\\
\cE\quotient\hrad\cK &= (\cE\quotient\cK\mono)\quotient \hrad{\cK\epi}\,.
\end{align*}

\item\label{thm:tower:4} 
The hyper-radical of all the stages but the limit one is the image of $\hrad\cK$.
The hyper-reduction of all the logoi of the tower is $\cE\quotient\hrad\cK$.

\end{enumerate}
\end{theorem}

\begin{proof}
\noindent\eqref{thm:tower:1}
By \cref{prop:filt-mono}, we have $(\cK^n)\mono = \cK\mono$.
The result follows by definition of the monogenic--epic factorization in \cref{sec:triple-facto}.

\smallskip
\noindent\eqref{thm:tower:2}
We apply \cref{lem:cancel-epic} to the quotients $\cE\quotient\cK\mono \to \cE\quotient \cK^n$ and $\cE\quotient \cK^n\to\cE\quotient\cK^m$.

\smallskip
\noindent\eqref{thm:tower:3}
To simplify notations, we put $\cE':=\cE\quotient\cK\mono$ and $\psi=\phi\mono:\cE\to \cE'$.
We have $\cK\epi=\psi(\cK)\ac$ and the statement will be proved if we show that
\begin{enumerate}[label=\roman*)]
\item $\psi(\cK^n)=(\psi(\cK)\ac)^n$	,
\item $\psi(\nrad\cK)=\nrad{(\psi(\cK)\ac)}$	, and
\item $\psi(\hrad\cK)=\hrad{\cK\epi}$	.
\end{enumerate}
The first equality is from \cref{prop:naturality:cong}.
This proves that the finite stages of the towers coincide, and therefore the limits as well.
The statement for the nilreduced parts is a direct consequence of \cref{prop:naturality-facto:cong,lem:image-radical}.
For the hyper-reduced parts we observe that in $\cE'$ we have $(\cK\epi)\mono = \Iso$.
Thus $\hrad{\cK\epi} = \hrad{\Iso}=\Conn\infty$ by \cref{lem:hypercovering-mono}.
By \cref{lem:hypercovering} we have $\hrad\cK = \psi^{-1}(\Conn\infty)$.
The bijection of \cref{prop:transport-cong:quotient-cong} implies $\psi(\hrad\cK)\ac = \Conn\infty$, which gives $\psi(\hrad\cK)\ac = \hrad{\cK\epi}$.

\smallskip
\noindent\eqref{thm:tower:4}
By definition, the hyper-reduction of $\cE\quotient\cK\mono$ is $\cE\quotient\hrad{\cK\mono}$.
Using \cref{lem:hypercovering-mono}, we get $\hrad{\cK\mono} = \hrad\cK$.
This shows that the hyper-radical of $\cE\quotient\cK\mono$ is the image of $\hrad\cK$
and that its hyper-reduction is $\cE\quotient\hrad\cK$.
Then the first statement follows from \eqref{thm:tower:2} and \cref{lem:transport-hyper-radical}, 
and the second statement from \cref{lem:hyper-reduction-epic}.
\end{proof}

\medskip

Our next result concerns the structure of layers of the completion tower.
Recall from \cref{def:layer} that we denote by 
$\cK^n\quotient\cK^{n+1}:=\cK^n\cap \rorth{(\cK^{n+1})}$ 
the $(n+1)$\=/st layer of the completion tower.
For any object $X\in \cE$, we shall denote by $(\cK^n\quotient\cK^{n+1})_X$ the full subcategory of $\cE\slice X$ spanned by maps $Y\to X$ in $\cK^n\quotient\cK^{n+1}$ and call it the {\it $(n+1)$\=/st layer at $X$}.
Notice that the $(n+1)$\=/st layer at the terminal object $1\in \cE$ can be defined as the fiber of the functor $\cE\quotient \cK^{n+1} \to \cE\quotient \cK^n$ at the terminal object:
\[
\begin{tikzcd}
(\cK^n\quotient \cK^{n+1})_1 \ar[r]\ar[d] \pbmark & \cE\quotient \cK^{n+1}\ar[d]\\
1\ar[r,"1"] & \cE\quotient \cK^n\,.
\end{tikzcd}
\]
In other words, the objects of $(\cK^n\quotient \cK^{n+1})_1$ are the objects $E\in \cE$ such that $E=P_{n+1}E$ and $P_nE = 1$.

The identity $1_X$ of $X$ is a terminal object in $(\cK^n\quotient\cK^{n+1})_X$, and an object of $(\cK^n\quotient\cK^{n+1})_X$ is said to be pointed if it equipped with a morphism from $1_X$.
A morphism of pointed object is a morphism preserving the morphism from $1_X$.
We denote by $(\cK^n\quotient\cK^{n+1})_X\pointed$ the category of pointed objects in $(\cK^n\quotient\cK^{n+1})_X$ and call it the {\it pointed $n$\=/th layer at $X$}.

For the ground layer $\cK^0\quotient\cK=\All\quotient\cK$, the category $(\All\quotient\cK)_X = (\cE\slice X)\quotient(\cK\slice X)$ is always a logos.
To describe the structure of the higher layers, we will need the following weakening of the notion of stable category \cite[Chapter 1]{Lurie:HA}.

\begin{definition}[Pre-stable category]
We shall say that a category is {\it pre-stable} if it has pullbacks and pushouts, and if a commutative square is cartesian if and only if it is cocartesian.
\end{definition}

\begin{examples}
\label{rem:unpointed-stable}
\begin{exmpenum}
\item Any stable category is pre-stable.
\item A pre-stable category is stable if and only if it is pointed.
If a pre-stable category has a terminal object, the associated category of pointed objects is a stable category.
\item Any slice or coslice of a stable category is a pre-stable category.
These examples have furthermore an initial and a terminal object, but they might not coincide.
\end{exmpenum}
\end{examples}

\begin{theorem}[Pre-stability of layers]
\label{thm:layer}
For every $n\geq 1$ and every $X\in \cE$, the $(n+1)$\=/st layer $(\cK^n\quotient\cK^{n+1})_X$ is a pre-stable category with a terminal object.
The pointed $(n+1)$\=/st layer $(\cK^n\quotient\cK^{n+1})_X\pointed$ is a stable category.
\end{theorem}
\begin{proof}
The result will be a consequence of the generalized Blakers--Massey theorem of \cite{ABFJ:GBM} in the line of its application to Goodwillie calculus in \cite{ABFJ:GC}.
By working in $\cE\slice X$, we can assume that $X=1$.
We put $\cF=\cE\quotient\cK^{n+1}$ and let $\psi:\cE\to \cF$ be the quotient map.
We consider the congruence $\cJ:=\psi(\cK^n)\ac$ in $\cF$.
Then $\cE\quotient\cK^n = \cF\quotient\cJ$.
Moreover, we have $\cJ^2 = \psi(\cK^{2n})\ac$, since $\psi(-)\ac$ is a morphism of $\otimes$\=/frames by \cref{prop:naturality:cong}.
By assumption $n\geq 1$, thus $2n\geq n+1$ and $\cK^{2n}\subseteq \cK^{n+1}$.
This shows that $\cJ^2=\Iso$ in $\cF=\cE\quotient\cK^{n+1}$.
Thus $\cF=\cF\quotient\cJ^2$ and the $(n+1)$\=/st layer 
$(\cK^n\quotient\cK^{n+1})_1$ 
is equivalent to 
$(\cJ\quotient\cJ^2)_1$.
This reduces to prove the result in the case $n=1$.

The inclusion $(\cJ\quotient \cJ^2)_1 = \{E\in \cF\,|\,E\to 1\in \cJ\}\subseteq \cF$ preserves pushouts and pullbacks, so all computations can be done in $\cF$.
The terminal object of $\cF$ is an object in $\cJ\quotient\cJ^2$, thus it has all finite limits.
To prove that $(\cJ\quotient\cJ^2)_1$ is a pre-stable category, we are left to show that a square is cartesian if and only if it is cocartesian.
We consider a commutative square in $(\cJ\quotient \cJ^2)_1$
\begin{equation}
\label[Square]{square:layer}
\begin{tikzcd}
F \ar[r,"g"]\ar[d,"f"'] & H \ar[d]\\
G\ar[r]& K\,.
\end{tikzcd}
\end{equation}
All the maps of the square are in $\cJ$ by the 3-for-2 property of congruences.
If $\cR:= \cJ^\perp$, then the pair $(\cJ,\cR)$ is a modality by \cref{examplemodality:3}.

We assume that the square \eqref{square:layer} is cocartesian.
We apply the generalized Blakers--Massey theorem \cite[Theorem 4.1.1]{ABFJ:GBM} in the logos $\cF$ with the modality $(\cJ,\cR)$.
Since congruences are closed under diagonal, we have $\Delta f\pp \Delta g\in \cJ\aprod \cJ = \Iso$.
Then the cartesian gap map of the square is also an isomorphism.
This proves that cocartesian square are cartesian.
The converse is proved the same way using the dual Blakers--Massey theorem \cite[Theorem 3.5.1]{ABFJ:GBM}.
This shows that the category $(\cJ\quotient \cJ^2)_1$ is pre-stable.
The second statement is deduced from the first one by \cref{rem:unpointed-stable}.
\end{proof}

\subsection{Goodwillie towers}
\label{sec:Goodwillie-tower}

We now apply our general formalism to show that Goodwillie towers are instances of completion towers.
We start with recollections on Goodwillie theory \cite{G03} necessary to state \cref{thm:Goodwillie:tower}.
The precise definitions and the proof of the theorem are postponed until \cref{sec:Goodwillie-proof}.

Let $\cS$ be the category of spaces and $\cS\pointed$ that of pointed spaces.
We consider the
{\it unpointed Goodwillie theory}
of finitary functors $\cS\to\cS$
(\ie functors $\Fin\to\cS$, where $\Fin$ is the category of finite spaces)
as well as the
{\it pointed Goodwillie theory}
of finitary functors $\cS\pointed\to\cS$
(\ie functors $\Fin\pointed\to\cS$, where $\Fin\pointed$ is the category of pointed finite spaces).
The category $\fun\Fin\cS$ is the free logos on one object $\S X$, and the category $\fun\Finp\cS$ is the free logos on a pointed object $\S {X\pointed}$ (see~\cref{def:free-logos}, \cref{exmp:mono-congruence:oo-conn-obj} and~\cite[Section 5.1]{ABFJ:HS}). 
This will be useful to provide classifying properties for the stages of the towers.

To simplify the exposition, we concentrate on the unpointed case. 
A functor $F:\Fin\to \cS$ is called {\it $n$\=/excisive} if it sends strongly cocartesian $(n+1)$\=/cubes to cartesian cubes 
(we shall recall definitions about cubes in the next section).
We denote by $\fun\Fin\cS\exc n\subseteq \fun\Fin\cS$ the full subcategory spanned by $n$\=/excisive functors.
Goodwillie proves that this subcategory has a reflection $P_n:\fun\Fin\cS\to \fun\Fin\cS\exc n$
and one can verify from his description that $P_n$ is a left-exact functor.
This describes $\fun\Fin\cS\exc n$ as a quotient of the logos $\fun\Fin\cS$.
We denote by $\cG_n=P_n^{-1}(\Iso)$ the  congruence of $P_n$.
When $n=0$, a functor is $0$\=/excisive if and only if it is constant, 
and the reflection $P_0:\fun\Fin\cS\to \cS$ is given by evaluating a functor $F:\Fin \to \cS$ at the contractible space, $P_0(F)=F(1)$.
We shall simply denote by $\cG$ the congruence $\cG_0$.

An $n$\=/excisive functor is always $(n+1)$\=/excisive.
The sequence of inclusions $\fun\Fin\cS\exc n\subseteq \fun\Fin\cS\exc {n+1}$ is dual to an decreasing sequence of congruences
\[
\dots
\ \subseteq\ 
\cG_2
\ \subseteq\ 
\cG_1
\ \subseteq\ 
\cG_0=\cG\,,
\]
and adjoint to a tower of quotients of logoi
\[
\fun\Fin\cS \stto \dots \stto \fun\Fin\cS\exc 2 \stto \fun\Fin\cS\exc 1\stto \fun\Fin\cS\exc 0=\cS\,,
\]
that we call the {\it Goodwillie tower} of the logos $\fun\Fin\cS$.
Internally to $\fun\Fin\cS$, this corresponds to a tower of left-exact idempotent monads (that we still denote $P_n$) providing, for each functor $F$,
a sequence of $n$\=/excisive approximations
\[
F \stto \dots \stto P_2(F) \stto P_1(F)\stto P_0(F)=F(1)\,,
\]
called the {\it Goodwillie tower} of the functor $F$ (which Goodwillie himself calls {\it Taylor tower} \cite{G03}).

\bigskip

We shall prove the following results (see the end of the next section).
We state them only in the unpointed case, but similar statements hold in the pointed case.
Recall that the logos $\fun\Fin\cS$ is the free logos $\S X$ on one generator, 
and that the generator $X$ is the canonical inclusion $\Fin\to \cS$ (\cref{def:free-logos}).

\begin{theorem}[Structure of Goodwillie tower]
\label{thm:Goodwillie:tower}
Recall that we put $\idl X = \{X\to 1\}\cong$.

\begin{enumerate}

\item\label{thm:Goodwillie:tower:power}
The Goodwillie tower of $\S X = \fun\Fin\cS$ is the completion tower of the congruence $\idl X$.
Precisely, $\cG=\idl X$ and $\cG_n=\cG^{n+1}$.

\item\label{thm:Goodwillie:tower:mono}
The hypo-separation stage of the Goodwillie tower is the logos $\S{X\connected\infty} = \S X \quotient \idl X \mono$ classifying \oo connected objects.

\item\label{thm:Goodwillie:tower:nilpotent}
The $n$\=/th stage of the Goodwillie tower is the logos $\S {X\exc n} := \S X \quotient \idl X ^{n+1} = \fun\Fin\cS\exc n$ classifying $n$\=/nilpotent objects.

\item\label{thm:Goodwillie:tower:hyper-reduction}
The hyper-radical of all the stages but the limit one is generated by the image of $X$.
The hyper-reduction of all the stages is the canonical morphism to $\cS$ induced by $P_0$.

\end{enumerate}
\end{theorem}

\begin{remark}[Limit of the tower]
\label{rem:pro-lim}
The limit $\S X^\wedge_{\idl X}$ of the Goodwillie tower is the category of countable sequences $\dots \to F_1\to F_0$ where each $F_n$ is an $n$\=/excisive functor $\Fin \to \cS$, and such that $F_{n+1}\to F_n$ is the reflection of $F_{n+1}$ in $n$\=/excisive functors.
We do not know what the limit of the tower classifies as a logos.
If the limit is considered as a pro-object in logoi, then it represents the functor $\mathsf{Nil}:\Logos\to\CAT$ sending a logos to its full subcategory of nilpotent objects of finite order.
\end{remark}

\begin{remark}[Separation stage]
The separation part of the tower is the quotient of $\S X$ by the congruence $\bigcap_n \cG^n$.
A morphism $F\to G\in \S X$ is on $\bigcap_n \cG^n$ if and only if it induces an equivalence of Goodwillie towers.
We do not know whether the inclusion $\cG\mono \subseteq \bigcap_n \cG^n$ is strict.
The logos $\S X\quotient\bigcap_n \cG^n$ classifies \oo connected objects with an unidentified extra-property. 
\end{remark}

Recall that a functor $F:\Fin\to \cS$ is reduced if $F(1)=1$
\begin{corollary}[Hyper-radical]
\label{cor:Goodwillie:tower}
The hyper-radical (the \oo connected maps) of $\fun\Fin\cS\exc n$ is the class of $P_0$\=/equivalences, that is the maps $F\to G$ inducing an isomorphism $F(1)=G(1)$.
In particular, the \oo connected objects of $\fun\Fin\cS\exc n$ are the reduced $n$\=/excisive functors.
Moreover, the hyper-radical is the congruence $\idl {X\exc n}:= \{X\exc n\to 1\}\cong$ generated by the image of $X$, and all \oo connected maps are $n$\=/nilpotent.
\end{corollary}
\begin{proof}
We use the equivalence $\fun\Fin\cS\exc n = \S X\quotient \idl X^{n+1}$ from \cref{thm:Goodwillie:tower}~\eqref{thm:Goodwillie:tower:nilpotent}.
The hyper-reduction of $\S X\quotient \idl X^{n+1}$ is $P_0:\S X\quotient \idl X^{n+1}\to \cS$
by \cref{thm:Goodwillie:tower}~\eqref{thm:Goodwillie:tower:hyper-reduction}.
This show the first statement.
The second one follows since $F$ is reduced if and only if $F\to 1$ is a $P_0$\=/equivalence.
The hyper-radical is $\idl {X\exc n}:= \{X\exc n\to 1\}\cong$ 
by \cref{thm:Goodwillie:tower}~\eqref{thm:Goodwillie:tower:hyper-reduction} also.
By construction, we have $(\{X\exc n\to 1\}\cong)^{n+1}=\Iso$.
Then the coincidence $\Conn\infty = \idl {X\exc n}$ proves the last statement.
\end{proof}

The following corollary is a direct application of \cref{thm:layer}.
A functor $F:\Fin\to \cS$ is $n$\=/homogeneous if it is $n$\=/excisive and $(n-1)$\=/reduced, that is if $F=P_nF$ and $P_{n-1}F=1$.
We denote by $\cS\pointed$ the category of pointed spaces.

\begin{corollary}[Layers]
\label{cor:Goodwillie:layer}
The layer $(\cG^{n-1}\quotient\cG^n)_1$ of the Goodwillie tower is the category of $n$\=/homogeneous functors $\Fin\to \cS$.
This is a pre-stable category.
The pointed layer $(\cG^{n-1}\quotient\cG^n)\pointed_1$ is the category of $n$\=/homogeneous functors $\Fin\to \cS\pointed$.
This is a stable category.
\end{corollary}

\begin{remark}[Pointed case]
\label{rem:pointed}
To adapt \cref{thm:Goodwillie:tower} in the pointed cased, the object $X$ has to be replaced by the universal pointed object $X\pointed\in \S {X\pointed}=\fun\Finp\cS$, that is the forgetful functor $\Finp\to \cS$ (it is also the functor represented by $S^0\in \Fin\pointed$).
In this case, the congruence $\idl{X\pointed}:=\{X\pointed\to 1\}\cong$ is also generated by the canonical map $1\to X\pointed$.
This new generator has some advantages: 
since the $n$\=/th diagonal of the map $1\to X\pointed$ is $1\to \Omega^{n+1}X\pointed$, 
we have $\idl{X\pointed} = \{1\to \Omega^n X\pointed\,|\,n\geq 0\}\ac$ by \cref{ex:lexgen:diag-comp}.
\end{remark}

The following lemma compares the Goodwillie towers of $\fun\Fin\cS$ and $\fun\Finp\cS$.

\begin{lemma}
\label{lem:pushout-tower}
The Goodwillie tower of $\fun\Finp\cS=\S {X\pointed}$ is the pushout of that of $\fun{\Fin}\cS=\S X$ along the logos morphism $\phi:\S X \to \S {X\pointed}$ sending $X$ to $X\pointed$:
\[
\begin{tikzcd}
\fun\Fin\cS \ar[r]\ar[d,"\phi"]& \dots \ar[r]& \fun\Fin\cS\exc 2 \ar[r]\ar[d]& \fun\Fin\cS\exc 1\ar[r]\ar[d]& \fun\Fin\cS\exc 0=\cS\ar[d,equal]\\
\fun\Finp\cS \ar[r]& \dots \ar[r]& \fun\Finp\cS\exc 2 \pomark \ar[r]& \fun\Finp\cS\exc 1 \pomark \ar[r]& \fun\Finp\cS\exc 0=\cS \pomark \,.
\end{tikzcd}
\]
\end{lemma}
\begin{proof}
By \cref{prop:naturality:cong}, $\phi(\idl X ^n)\ac = \big(\phi(\idl X)\ac\big)^n = \idl {X\pointed}^n$.
Then the result about pushouts is an application of \cref{lem:pushout-cong}.
\end{proof}

\smallskip
By a theorem of Goodwillie \cite{G03}, we have an equivalence of categories between $\fun\Finp\cS\exc 1$ and the category $\PSp$ of parametrized spectra.
By the pointed version of \cref{thm:Goodwillie:tower} (or by  \cref{lem:pushout-tower}), 
we get
\[
\PSp
\ =\ 
\fun\Finp\cS\exc 1
\ =\ 
\S {X\pointed}\quotient \idl {X\pointed}^2\,,
\]
where $\idl {X\pointed}=\{X\pointed\to 1\}\cong$.
In the description of parametrized spectra as 1-excisive functors $F:\Finp\to \cS$, the base space is the value $F(1)$, and spectra $\Sp\subseteq \PSp$ correspond to the 1-excisive functors that are also {\it reduced} ($F(1)=1$). 
We then deduce the following classifying property of $\PSp$.

\begin{proposition}
\label{prop:parametrized-spectra}
The logos $\PSp$ of parametrized spectra classifies pointed 1\=/nilpotent objects, 
that is pointed objects $1\to E$ such that the maps $\Delta^m E \pp \Delta^n E$ are invertible for every $m,n\in \NN$, 
or, equivalently, such that the maps
$\Omega^m E\vee \Omega^n E \to \Omega^m E\times \Omega^n E$ are invertible for every $m,n\in \NN$.
\end{proposition}

\begin{proof}
\Cref{lem:pushout-tower}
shows that the localization $\S {X\pointed}\quotient \idl {X\pointed}^2$ is the pushout of logoi
\[
\begin{tikzcd}
\S X \ar[r]\ar[d] & \S {X}\quotient \idl {X}^2 \ar[d]\\
\S {X\pointed} \ar[r] & \S {X\pointed}\quotient \idl {X\pointed}^2 \pomark
\end{tikzcd}
\]
This shows that $\PSp$ classifies objects that are pointed and 1\=/nilpotent.
The descriptions of pointed 1\=/nilpotent objects follow from \cref{cor:computation-prod-cong}.
The first one used the presentation $\idl {X\pointed}= (\{X\to 1\}\diag)\ac$.
The second description uses the alternative presentation $\idl {X\pointed}= \{1\to \Omega^n X\pointed\,|\,n\geq 0\}\ac$ of \cref{rem:pointed}.
\end{proof}

\begin{examples}
\begin{exmpenum}

\item In the logos $\PSp$, spectra $\Sp \subseteq \PSp$ are 1\=/nilpotent objects, in particular, they are \oo connected.

\item More generally, in the logos $\S X\quotient \idl X ^{n+1}=\fun\Fin\cS\exc n$, the nilpotent objects are the reduced $n$\=/excisive functors.
They are all $n$\=/nilpotent.

\item Recall that a functor $F:\Fin\to \cS$ is $n$\=/homogenous functors if $F=P_nF$ and $P_{n-1}F=1$. These functors define 1\=/nilpotent elements in $\S X\quotient \idl X ^{n+1}$.

\item As mentioned in \cref{rem:nilrad-in-hrad}, we suspect that the only nilpotent object in $\S {X\connected\infty}$ is the terminal object.

\end{exmpenum}	
\end{examples}

\subsection{Proofs}
\label{sec:Goodwillie-proof}

This section proves \cref{thm:Goodwillie:tower}.
We will deduce them from \cref{thm:Goodwillie-main}.
A number of results of the section were already shown in \cite[Section 3]{ABFJ:GC}.

\bigskip
In this section we fix a category $\cC$ with finite colimits and we consider the logos $\cS^\cC=\fun \cC \cS$.
This includes categories like $\cC=\Fin$ or $\cC=\Fin\pointed$.
For such a $\cC$, the inclusion of constant functors $\cS\to \fun \cC \cS$ is fully faithful with a left adjoint given by the colimit functor $P_0=\colim_{\cC}:\fun \cC \cS \to \cS$.
The functor $P_0$ is left-exact since $\cC$ is filtered.
If $\cC$ has a terminal object, $P_0$ is the evaluation at this object, but we shall not need to assume this.
The functor $P_0$ is thus a left-exact localization, that is a quotient of the logos $\fun \cC \cS$.
For such a category $\cC$ also, the notions of strongly cocartesian cube and excisive functor make sense.

\begin{definition}
\label{def:n-exc}
\label{def:reduced}
Let $[1]=\{0<1\}$ be the one arrow category.
We shall denote by $\delta_i=(0,\dots,1,\dots,0)$ the object of $[1]^n$ with the only non-zero coordinate in dimension $i$. We denote by $[1]^n_*$ the cube without its initial object.
An {\it $n$\=/cube} in a category $\cC$ is a diagram $\chi:[1]^n\to \cC$.
An $n$\=/cube is {\it cartesian} if the cartesian gap map $\chi(0,\dots,0)\to\lim_{[1]^n_*}\chi(i_1,\hdots,i_n)$  is an isomorphism. An $n$\=/cube is {\it strongly cocartesian} if every face of dimension 2 is a cocartesian square.
If $\cC$ has finite sums $\cC^n\xto + \cC$,
any set of $n$ maps $f_i:[1]\to \cC$ defines a cube 
\begin{align*}
f_1\boxplus \dots \boxplus f_n : [1]^n &\nntto{(f_1\dots,f_n)} \cC^n\ntto + \cC
\,.
\end{align*}
which is always strongly cocartesian.
We shall say that such a cube is {\it free cocartesian}.

If $\cC$ is a category with finite colimits, a functor $F:\cC\to \cS$ is {\it $n$\=/excisive} if, 
for any strongly cocartesian $(n+1)$\=/cube $\chi$, the image $F(\chi)$ is a cartesian cube.
A functor $F:\cC\to \cS$ is {\it reduced} if $\colim_{\cC}F = 1$.
\end{definition}

\begin{lemma}[{\cite[Lemma 3.2.9]{ABFJ:GC}}]
\label{lem:cocart}
If $\cC$ has finite colimits, a cube in $\cC$ is strongly cocartesian cube if and only if it is a cobase change of a free cocartesian cube.
\end{lemma}
\begin{proof}
A cobase change of a strongly cocartesian cube is always strongly cocartesian.
Since the free cocartesian cubes are strongly cocartesian, this shows that the condition is sufficient.
Let $\chi$ be a strongly cocartesian cube.
We put $f_i:\chi(0,\dots,0) \to \chi(\delta_i)$ and consider the cube ${\chi':f_1\boxplus \dots \boxplus f_n}$.
Then $\chi'(0,\dots,0) = \coprod_n\chi(0,\dots,0)$ and we let $\nabla:\chi'(0,\dots,0)\to \chi(0,\dots,0)$ be the codiagonal map.
It is easy to see that $\chi$ is the cobase change of $\chi'$ along $\nabla$.
This shows that the condition is necessary.
\end{proof}

Let $R^{(-)}:\cC\op\to \fun \cC \cS$ be the Yoneda embedding.
For a cube $\chi$ in $\cC$, we denote by 
\[
\Gamma(\chi) := \colim_{U\in ([1]^n_*)\op} R^{\chi(U)}\,\ntto{\gamma(\chi)} R^{\chi(0,\dots,0)}.
\]
the cocartesian gap map of the cube $R^\chi$, the Yoneda image of $\chi$.

\begin{lemma}[{\cite[Theorem 3.3.1]{ABFJ:GC}}]
\label{lem:n-exc}
A functor $F:\cC\to \cS$ is $\gamma(\chi)$\=/local, $\gamma(\chi)\perp F$, if and only if $F(\chi)$ is cartesian.
In particular, $F$ is $(n-1)$\=/excisive if and only if, for any strongly cocartesian $n$\=/cube $\chi$ in $\cC$, we have $\gamma(\chi)\perp F$.
\end{lemma}

\begin{proof}
By definition, the relation $\gamma(\chi)\perp F$ holds if the map $\Map{\gamma(\chi)} F$ is invertible.
This map is 
\begin{align*}
\Map{R^{\chi(0,\dots,0)}} F &\stto \Map{\Gamma(\chi)} F	 \\
=\qquad\qquad F(\chi(0,\dots,0)) &\stto \Map{\colim_{U\in ([1]^n_*)\op} R^{\chi(U)}} F	 \\
=\qquad\qquad F(\chi(0,\dots,0)) &\stto \lim_{U\in [1]^n_*} F(\chi(U))\,,
\end{align*}
which is exactly the cartesian gap map of the cube $F(\chi)$.
This shows the equivalences of the statements.
\end{proof}

Recall from \cref{fperp}, that two maps $u$ and $f$ are fiberwise orthogonal $u\fperp f$ if $u'\perp f$ for every base change $u'$ of $u$.

\begin{lemma}\label{lem:fiberwise-right-orthogonal}
Let $\cC$ be a small category with pushouts,
$F:\cC\to\cS$ be a functor, and $\chi$ be a cube in $\cC$. 
Then $\gamma(\chi)\fperp F$ if and only if $F$ sends any cobase change of $\chi$ to a cartesian cube.
\end{lemma}

\begin{proof}
Let us say that a morphism of cubes $\alpha:\chi\to\chi'$ is cartesian (cocartesian) if for any map $A\to B\in [1]^n$, the square $\chi(A)\to \chi'(A)$
\[
\begin{tikzcd}
\chi(A)\ar[r]\ar[d]&\chi'(A)\ar[d]\\
\chi(B)\ar[r]&\chi'(B)
\end{tikzcd}
\]
is cartesian (cocartesian).
Let $F:\cC\to\cS$ be a functor and let $\chi\to\chi'$ be a cocartesian morphism of cubes. 
Then the induced morphism $R^{\chi'}\to R^{\chi}$ is a cartesian morphism of cubes.
By universality of colimits in logoi the cocartesian gap map $\gamma(\chi')$ is a base change of $\gamma(\chi)$.

If $\gamma(\chi)\fperp F$ then $F$ is local with respect to every base change of $\gamma(\chi)$. 
In particular, it is $\gamma(\chi')$\=/local.
By \cref{lem:n-exc} it follows that $F(\chi')$ is cartesian.
For the reverse implication, we use \cref{lem:generation-acyclic-class} to observe first that it is enough to consider only base changes of $\gamma(\chi)$ over a representable functor $R^A$ for some $A$ in $\cC$ since representable functors form a set of generators for $\fun \cC \cS$.
Because $\cC$ has pushouts, these base changes over $R^A$ become the Yoneda image of cubes that are cobase changes of $\chi$ with initial object $A$, and any such cobase change is of this form. If we collect them for all $A$ in $\cC$ we end up with (the Yoneda image of) all cobase changes $\chi'$ of $\chi$. Since by assumption $F$ sends all $\chi'$ to cartesian cubes, by Lemma~\ref{lem:n-exc} we see that $F$ is $\gamma(\chi')$\=/local for all $\chi'$. This is equivalent to stating $\gamma(\chi)\fperp F$.
\end{proof}

Recall that $P_0=\colim_{\cC}:\fun \cC \cS\to \cS$ is a quotient of logoi when $\cC$ has finite colimits.

\begin{theorem}
\label{thm:Goodwillie-main}
For $\cC$ a small category with finite colimits, let $\cG$ be the congruence of the logos morphism $P_0:\fun \cC\cS\to \cS$.
Then the right adjoint to the quotient $\fun \cC\cS\to \fun \cC\cS\quotient\cG^{n+1}$ identifies $\fun\cC\cS\quotient\cG^{n+1}$ with the subcategory $\fun \cC\cS\exc n$ of $n$\=/excisive functors $\cC \to \cS$.
In other terms, we have equalities of congruences $\cG_n=\cG^{n+1}$ for all $n$.
\end{theorem}

\begin{proof}
In this proof we will shift the statement from $n$ to $n-1$. Then we need to show that the local objects for the $n$\=/fold acyclic power $\cG^n := \cG^{\aprod n}$ of the congruence $\cG=P_0^{-1}(\Iso)$ are exactly the $(n-1)$\=/excisive functors.

Let $\Sigma$ be the set of maps $R^f$ for a map $f$ in $\cC$ (it is a set since $\cC$ is small).
The localization $P_0:\fun \cC\cS\to \cS$ is generated by $\Sigma$, hence $\Sigma\ssat=\cG$.
Since $\cG$ is a congruence, this proves also that $\Sigma\cong=\cG$.

Let $\nabla(f)=s_0\pp f$ denote the codiagonal of $f$ which exists in $\cC$ because $\cC$ is finitely cocomplete (see \cref{sec:codiagonal}).
Then the diagonal of $R^f$ is $\Delta(R^f)=R^{\nabla f}$.
This implies $\Delta(\Sigma) \subseteq \Sigma$ and $\Sigma\diag=\Sigma$.
Using \cref{sigmaac=sigmacong}, we get $\cG= \Sigma\cong = (\Sigma\diag)\ac = \Sigma\ac$.
In other terms, $\Sigma$ is a lex generator for the congruence $\cG$.
By~\cref{boxprodlexgen} this gives us a lex generator for the acyclic powers: $(\Sigma^{\pp n})\ac =(\Sigma^{\pp n})\cong = \cG^n$.
An element of $\Sigma^{\pp n}$ is of the form $R^{f_1}\pp\hdots\pp R^{f_n}$ for some maps $f_i$ in $\cC$.
Recall from~\cref{def:n-exc} that $f_1\boxplus \dots \boxplus f_n$ is a free cocartesian cube.
By definition of the pushout product, the map $R^{f_1}\pp\hdots\pp R^{f_n}$ is the cocartesian gap map of the cube $R^{f_1\boxplus \dots \boxplus f_n}$.
In other words, the set of free cocartesian cubes in $\cC$ provides a lex generator for the congruence $\cG^n$.

Now, let $F:\cC\to\cS$ be a functor which is local for $\cG^n$.
Since $\cG^n = (\Sigma^{\pp n})\ac$, this is equivalent to ${\gamma(f_1\boxplus \dots \boxplus f_n)\fperp F}$ for all free cocartesian $n$\=/cubes in $\cC$.
Using \cref{lem:fiberwise-right-orthogonal,lem:cocart} this is equivalent to $F$ being $(n-1)$\=/excisive.
The result $\cG_{n-1}=\cG^n$ follows.
\end{proof}

\begin{proof}[Proof of \cref{thm:Goodwillie:tower}]

We are going to apply \cref{thm:Goodwillie-main} to $\cC=\Fin$.
Here $P_0=\colim_{\cC}$ becomes the evaluation at the terminal object $1$ of $\Fin$.
This identifies the general case above with the situation in~\cref{thm:Goodwillie:tower}. 

\smallskip
\noindent\eqref{thm:Goodwillie:tower:power}
The proof of $\cG=\idl X$ is \cref{lem:SX}.
The proof that $\cG_n=\cG^{n+1}$ is \cref{thm:Goodwillie-main}.

\smallskip
\noindent\eqref{thm:Goodwillie:tower:mono}
is \cref{exmp:mono-congruence:oo-conn-obj}.

\smallskip
\noindent\eqref{thm:Goodwillie:tower:nilpotent}
An object $E$ in a logos $\cE$ is $n$\=/nilpotent precisely if $(\{E\to 1\}\cong)^{n+1} = \Iso$.
The logos $\S X \quotient \idl X ^{n+1}$ thus classifies $n$\=/nilpotent object by definition.

\smallskip
\noindent\eqref{thm:Goodwillie:tower:hyper-reduction}
We saw in \cref{exmp:rad-congruence2} that the hyper-reduction of $\S {X\connected \infty}$ is $\cS$.
Then the result follows from \cref{thm:tower}~\eqref{thm:tower:4}.
\end{proof}

\begin{remark}
The same arguments with $\cC=\Finp$ provide the pointed version of \cref{thm:Goodwillie:tower}.
\end{remark}


\appendix

\section{Tensor-frames}
\label{sec:tensor-frames}

This appendix collects some results about what we called {\it $\otimes$\=/frames} (tensor-frames) 
which are commutative monoids in suplattices such that the unit is the top element (\cref{def:tensor-frame}).
Our guiding example is  the $\otimes$\=/frame of ideals in a commutative ring $A$ with the product of ideals.
The $\otimes$\=/frame $\Ideal A$ enhances the frame of Zariski open subsets of $\Spec A$ (which consists only of the radical ideals) by keeping the information about {\it powers} of ideals.
These powers are related to the infinitesimal structure of $A$ and $\Spec A$.
In this sense, a $\otimes$\=/frame can be thought of as a frame with an extra infinitesimal structure.

A frame is a $\otimes$\=/frame such that the product is idempotent.
We shall compare the two notions by proving that the inclusion of frames into $\otimes$\=/frames has a left adjoint, sending a $\otimes$\=/frame to its frame of radical elements (recovering the projection $\Ideal A \to \Op{\Spec A}$), as well as a right adjoint extracting the subframe of idempotent elements (see the diagram of adjunctions below).
These two frames are not in general isomorphic.

The geometric dual of a $\otimes$\=/frame can be called a $\otimes$\=/locale.
It is possible to define notions of points, of open and closed inclusions, of surjections and inclusions for $\otimes$\=/locales that generalize that of locales.
This makes $\otimes$\=/locales worthy of being geometric objects.
Such developments will not be presented here though, where we shall only focus on the algebraic notion of $\otimes$\=/frame.

\medskip
The first two sections of this appendix focus on the more general notion of a $\otimes$\=/lattice (monoid in suplattices without the condition that the unit is the maximal element).
The main result is \cref{localizationquantale} which will be one of our main tools to construct monoidal structures.
The comparisons between the categories of $\otimes$\=/lattices, $\otimes$\=/frames, and frames are summarized in the following diagram (where left adjoints are on top of their right adjoint).

\[
\begin{tikzcd}
\TLattice \ar[from=rr, hook']\ar[rr, shift left=3,"N"]\ar[rr, shift right=3,"(-)\downslice 1"']
&& \TFrame \ar[from=rr, hook']\ar[rr, shift left=3,"\sqrt-"]\ar[rr, shift right=3,"(-)\idem"']
&& \Frame
\end{tikzcd}
\]

\begin{table}[h!]
\caption{Some Algebra--Geometry dualities}
\label[table]{table:tensor-locales}
\begin{center}	
\renewcommand{\arraystretch}{1.6}
\begin{tabularx}{.6\textwidth}{
|>{\centering\arraybackslash}X
|>{\centering\arraybackslash}X|
}
\hline
{\it Algebra} & {\it Geometry}\\
\hline
commutative ring & affine scheme\\
\hline
Boolean algebras & Stone spaces\\
\hline
logos & topos\\
\hline
frame & locale\\
\hline
{\it $\otimes$\=/frame} & {\it $\otimes$\=/locale} \\
\hline
(commutative) $\otimes$\=/lattice & (commutative) quantale\\
\hline
\end{tabularx}
\end{center}	
\end{table}

\subsection{Tensor-lattice and quantales}
\label{sec:mframe}

Recall that a poset $P$ is said to be a {\it suplattice} if every subset $S\subseteq P$ has a supremum $\mathsf{sup}(S)=\bigvee_{x\in S} x$.
Every suplattice has a minimum element $\bot=\mathsf{sup}(\emptyset)$.
Every suplattice $P$ is also an {\it inf-lattice}, since 
\[
\mathsf{inf}(S)=\mathsf{sup}\{x\in P \ | \ \forall s\in S, \ x\leq s \}
\]
for every subset $S\subseteq P$. 
In particular, $P$ has a maximum element $\top=\mathsf{inf}(\emptyset)=\mathsf{sup}(P)$.
A {\it morphism} of suplattices is a map preserving suprema $\phi:P\to Q$. In particular, $\phi(\bot)=\bot$.
Every morphism of suplattices $\phi:P\to Q$ has a right adjoint $\phi_*:Q\to P$ given by
\[
\phi_*(y)=\mathsf{sup}\{ x\in P \ | \ \phi(x)\leq y \}
\]
for every $y\in Q$.
The map $\phi_*$ preserves infima.

\begin{definition}
A (commutative) {\it tensor-lattice}, or {\it $\otimes$\=/lattice}, is a suplattice $Q$ equipped with a commutative monoid structure $(Q,\star,1)$ for which the product $\star: Q\times Q\to Q$ preserves suprema in each variable.
If $P$ and $Q$ are $\otimes$\=/lattices, we shall say that a map  $\phi:P\to Q$ is a {\it morphism} if it is both a morphism of suplattices and a morphism of monoids.
We denote by $\TLattice$ the category of $\otimes$\=/lattices.
\end{definition}

Let us spell out the notion of $\otimes$\=/lattice.
By definition, the map $a\star(-): Q\to Q$ preserves suprema for every $a\in Q$.
Thus, we have
\[
a\star \bigl(\bigvee_{x\in S} x \bigr)=\bigvee_{x\in S} a\star x
\]
for every subset $S\subseteq Q$.
It follows that the map $a\star(-): Q\to Q$ has a right adjoint, 
denoted $a\div (-):Q\to Q$, and called the {\it division} by $a\in Q$. 
For every $a,b,c\in Q$ we have
\[
a\star b\leq c \ \Leftrightarrow \  b\leq a\div c\,. 
\]

\begin{remark}
A poset is a special kind of category, and a $\otimes$\=/lattice $(Q,\star,1)$ is a special kind of  a symmetric monoidal closed category.
By definition, the tensor product of two elements $x$ and $y$ of $Q$ is $x\star y$, and the internal hom $Hom(x,y)$ is the element $x\div y$ for every $x,y\in Q$.
\end{remark}

\begin{examples}
\begin{exmpenum}
\item\label{quantfrommonoid}
The power set $\powerset M$ of a commutative monoid $M$ has the structure of a $\otimes$\=/lattice where $A\star B=\{xy \ | \ x\in A, y\in B\}$, and where  $A\div B=\{x\in M \ | \ Ax\subseteq B\}$.
The unit of $\powerset M$ is the singleton $\{1\}$, where $1\in M$ is the unit of the monoid $M$.

\item\label{ex:ring-ideal}
The poset of ideals $\Ideal A$ of a commutative ring $A$ has the structure of a $\otimes$\=/lattice, where $I\star J$ is the ideal generated by the set $\{xy \ | \ x\in I, y\in J\}$, and where $I\div J=\{x\in A \ | \ Ix\subseteq J\}$.
The unit of the $\otimes$\=/lattice $\Ideal A$ is the largest ideal $A$.

\item\label{quantinterval01}
The real interval $[0,1]$ has the structure of a commutative $\otimes$\=/lattice, where for numbers $x,y\in [0,1]$, $x\star y=xy$ is the product and $x\div y= x^{-1}y$ is the division.

\item\label{framearequantal}
Recall that suplattice $F$ is said to be a {\it frame} if the map $x\wedge (-):F\to F$ preserves suprema for every element $x\in F$.
A frame $F$ is a $\otimes$\=/lattice with $x\star y:=x\wedge y$.
The unit element is the largest element of $F$.
We shall see in \cref{cor:charac-frames} that a $\otimes$\=/lattice is a frame iff $x\star x = x$ for all $x$.

\item\label{opensubset}
The lattice of open subsets of a topological space $X$ is a frame $\Op X$, thus a $\otimes$\=/lattice.

\item\label{Sierpinski} In particular, the Sierpi\'nski frame $S:=\{\bot<x<\top\}$ (with product defined by $x\star x = x$) is a $\otimes$\=/lattice.
A $\otimes$\=/lattice morphism $S\to Q$ is the same thing as an idempotent element in $Q$.
The $\otimes$\=/lattice $S$ is a frame, and if $Q$ is another frame, a $\otimes$\=/lattice morphism $S\to Q$ is the same thing as an element in $Q$.

\item\label{ex:free-mframe} The free $\otimes$\=/lattice on one generator is the powerset $\powerset{\mathbb N}$ with the product extending the sum of integers by commutation to suprema. The universal element is $1\in \mathbb N$ and the multiplication unit is $0\in \mathbb N$.

\end{exmpenum}
\end{examples}

Generalizing \cref{ex:free-mframe}, we have the following result whose proof we leave to the reader.
\begin{proposition}
\label{prop:free-quantale}
The free $\otimes$\=/lattice on a set $X$ is the powerset $\powerset{M(X)}$, where $M(X)$ is the free commutative monoid on the set $X$.
\end{proposition}

\subsection{Closure operators and quotients of tensor-lattices}

The main result of this section is \cref{localizationquantale} which is a criteria for the existence of monoidal structures on reflective subposets of a $\otimes$\=/lattice.

Recall that a subset $P'$ of a poset $P$ is {\it reflective} if 
the inclusion $i:P'\subseteq P$ has a left adjoint $\rho:P\to P'$ called the {\it reflector}. 
The map $R:=i\rho:P\to P$ is then a closure operator:
\begin{enumerate}[label=(\roman*)]
\item $x\leq y\Rightarrow R(x)\leq R(y)$ for every $x,y\in P$;
\item $x\leq R(x)$ for every $x\in P$ 
\item $RR(x)=R(x)$ for every $x\in P$.
\end{enumerate}
Conversely, if $R:P\to P$ is a closure operator, then the subset of $R$\=/closed elements $P'=\{x\in P \ | \ R(x)=x\}$ is reflective with reflector $\rho:P\to P'$ given by $\rho(x)=R(x)$.
Then we get back $R=i\rho:P\to P$, where $i:P'\subseteq P$ is the inclusion.
The subset $P'$ is a suplattice if $P$ is a suplattice: the supremum in $P'$ of a subset $S\subseteq P'$ is the image by $\rho$ of supremum of $S$ in $P$.

\begin{proposition}
\label{localizationquantale}
Let $P=(P,\star, 1)$ be a commutative $\otimes$\=/lattice, let $R:P\to P$ be a closure operator and let $P'=\{x\in P \ | \ R(x)=x\}$.
Then the following two conditions are equivalent:
\begin{propenum}
\item\label{localizationquantale:1}
    for every $x\in P$, $x\div P'\subseteq P'$;
\item\label{localizationquantale:2}
    for every $x,y\in P$, $R(x\star y)=R(x\star R(y))$.
\end{propenum}
If these conditions hold, let us put $x\cdot y :=R(x\star y)$ for every $x,y\in P$.
Then
\[
x\cdot y
\ =\ 
R(x)\cdot y
\ =\ 
x\cdot R(y)
\ =\ 
R(x)\cdot R(y)
\]
for every $x,y\in P$.
If $y\in P'$, then $R(x)\div y=x\div y$ for every $x\in P$.
Moreover, the poset $P'$ has the structure of a $\otimes$\=/lattice with the operation $(x,y)\mapsto x\cdot y$, and the reflector $\rho:P\to P'$ is a morphism of $\otimes$\=/lattices.
\end{proposition}

\begin{proof} 
\noindent (1)$\Ra$(2) 
For every $x,y\in P$, we have $x\star y\leq R(x\star y)$
and hence $y\leq x\div R(x\star y)$.
Thus, $R(y)\leq x\div R(x\star y)$ by (1), since $R(x\star y)\in P'$. 
It follows that $x\star R(y)\leq R(x\star y)$ and hence that $R(x\star R(y))\leq R(x\star y)$.
But we have $R(x\star y)\leq R(x\star R(y))$, since $y\leq R(y)$.
The equality $R(x\star y)=R(x\star R(y))$ is proved. 

\smallskip
\noindent (2)$\Ra$(1) 
If $y\in P'$, let us show that $x\div y\in P'$ for every $x\in P$.
For every $z\in P$, we have
\[
z\ \leq x\div y
\ \Rightarrow \ 
x\star z\ \leq\ y 
\ \Rightarrow \ 
R(x\star z)\ \leq\ y
 \ \Rightarrow \
R(x\star R(z))\ \leq\ y
 \ \Rightarrow \
x\star R(z)\ \leq\ y
 \ \Rightarrow \
R(z)\ \leq\ x\div y
\]
since $R(x\star z)=R(x\star R(z))$ by (2).
In particular, if $z:=x\div y$ then $R(x\div y)\leq x\div y$. 
This proves that $x\div y\in P'$. 
The equivalence (1)$\iff$(2) 
is proved.

\smallskip
If the conditions 
(1) 
and 
(2) 
hold,
let us show that $P'$ has the structure of a $\otimes$\=/lattice  with the operation defined by putting $x\cdot y :=R(x\star y)$ for $x,y\in P'$. 
First, $P'$ is a suplattice and the reflector $\rho:P\to P'$ preserves suprema, since $P$ is a suplattice.
The product $\cdot:P'\times P'\to P'$ is commutative since $x\cdot y =R(x\star y)=R(y\star x)=y\cdot x$.
Observe that for every $x,y\in P$, we have $R(x\star y)=R(y \star x)=R(y\star R(x))=R(R(x)\star y)$ by 
condition (2)%
, since the operation $\star$ is commutative.
It follows that for every $x,y,z\in P'$ we have
\[
x\cdot (y\cdot z)\ :=\ R(x\star (y\cdot z))\ =
R(x\star R(y\star z))\ =\ R(x\star (y\star z))
\]
\[
= R((x\star y)\star z)\ =\ R(R(x\star y)\star z)\ =
R((x\cdot y)\star z)\ =\ (x\cdot y)\cdot z\,.
\]
Hence the product $\cdot: P'\times P'\to P'$ is associative.
The element $R(1)\in P'$ is a unit, since
\[
x\cdot R(1)
\ =\ 
R(x\star R(1))
\ =\ 
R(x\star 1)
\ =\ 
R(x)
\ =\ 
x
\]
for every $x\in P'$ by condition (2). 
The map $x\div (-):P'\to P'$ is right adjoint to the map $x\cdot (-):P'\to P'$  for every $x\in P'$, since we have 
\[
x\cdot z \leq y \ \Leftrightarrow  \ R(x\star z)\leq y
\  \Leftrightarrow \ x\star z\leq y \ \Leftrightarrow \
  z\leq x\div y
\]
for every $y,z\in P'$.
This completes the proof that $P'$ is a $\otimes$\=/lattice.
The map  $\rho:P\to P'$ is a morphism of monoids, since we have
\[
\rho(x\star y)
\ =\ 
R(x\star y)
\ =\ 
R(x\star R(y))
\ =\ 
R(R(x)\star R(y))
\ =\ 
R(x)\cdot R(y)
\ =\ 
\rho(x)\cdot \rho(y)
\]
for every $x,y\in P$. 
Hence the map $\rho:P\to P'$ is a morphism of $\otimes$\=/lattices, since it also preserves suprema.
\end{proof}

\begin{remark}
Quotients of $\otimes$\=/lattices are localizalitions of the underlying suplattice.
As such they are controlled by the class $W$ or relations $x\leq x'$ that become equalities in the quotient.
For the quotient to be compatible with the monoidal structure, $W$ need to satisfy the following compatibility with the product: if $x\leq x'$ is in $W$, then, for every $y$, $x\star y\leq x'\star y$ must be in $W$.
When translated as a condition on the $W$\=/local objects, this is the condition \cref{localizationquantale:1}.
\end{remark}

\begin{lemma}
\label{divisionmorphism}
Let $\phi:P\to Q$ be a morphism of $\otimes$\=/lattices with right adjoint $\phi_*:Q\to P$.
Then the following identity holds
\[
x\div \phi_*(y)
\ =\ 
\phi_*\big(\phi(x)\div y\big)
\]
for every $x\in P$ and $y\in Q$.
\end{lemma}

\begin{proof}
For any $z\in P$, we have equivalences 
\begin{align*}
z\ \leq\ x\div \phi_*(y)
    &\ \Leftrightarrow\  x\star z\ \leq\ \phi_*(y)\\
    &\ \Leftrightarrow\  \phi(x\star z)\ \leq\ y\\
    &\ \Leftrightarrow\  \phi(x)\star \phi(z)\ \leq\ y\\
    &\ \Leftrightarrow\  \phi(z)\ \leq\ \phi(x)\div y\\
    &\ \Leftrightarrow\  z\ \leq\ \phi_*(\phi(x)\div y)\,.\qedhere
\end{align*}
\end{proof}

\begin{proposition}[Image factorization]
\label{prop:tensor-lattice:image-facto}
Let $\phi:P\to Q$ be a morphism of $\otimes$\=/lattices, with right adjoint ${\phi_*: Q\to P}$.
Then the map $R:=\phi_* \phi:P\to P$ is a closure operator satisfying the conditions of \cref{localizationquantale}.
If $P'=\{x\in P\ | \ R(x)=x\}$ then the morphism $\phi:P\to Q$ admits a factorisation in the
category of $\otimes$\=/lattices
\[ \begin{tikzcd}
P\ar[rr,"\phi"] \ar[rd,"{\rho}"']&& Q \\
& P' \ar[ru,"{\phi'}"']
\end{tikzcd}
\] 
where $\rho:P\to P'$ is the reflector. 
Moreover, both the right adjoint to $\rho$ and $\phi'$ are injective morphisms of posets.
\end{proposition}
\begin{proof}
To see that $R=\phi_*\phi$ is a reflector, we need only to check that it is idempotent (the other conditions are clear).
The triangular relations for the adjunction $\phi\dashv\phi_*$ give
$\phi_*\leq \phi_*\phi\phi_*\leq \phi_*$, that is $\phi_*\phi\phi_*=\phi_*$.
From there, we get $\phi_*\phi\phi_*\phi = \phi_*\phi$, that is $R^2=R$.
Dually, putting $S=\phi\phi_*$, we find that $S^2=S$ as an endomorphism of $Q$.
In other words, every adjunction of posets is idempotent.
This implies that $\phi\dashv\phi_*$ restricts to an isomorphism between the posets of fixed points of $R$ and $S$.
We get a factorization of $\phi$:
\[ \begin{tikzcd}
P\ar[rr,"\phi"] \ar[d,"{\rho}"']&& Q \\
P'={\sf Fix}(R) \ar[rru,"{\phi'}"']\ar[rr,equal] && Q'={\sf Fix}(S)\ar[u, "\sf can"', hook']\,,
\end{tikzcd}
\] 
and we define $\phi'$ as the composition of the isomorphism $P'=Q'$ and the inclusion $Q'\to Q$.

Using the relation $\phi_*\phi\phi_*=\phi_*$, the set $P'$ is equivalently described as  $P'=\{x\in P\,|\,\exists y\in Q, x=\phi_*(y)\}$ (\ie the image of $\phi_*$).
Then the relation $x\div P'\subseteq P'$ of \cref{localizationquantale:1} is satisfied by \cref{divisionmorphism}, and $\rho:P\to P'$ is a morphism of $\otimes$\=/lattices.

Let us see that $\phi'$ is also such a morphism.
By definition of $\phi'$, it is enough to check that the subposet $Q'\subseteq Q$ is closed under suprema and products.
We use the description of $Q'$ as the image of $\phi$, $Q'={\{y\in Q\,|\,\exists x\in P, y=\phi(x)\}}$.
Let $x'_i$ be a family of elements in $Q'$. 
By characterization of $Q'$, it is the image of a family $x_i$ of elements in $P$.
Then $\bigvee x'_i = \bigvee\phi(x_i) = \phi(\bigvee x_i)$ shows that the suprema is in $Q'$.
Similarly, given two elements $x'$ and $y'$ in $Q'$, with lifts $x$ and $y$ in $P$, the equality
$x'y' =\phi(x)\phi(y) = \phi(xy)$ show that their product is in $Q'$.

For the last statement:
the morphism $\phi'$ is injective by construction, and $\rho_*$ is injective since $\rho$ is a localization of suplattices.
\end{proof}

\begin{definition}
\label{def:quotient}
We shall say that a morphism of $\otimes$\=/lattices $\phi:P\to Q$ is a {\it quotient} if the map $\phi_*:Q\to P$ is injective.
The corresponding morphism of commutative quantales is called an {\it inclusion}.
We shall say that a morphism of $\otimes$\=/lattices $\phi:P\to Q$ is an {\it injection} if $\phi$ is injective on objects.
The corresponding morphism of commutative quantales is called a {\it surjection}.
Geometrically, \cref{prop:tensor-lattice:image-facto} provides a factorization of a morphism of commutative quantales into a surjection followed by an inclusion.
\end{definition}

\subsection{Tensor-frames}
\label{sec:normal-quantale}

\begin{definition}
\label{def:tensor-frame}
Let $Q=(Q,\star,1)$ be a $\otimes$\=/lattice. 
We shall say that an element $x\in Q$ is {\it normal} if $\top \star x= x$, where $\top $ is the top element of $Q$.
We shall say that the $\otimes$\=/lattice $Q=(Q,\star,1)$ is a {\it $\otimes$\=/frame} if $1=\top$ (if every element $x\in Q$ is normal).
The category $\TFrame$ of $\otimes$\=/frames is defined as a full subcategory of that of $\otimes$\=/lattices.
\end{definition}

\begin{examples}
\begin{exmpenum} 

\item \label{quantfrommonoid2}
The power set $\powerset M$ of a commutative monoid $M$ has the structure of a $\otimes$\=/frame by  \Cref{quantfrommonoid}.
A subset $A\subseteq M$ is normal in $\powerset M$ if and only if it is an {\it ideal} (\ie if $M\star A\subseteq A$).
The subset $Mx :=\{y\star x\,|\,y\in M\}$ is the smallest ideal containing the element $x$.
The relation $y\in Mx$ is a preorder $x\leq y$ on $M$ such that all invertible elements (in particular the unit) are maximal.
A subset $A\subseteq M$ is normal if and only it is upward closed with respect to this preorder.
The set of ideals of $M$ is a $\otimes$\=/frame for the product of ideals.

\item \label{ex:ring-ideal2} 
The $\otimes$\=/lattice $\Ideal A$ of \cref{ex:ring-ideal} is a $\otimes$\=/frame.

\item 
The $\otimes$\=/lattice $[0,1]$ of \cref{quantinterval01} is a $\otimes$\=/frame. 

\item \label{framenormalquantale}
Any frame is a $\otimes$\=/frame.

\item \label{subnormal}
If $(Q,\star, 1)$ is a $\otimes$\=/lattice, then the poset $Q\downslice 1=\{x\in Q \ | \ x\leq 1\}$ is a sub-$\otimes$\=/lattice which is an $\otimes$\=/frame. 
The $\otimes$\=/lattice $Q\downslice 1$ is the coreflection of $Q$ in $\otimes$\=/frames: for any $\otimes$\=/frame $N$, any morphism $N\to Q$ sends $\top$ to 1 and therefore factors through $Q\downslice 1$.

\item \label{normal-quotient}
Generalizing \cref{quantfrommonoid2}, if $(Q,\star, 1)$ is a $\otimes$\=/lattice, the the poset ${\sf N}(Q) = \{x\,|\, \top\star x = x\}$ of normal elements is a $\otimes$\=/frame with unit $\top$.
The inclusion ${\sf N}(Q)\subseteq Q$ has a right adjoint (sending $x$ to $\top\star x$) which is a morphism of $\otimes$\=/lattices.
This morphism $Q\to {\sf N}(Q)$ is the reflection of $Q$ in $\otimes$\=/frames: 
if $P$ is a $\otimes$\=/frame any morphism $Q\to P$ factors uniquely through ${\sf N}(Q)$.

\item The free $\otimes$\=/frame on one generator $x$ is $F(x):=\{\bot<\dots<x^2<x<\top\}$.
If $Q$ is a $\otimes$\=/frame, a $\otimes$\=/frame morphism $F(x)\to Q$ is the same thing as an element in $Q$.
A $\otimes$\=/frame morphism $Q\to F(x)$ can be called a $\infty$\=/jet of arc in $Q$.

\item The $n$\=/nilpotent $\otimes$\=/frame on one generator $x$ is $F_n(x):=\{\bot<\dots<x^2<x<\top\}$ with the relation $x^{n+1}=\bot$.
If $Q$ is a normal frame, a $\otimes$\=/frame morphism $F_n(x)\to Q$ is the same thing as a nilpotent element of order $n$ in $Q$.
A $\otimes$\=/frame morphism $Q\to F_n(x)$ can be called a $n$\=/jet of arc in $Q$.

\end{exmpenum}
\end{examples}

We expand a bit on \cref{quantfrommonoid2}.
Let $[1]$ be the monoid $(\{0<1\},{\sf inf})$.
The characteristic functions of ideals define a bijection between ideals of $M$ and morphisms of preordered monoids $\phi:M\op\to [1]$.
The Yoneda map $M\to {\sf Ideal}(M) = \fun {M\op}{[1]}$ sends an element to its principal ideal $Mx$.
The product of $M$ extends in an $\otimes$\=/frame structure on ${\sf Ideal}(M)$ by commutation with suprema.
The proof of the following result is left to the reader.
\begin{proposition}
\label{prop:free-normalquantale}
The free $\otimes$\=/frame on a commutative monoid is ${\sf Ideal}(M)$.
The free $\otimes$\=/frame on a set $X$ is ${\sf Ideal}(M(X))$, where $M(X)$ is the free commutative monoid on $X$.
\end{proposition}

\begin{lemma}
\label{normquantalelemma}
In a $\otimes$\=/frame $Q=(Q,\star,1)$, we always have $x\star y\leq x\wedge y$ for every $x,y\in Q$.
Moreover, the infinite sequence $\{x^{\star n}\ |\ n\in \mathbb{N}\}$ is decreasing for every element $x\in Q$.
\end{lemma}

\begin{proof}
We have $x\star y\leq 1\star y$, since $x\leq 1$.
Thus, $x\star y\leq y$.
Similarly, we have $x\star y\leq x$.
Thus, $x\star y\leq x\wedge y$.
Moreover, $x^{\star (n+1)} =x\star x^{\star n}\leq x^{\star n}$.
\end{proof}

The following result is a reformulation of \cref{subnormal,normal-quotient}.

\begin{proposition}
There exists a triple adjunction 
\[
\begin{tikzcd}
\TLattice \ar[from=rr, hook',"{\sf can.}" description]\ar[rr, shift left=3,"{\sf N}"]\ar[rr, shift right=3,"(-)\downslice 1"']
&& \TFrame
\end{tikzcd}
\]
\end{proposition}

\begin{lemma}
\label{quotient-tensor-frame}
If $P\to Q$ is a quotient of $\otimes$\=/lattices and $P$ is a $\otimes$\=/frame, then so is $Q$.	
\end{lemma}

\begin{remark}
\label{rem:tensor-frame:quotient}
If $P\to Q$ is a quotient of $\otimes$\=/lattices and $P$ is a $\otimes$\=/frame, then so is $Q$.	
The category of $\otimes$\=/frames admits a factorization system in which the left class is the class of quotients and the right class that of injections.
Geometrically, this provides a factorization of a morphism of $\otimes$\=/locales into a surjection followed by an inclusion.
Restricted to frames, this gives back the classical image factorization.
\end{remark}

\subsection{Radical and idempotent elements, frame (co)reflections}

From now on, we specialize our study to $\otimes$\=/frames.

\begin{definition}
\label{def:radical-elt}
An element $x$ in a $\otimes$\=/frame $Q$ is called {\it radical} if, for every $y$, the implication $y^2\leq x \Rightarrow y\leq x$ holds.
\end{definition}

It is easy to see that $x$ is radical if and only if, for every $y$ and every $n$, the implication $y^n\leq x \Rightarrow y\leq x$ holds.

\begin{examples}
\begin{exmpenum} 

\item The elements $\top$ and $\bot$ are always radical.

\item Every element is radical in a frame.
We shall see in \cref{cor:charac-frames} that is characterization of frames.

\item In the free $\otimes$\=/frame $\{\bot<\dots<x^2<x<\top\}$, the radical elements are $\bot$, $\top$ and $x$.

\end{exmpenum}
\end{examples}

\begin{proposition}
\label{prop:frame-refl}
The poset $\sqrt Q$ of radical elements of $Q$ is reflective and the reflection $Q\to \sqrt Q$ is a morphism of $\otimes$\=/frames.
Moreover, this localization is generated by the forcing the relations $x^2\leq x$ to become equalities for every $x$.
\end{proposition}
\begin{proof}
Since suplattices are also complete, the radical reflection $\sqrt x$ of an element $x$ can always be computed as
\[
\sqrt x = \mathsf{inf}\{\, y \,|\, \text{$y$ is radical and $x\leq y$}\,\}\,.
\]
(See also \cref{rem:radical-soa}.)
We can then apply \cref{localizationquantale}.
We verify condition \eqref{localizationquantale:1}:
for $y$ be a radical element and $x$ an arbitrary element, we need to show that $x\div y$ is radical.
We need to show that the condition $z^n \leq x\div y$ implies $z \leq x\div y$:
\begin{align*}
z^n \leq x\div y 
&\ \Leftrightarrow\ z^n\star x \leq y	\\
&\ \Rightarrow\ (z\star x)^n \leq y	&& \text{because $z^n\star x^n\leq  z^n\star x$}\\
&\ \Leftrightarrow\ z\star x \leq y	&& \text{because $y$ is radical}\\
&\ \Leftrightarrow\ z\leq  x\div y	\,. 
\end{align*}
This shows the first statement. 
For the second one, we use that, by definition, radical objects are the local objects for the relations $x^2\leq x$.
\end{proof}

\begin{remark}
\label{rem:radical-soa}	
Notice that radical objects are the local objects for the relations $y^2\leq y$, so we can also construct $\sqrt x$ by a small object argument.
Let us put
\[
\sqrt[1] x = \mathsf{sup}\{\, y \,|\, \exists k,\, y^k\leq x\,\}
\]
and $\sqrt[n+1] x = \sqrt[1] {\sqrt[n] x}$.
Then the transfinite colimit of $x\leq \sqrt[1] x \leq \sqrt[2] x \leq \dots$ is $\sqrt x$.
In the example of ring ideals, we have in fact $\sqrt[1] x = \sqrt x$, 
but this might not happen in general.
\end{remark}

The following result says that the subcategory of frames is reflective in that of $\otimes$\=/frames.
\begin{proposition}
\label{prop:refl-frame}
The $\otimes$\=/frame $\sqrt Q$ is a frame and the morphism $Q\to \sqrt Q$ is the universal reflection of $Q$ in frames.
\end{proposition}
\begin{proof}
Let us see that $\sqrt Q$ is a frame.
By \cref{localizationquantale}, the product $\cdot$ in $\sqrt Q$ is defined by $x\cdot y = \sqrt{x\star y}$.
Then we see that $x\cdot x = \sqrt{x\star x} = \sqrt x$ by \cref{prop:frame-refl}
This proves that $\sqrt Q$ is a frame by \cref{cor:charac-frames}.

Let us show now that the reflection $\rho$ is universal.
Let $\phi:Q\to F$ be a $\otimes$\=/frame morphism where $F$ is a frame.
Since any element is idempotent in a frame, we must have $\phi(x^2)=\phi(x)$.
Using \cref{prop:frame-refl}, we get a factorization $\phi=\phi'\rho:Q\to \sqrt Q\to P$ of $\phi$ in the category of suplattices (where $\phi'(\sqrt x) =\phi(x)$).
Let us verify that $\phi':\sqrt Q\to F$ is a frame morphism:
\begin{align*}
\phi'(\sqrt x\cdot \sqrt y)
	&= \phi'(\sqrt {x\star y})\\
	&= \phi( x\star y)\\
	&= \phi(x)\wedge \phi(y)\\
	&= \phi'(\sqrt x)\wedge \phi'(\sqrt y)\,.
\end{align*}
This proves that $\rho:Q\to \sqrt Q$ is the frame reflection of $Q$.
\end{proof}

\begin{examples}
\begin{exmpenum}

\item The frame reflection of the free $\otimes$\=/frame $\{\bot<\dots<x^2<x<\top\}$ is the Sierpi\'nski frame $\{\bot<x<\top\}$.

\item \label{ex:ring-ideal3} 
In the $\otimes$\=/frame of ideals $\Ideal A$ the radical elements are exactly the radical ideals.
Let $\RIdeal A\subseteq \Ideal A$ be the subposet of radical ideals.
The frame reflection of $\Ideal A\to \RIdeal A$ is sending an ideal $I$ to its radical $\sqrt I$.
In this example, $\sqrt I$ is also the infima of prime ideals containing $I$, 
but this result does not generalize to arbitrary $\otimes$\=/frames.

\end{exmpenum}
\end{examples}

Let $Q\idem\subseteq Q$ be the subposet spanned by idempotent elements of a $\otimes$\=/frame $Q$.

\begin{lemma}
\label{lem:corefl-frame}
$Q\idem$ is a sub-$\otimes$\=/frame of $Q$ which is a frame.
\end{lemma}
\begin{proof}
The poset $Q\idem$ is closed under product since, for every $a,b\in Q$, we have $(a\star b)^2 = a^2\star b^2 = a\star b$.
Let us see that it is closed under suprema:
for any family of elements $a_i\in Q$, we have 
\[
\bigvee a_i \geq  (\bigvee a_i)^2 = \bigvee a_i\star a_j \geq \bigvee a_i^2 = \bigvee a_i\,.
\]
This shows that $\bigvee a_i = (\bigvee a_i)^2$ and that $Q\idem$ is a sub-$\otimes$\=/frame of $Q$.
Then it is a frame since every element is idempotent.
\end{proof}

\begin{proposition}
\label{prop:corefl-frame}
The inclusion $Q\idem\to Q$ is the universal coreflection of $Q$ into frames.
\end{proposition}
\begin{proof}
Let $F$ be a frame and $F\to Q$ a morphism of $\otimes$\=/frames.
Since every element of $F$ is idempotent the morphism has values in $Q\idem$.
The morphism $F\to Q\idem$ is a frame morphism since $Q\idem\to Q$ is a sub-$\otimes$\=/frame which is a frame by \cref{lem:corefl-frame}.
\end{proof}

Let $\Frame$ be the category of frames.
The following result is a consequence 
 of \cref{prop:refl-frame,prop:corefl-frame}.

\begin{corollary}
There exists adjoint functors
\[
\begin{tikzcd}
\TLattice \ar[from=rr, hook',"{\sf can.}" description]\ar[rr, shift left=3,"{\sf N}"]\ar[rr, shift right=3,"(-)\downslice 1"']
&&
\TFrame \ar[from=rr, hook',"{\sf can.}" description]\ar[rr, shift left=3,"\sqrt-"]\ar[rr, shift right=3,"(-)\idem"']
&& \Frame\,.
\end{tikzcd}
\]
\end{corollary}

\begin{corollary}
\label{cor:charac-frames}
The following conditions on a $\otimes$\=/frame $Q$ are equivalent:
\begin{corenum}
\item $Q$ is a frame;
\item every element is radical;
\item every element is idempotent.
\end{corenum}
\end{corollary}

Notice that the radical elements are the elements that are local for the the relation $x^n\leq n$ and that the idempotent elements are the colocal elements for the same relations.
This implies the following result.

\begin{proposition} 
\label{prop:idem-coreflection}
If the frame reflection $Q\to \sqrt Q$ admits a left adjoint, then its image is the poset $Q\idem$, and the morphism $Q\idem \to \sqrt Q$ is an isomorphism of frames.
\end{proposition}

\begin{proposition}
\label{overquantale}
Let $Q=(Q,\star,1)$ be a $\otimes$\=/frame and $w\in Q$.
The poset 
$Q\upslice w =\{x\in Q \ | \ x\geq w\}$ has the structure of a $\otimes$\=/frame with the 
multiplication $\star_w: Q\upslice w \times Q\upslice w \to Q\upslice w $ defined by putting $x\star_w y=(x\star y)\vee w$ for every $x,y\in Q\upslice w$.
Moreover, the map
$w\vee(-):Q\to  Q\upslice w$
is a morphism of $\otimes$\=/frames.
\end{proposition}

\begin{proof}
We shall use \Cref{localizationquantale} with $Q':=Q\upslice w$.
The map $R=w\vee(-):Q\to  Q\upslice w$ is a quotient.
If $y\in  Q\upslice w$ let us show that $x\div y\in Q\upslice w$ for every $x\in Q$.
We have $x\star w \leq 1\star y=y$, since $x\leq 1$ and $w\leq y$.
Thus, $w\leq x\div y$ and hence $x\div y\in Q\upslice w$.
The result then follows from \Cref{localizationquantale}.
\end{proof}

\begin{lemma}
\label{lem:image-radical}
If $\phi:Q\to Q\upslice w$ is the $\otimes$\=/frame morphism of \cref{overquantale}, 
then the radical elements of $Q\upslice w$ are in bijection with the radical elements of $Q$ which are above $w$.
In particular, $\phi(\sqrt w) = \sqrt{\phi(w)}$.
\end{lemma}
\begin{proof}
Let us denote by $(\sqrt Q)\upslice w$, the poset of radical elements of $Q$ which are above $w$.
We want to show that $\sqrt{Q\upslice w} = (\sqrt Q)\upslice w$.
The morphism $\phi$ preserves the relations $x\leq x^2$, hence its right adjoint preserve the radical elements since they are the local objects for the relations $x\leq x^2$.
This proves that $\sqrt{Q\upslice w} \subseteq \sqrt Q \upslice w$.
Conversely, if $y$ is radical in $Q$ and $w\leq y$, then for any $x\geq w$, the relation $x\star_wx = x^2\vee w \leq y$ implies 
$x^2\leq y$ and $x\leq y$ since $y$ is radical. 
This shows that $ \sqrt Q \upslice w\subseteq \sqrt{Q\upslice w}$.
The equality of posets $\sqrt{Q\upslice w} = (\sqrt Q)\upslice w$ implies that their minimal element are the same, hence the last relation of the statement.
\end{proof}

\medskip
 We finish with a study of the naturality of the frame reflection with respect to $\otimes$\=/frame morphisms.
Let $\phi:Q\to P$ be a $\otimes$\=/frame morphisms, with right adjoint $\phi_*$.
We denote by $\sqrt[Q] x$ the reflection of $x\in Q$ into $\sqrt Q$ and 
and by $\sqrt[P] y$ the reflection of $y\in P$ into $\sqrt P$.

\begin{proposition}
\label{prop:rad:square-adjoints}
A $\otimes$\=/frame morphism $\phi:Q\to P$ induces a frame morphism $\sqrt[P]\phi:\sqrt Q\to \sqrt P$, which is right adjoint to the restriction of $\phi_*$.
Moreover, there is a natural commutative square $\otimes$\=/frames morphism (in solid arrows):
\[
\begin{tikzcd}
\sqrt Q
\ar[from=rrr , "{\sqrt[Q]-}"', shift right=1.5] 
\ar[rrr, shift right=1.5, hook, "\sf can"', dashed]
\ar[dd,"{\sqrt[P]\phi}"', shift right=1.5]
\ar[from=dd, shift right=1.5,"\phi_*"', dashed]
&&& Q
\ar[dd,"\phi"', shift right=1.5 ]
\ar[from=dd, shift right=1.5,"\phi_*"', dashed]
\\
\\
\sqrt P
\ar[from=rrr, shift right=1.5, "{\sqrt[P]-}"'] 
\ar[rrr, shift right=1.5, hook, "\sf can"', dashed]
&&& P
\,.
\end{tikzcd}
\]
(The square of dashed arrows is the corresponding squares of right adjoints.)
\end{proposition}

\begin{proof}
We've seen in \cref{prop:frame-refl} that the $\otimes$\=/frame quotient $Q\to \sqrt Q$ is generated by forcing the relations $x^2\leq x$ to become equalities, for every $x$ in $Q$.
Since $\phi:Q\to P$ is a $\otimes$\=/frame morphism, it clearly sends the generating relations for $Q\to \sqrt Q$ into the generating relations for $P\to \sqrt P$.
Thus $\phi$ induces a morphism $\sqrt Q\to \sqrt P$.
Under the identification of $\sqrt Q$ with radical elements of $Q$, this morphism is given by 
$\sqrt[P]\phi:\sqrt Q\subseteq Q\xto\phi P\xto{\sqrt[P]-}\sqrt P$.
The commutation of the solide square follows.
Another consequence of $\phi$ preserving the generating relations of the quotients, is that $\phi_*$ send local objects to local objects, that is, restricts to a morphism $\sqrt P\to \sqrt Q$.
\end{proof}


\providecommand{\bysame}{\leavevmode\hbox to3em{\hrulefill}\thinspace}

\end{document}